\documentclass[preprint,12pt]{elsarticle}
\usepackage{amsmath,amssymb,amsfonts}
\usepackage{bm}
\usepackage{graphicx}
\usepackage{mathtools}
\usepackage{amsthm}
\usepackage{hyperref}
\usepackage{enumitem}
\usepackage{tikz}
\usetikzlibrary{calc}
\usepackage{stmaryrd}
\usepackage{float}
\usepackage{placeins}
\usepackage{subcaption}
\journal{Journal of Computational Physics}

\newtheorem{theorem}{Theorem}
\newtheorem{lemma}{Lemma}

\newtheorem{remark}{Remark}[section]
\newtheorem{definition}{Definition}
\usepackage[lined,boxed,linesnumbered,ruled]{algorithm2e}

\usepackage[a4paper,left=0.75in,right=0.75in,top=0.85in,bottom=1in]{geometry}

\begin{document}

\begin{frontmatter}

\title{Positivity-Preserving and Entropy-Stable Oscillation-Eliminating DGSEM for the Compressible Euler Equations on Curvilinear Meshes with Adaptive Mesh Refinement}

\author[1]{{Jielin Yang}\corref{cor}\fnref{fn1}}
\ead{jyang38@nd.edu}

\author[1]{Guosheng Fu\fnref{fn1}}
\ead{gfu@nd.edu}

\address[1]{Department of Applied and Computational Mathematics and Statistics (ACMS), University of Notre Dame, Notre Dame, IN 46556, USA}

\cortext[cor]{Corresponding author}

\fntext[fn1]{This work is supported in part by NSF DMS-2410741.}

\begin{abstract}
We extend the entropy-stable oscillation-eliminating discontinuous Galerkin spectral element method (ES-OEDG) on curvilinear meshes \cite{YangFu26} to adaptive mesh refinement (AMR) grids with nonconforming interfaces. The formulation is developed for two-dimensional curvilinear quadrilateral meshes under a 2:1 refinement constraint, allowing a single level of hanging nodes. The elementwise volume discretization and geometric mapping are retained, while the oscillation elimination procedure and interface coupling are adapted to accommodate nonconforming interfaces.

{A central contribution of this work is the design and analysis of numerical fluxes for nonconforming interfaces. We first propose a new flux-interpolation coupling that, unlike the classical mortar approach, evaluates pairwise numerical fluxes between the original coarse- and fine-side nodal states before applying interface interpolation. The resulting coupling is globally conservative by construction and satisfies a semi-discrete entropy inequality when its pairwise dissipation parameters satisfy the scaled Guermond--Popov Riemann-fan wave-speed condition. For polynomial degrees $N\geq2$, however, the sign correction required by the entropy and positivity arguments reduces its formal consistency. The measured convergence rates remain below second order and, especially for $N\geq2$, are close to first order on the severe curvilinear nonconforming mesh used in the convergence test. We also analyze the classical mortar-based flux. Unlike the proposed entropy-stable flux, the mortar coupling retains its high-order character for sufficiently smooth solutions. However, its entropy stability cannot be established, and our numerical experiments demonstrate that it can indeed produce entropy violations. Furthermore, under suitable wave-speed and CFL conditions, we rigorously prove that both interface fluxes are positivity preserving on curvilinear nonconforming meshes: a forward Euler step preserves positive cell-average density and pressure, so that the Zhang--Shu scaling limiter can subsequently eliminate negative values at the nodal quadrature and interface interpolation points.}

{We further develop a complete AMR framework consisting of a new efficient indicator derived from the zero-order oscillation-elimination damping coefficient and an admissible data-transfer procedure. The latter is designed to preserve both conservation and the positivity structure during refinement and coarsening. Specially constructed transfer corrections, together with a carefully ordered application of the Zhang--Shu limiter, prevent the creation of negative density or pressure during mesh adaptation. Numerical experiments systematically compare the entropy-stable and mortar-based fluxes. Although the entropy-stable flux exhibits reduced convergence rates on the deliberately severe checkerboard curvilinear nonconforming mesh, it provides robustness, solution quality, and resolution of fine flow structures comparable to those of the mortar-based flux in practical shock-dominated AMR simulations.}
\end{abstract}


\end{frontmatter}

\section{Introduction}

High-order entropy-stable discontinuous Galerkin (DG) methods on curvilinear meshes \cite{fisher2012high,carpenter2016entropy,crean2018entropy,chan2019discretely} provide a robust and accurate framework for compressible flow simulations involving complex geometries. In our previous work \cite{YangFu26}, we developed an entropy-stable oscillation-eliminating discontinuous Galerkin spectral element method (ES-OEDG) for the compressible Euler equations on conforming curvilinear quadrilateral meshes, combining entropy-stable discretization with an oscillation-eliminating (OE) procedure to enhance robustness in the presence of strong shocks.

{
The extension of entropy-stable high-order discretizations to nonconforming
meshes is challenging because the nodal points and polynomial trace spaces on
the two sides of an interface do not coincide. Conservative coupling therefore
requires interpolation or projection between different trace spaces, and these
operations generally do not commute with nonlinear flux evaluations. Moreover,
the corresponding interpolation matrices may contain negative entries for
\(N\geq 2\), obstructing both the convex decomposition required for positivity
preservation and the sign-definite dissipation required by a Tadmor-type entropy
analysis.

Early nonconforming spectral-element methods employed mortar spaces to couple
elements with mismatched interfaces while retaining conservation and high-order
accuracy \cite{kopriva2002electromagnetic}. Friedrich et
al.~\cite{FriedrichEtAl2018} subsequently constructed an entropy-stable
\(h/p\)-nonconforming DGSEM on rectangular meshes. They showed that the
standard mortar treatment does not, in general, preserve entropy stability for
nonlinear conservation laws and introduced modified interface operators based
on the summation-by-parts property. For curved geometrically nonconforming
meshes, Kopriva et al.~\cite{KoprivaEtAl2019} identified the metric
compatibility and discrete geometric conservation conditions required for
free-stream preservation. Del Rey Fern\'andez et
al.~\cite{DelReyFernandezEtAl2020} then extended entropy-stable
nonconforming SBP discretizations to \(h/p\)-refined curvilinear meshes for the
compressible Euler and Navier--Stokes equations using decoupled interpolation
operators. More recently, Chan et al.~\cite{ChanEtAl2021} developed an
entropy-stable mortar-based formulation for nonconforming quadrilateral and
hexahedral meshes, including efficient constructions for both Lobatto and
Gauss collocation schemes and conditions for high-order accuracy. Collectively,
these works show that entropy stability on nonconforming meshes requires
interface operators specifically designed to preserve the underlying discrete
entropy identities.

On conforming meshes, positivity-preserving high-order DG methods based on the
Zhang--Shu scaling limiter are well developed
\cite{zhang2010positivity}. The limiter has also been used in practical
nonconforming AMR calculations. For example, Schaal et
al.~\cite{SchaalEtAl2015} incorporated a Zhang--Shu-type limiter into the
Cartesian AMR-DG code TENET and enlarged the limiter point set at interfaces
with finer neighbors. {Open-source frameworks such as
\texttt{Trixi.jl} combine positivity-preserving limiting with tree-based AMR
in large-scale compressible-flow simulations and transfer the
Jacobian-weighted conservative variables $\mathcal J\bm U$ between refinement
levels \cite{RanochaEtAl2022}. {In that construction, conservation
across refinement levels is tied directly to the transfer of the
Jacobian-weighted state.}} These developments
demonstrate the practical effectiveness of positivity limiting on adaptive
meshes, but do not by themselves provide a unified positivity analysis for
entropy-stable interface coupling on curvilinear nonconforming meshes. In
particular, the admissibility of all interpolated interface states must be
controlled, and refinement and coarsening operators must transfer the solution
without destroying conservation or positive density and pressure. Thus, a
rigorous framework combining nonconforming entropy stability, positivity
preservation, curved geometry, and admissible AMR data transfer remains largely
undeveloped.

In this work, we extend the conforming-curvilinear ES-OEDG framework of
\cite{YangFu26} to two-dimensional quadrilateral AMR meshes under a 2:1
refinement constraint. The elementwise DG discretization is retained from the
conforming formulation, while the coupling across hanging interfaces is
modified. {We study two complementary numerical fluxes for the
nonconforming interface: the classical mortar-based flux and a new
entropy-stable flux proposed in this work.}
\begin{itemize}
  \item {\textbf{Mortar-based flux.}
  This is a classical and widely used coupling for nonconforming interfaces.
  It first interpolates the conservative states to the fine-side mortar points
  and then evaluates standard two-point numerical fluxes there. It retains
  high-order consistency for sufficiently smooth solutions, but its entropy
  stability cannot in general be established.}

  \item {\textbf{Entropy-stable flux.}
  This is a new flux proposed in this work. Unlike the mortar flux, it first
  evaluates pairwise numerical fluxes between the original coarse- and fine-side nodal states and then
  interpolates or projects the resulting fluxes to the corresponding interface
  nodes. Whenever an interpolation coefficient is negative, the two states in
  the dissipative flux are exchanged. This construction is entropy stable and
  positivity preserving, but its observed convergence rates on the severe
  nonconforming curvilinear test lie between first and second order and, for
  $N\geq2$, are close to first order.}
\end{itemize}

The principal contributions of this work are summarized as follows:
\begin{itemize}
  \item We systematically analyze the classical mortar-based flux on
  curvilinear nonconforming interfaces and rigorously prove its
  positivity-preserving property under an appropriate CFL condition. We also
  investigate its entropy behavior through numerical examples and observe
  entropy violations in practical test problems.

  \item We construct the entropy-stable flux and rigorously prove that it
  preserves the conservative variables, satisfies a semi-discrete entropy
  inequality, and preserves positivity of the cell-average density and
  pressure under an appropriate CFL condition. Combined with an extended
  Zhang--Shu limiter, it maintains admissibility at all required nodal and
  interpolation points. To the best of the authors' knowledge, this is the only
  directly constructed interface flux for curvilinear nonconforming DGSEM for
  which both entropy stability and positivity preservation have been
  rigorously established, although {its observed spatial convergence
  remains below second order and is close to first order for $N\geq2$ in the
  present nonconforming curvilinear test.}

  \item {We define an admissible AMR data transfer to require both
  patchwise conservation and preservation of the positive-density and
  positive-pressure structure. We use parent-to-child nodal restriction for
  refinement and a deterministic child-to-parent nodal restriction for
  coarsening,
  and then introduce a rank-one correction that restores the old patch moment,
  followed by a single patchwise Zhang--Shu scaling over all new nodal and
  mortar checking points. We rigorously prove conservation and positivity for
  both refinement and coarsening. In contrast to the Jacobian-weighted transfer
  used in \texttt{Trixi.jl} \cite{RanochaEtAl2022}, the proposed construction
  restores the physical moments without transferring $\mathcal J\bm U$ or
  requiring a pointwise parent--child Jacobian interpolation identity.}
\end{itemize}

The remainder of the paper is organized as follows.
Section~\ref{sec:conforming_esdgsem} briefly reviews the entropy-stable DG
formulation for the Euler equations on conforming curvilinear quadrilateral
meshes developed in \cite{YangFu26}. Section~\ref{sec:amr} presents the
extension to nonconforming meshes, including {the entropy-stable
interface flux and the analysis of conservation, entropy
stability, and positivity preservation.} Section~\ref{sec:pp_nc} reviews the
classical mortar-based flux and discusses its high-order consistency,
positivity-preserving property, and entropy behavior.
Section~\ref{sec:oe_amr} describes the OE procedure on nonconforming meshes,
adaptive refinement and coarsening, and {the rank-one corrected
admissible solution transfer between successive AMR meshes.} Numerical experiments
demonstrating the accuracy, entropy behavior, robustness, and
positivity-preserving properties of the different interface treatments are
presented in Section~\ref{sec:num}. Finally, concluding remarks are given in
Section~\ref{sec:con}.
}

\section{Entropy-stable DGSEM on conforming curvilinear quadrilateral meshes}
\label{sec:conforming_esdgsem}

In this section, we briefly summarize the entropy-stable DGSEM on conforming curvilinear quadrilateral meshes developed in \cite{YangFu26}, which serves as the baseline scheme for the AMR extension proposed in this work. We retain the notation of \cite{YangFu26} whenever possible. Since the novelty of the present paper lies in the treatment of nonconforming interfaces, we restrict the discussion here to the ingredients needed later for the AMR formulation and its analysis.

\subsection{Euler equations and entropy variables}
\label{subsec:euler_entropy}

We consider the two-dimensional compressible Euler equations in conservative form,
\begin{equation}
\partial_t \bm U + \partial_x \bm f(\bm U) + \partial_y \bm g(\bm U) = 0,
\label{eq:euler}
\end{equation}
where the conservative variable is
\begin{equation}
\bm U = (\rho,\rho u,\rho v,E)^T.
\end{equation}
Here, $\rho$ denotes the density, $\bm u=(u,v)^T$ the velocity, and
\begin{equation}
E = \rho e + \frac12 \rho |\bm u|^2
\end{equation}
the total energy, with $e$ the specific internal energy. The Cartesian fluxes are given by
\begin{equation*}
\bm f(\bm U)=
(\rho u,\rho u^2+p,\rho uv,(E+p)u)^T,
\qquad
\bm g(\bm U)=
(\rho v,\rho uv,\rho v^2+p,(E+p)v)^T.
\end{equation*}
Throughout this work, we assume an ideal-gas equation of state,
\begin{equation}
p = (\gamma-1)\rho e, \qquad \gamma = 1.4.
\label{eq:eos}
\end{equation}

For entropy analysis, we adopt the mathematical entropy
\begin{equation}
\eta(\bm U) = -\frac{\rho s}{\gamma-1},
\qquad
s = \log(p\rho^{-\gamma}),
\label{eq:entropy}
\end{equation}
with associated entropy flux
\begin{equation}
\bm q(\bm U) = -\frac{\rho s\,\bm u}{\gamma-1}.
\end{equation}
The corresponding entropy variables are
{
\begin{equation}
\bm V = \frac{\partial \eta}{\partial \bm U}
=
\begin{pmatrix}
\dfrac{\gamma-s}{\gamma-1} - \beta |\bm u|^2\\[0.8ex]
2\beta u\\[0.4ex]
2\beta v\\[0.4ex]
-2\beta
\end{pmatrix},
\qquad
\beta=\frac{\rho}{2p},
\label{eq:entropy_variables}
\end{equation}
}
{and the entropy-potential vector is
\begin{equation}
\boldsymbol\psi = (\psi^f,\psi^g) = \rho \bm u.
\label{eq:entropy_potential}
\end{equation}}

\subsection{Curvilinear formulation}
\label{subsec:curvilinear_formulation}

Let $\Omega_h=\{\Omega_e\}_{e=1}^{N_e}$ be a conforming partition of the physical domain into curved quadrilateral elements. Each element $\Omega_e$ is obtained from the reference element
\[
\Omega_{\mathrm{ref}} = [-1,1]^2
\]
through a smooth mapping
\[
\bm x^e(\xi,\eta) = \bigl(x^e(\xi,\eta),y^e(\xi,\eta)\bigr).
\]
Pulling back \eqref{eq:euler} to $\Omega_{\mathrm{ref}}$ yields
\begin{equation}
\mathcal J^e \partial_t \bm U^e
+
\partial_\xi \widetilde{\bm f}^{\,e}
+
\partial_\eta \widetilde{\bm g}^{\,e}
= 0,
\label{eq:curv_euler}
\end{equation}
where $\mathcal J^e$ is the Jacobian of the mapping,
\begin{equation}
\mathcal J^e = x_\xi^e y_\eta^e - x_\eta^e y_\xi^e,
\end{equation}
and the contravariant fluxes are defined by
\begin{equation}
\widetilde{\bm f}^{\,e}
=
y_\eta^e \bm f(\bm U^e) - x_\eta^e \bm g(\bm U^e),
\qquad
\widetilde{\bm g}^{\,e}
=
- y_\xi^e \bm f(\bm U^e) + x_\xi^e \bm g(\bm U^e).
\label{eq:contravariant_fluxes}
\end{equation}

A compatible discretization of the metric terms is required in order to preserve the discrete geometric identities and, consequently, free-stream preservation and entropy stability \cite{Kopriva2006}. As in \cite{YangFu26}, we adopt an isoparametric discretization, in which the geometric mapping and the solution are approximated using the same polynomial space and nodal operators. This choice ensures that the discrete geometric conservation laws are satisfied.

\subsection{Entropy-stable DGSEM on the reference element}
\label{subsec:es_dgsem}

Let $\{\xi_i\}_{i=0}^N$ and $\{\eta_j\}_{j=0}^N$ denote the Legendre--Gauss--Lobatto (LGL) nodes on $[-1,1]$, with corresponding quadrature weights $\{w_i\}_{i=0}^N$. We approximate the solution on $\Omega_{\mathrm{ref}}$ in the tensor-product polynomial space
\begin{equation}
\mathbb Q_N(\Omega_{\mathrm{ref}})
=
\mathrm{span}\{\phi_i(\xi)\phi_j(\eta):0\le i,j\le N\},
\end{equation}
where $\phi_i$ are the one-dimensional Lagrange basis functions associated with the LGL nodes.

The one-dimensional differentiation matrix is defined by
\begin{equation}
D_{ij} = \phi_j'(\xi_i),
\qquad i,j=0,\dots,N,
\end{equation}
and satisfies the summation-by-parts (SBP) property
\begin{equation}
w_i D_{ij} + w_j D_{ji} = \delta_{iN}\delta_{jN} - \delta_{i0}\delta_{j0},
\; \text{and }
\sum_{j=0}^N D_{ij}=0.
\end{equation}

To obtain an entropy-stable discretization, we employ a split-form DGSEM based on symmetric two-point entropy-conservative fluxes. Let
\[
\bm F^\#(\bm U_L,\bm U_R)
=
\bigl(\bm f^\#(\bm U_L,\bm U_R),\bm g^\#(\bm U_L,\bm U_R)\bigr)
\]
be a consistent and symmetric numerical flux satisfying Tadmor's entropy conservation condition
\begin{equation}
(\bm V_R-\bm V_L)\cdot \bm F^\#(\bm U_L,\bm U_R)
=
{\boldsymbol\psi_R - \boldsymbol\psi_L},
\end{equation}
where $\bm V=\partial \eta/\partial \bm U$ are the entropy variables associated with \eqref{eq:entropy}, and {$\boldsymbol\psi=(\psi^f,\psi^g)$ is the corresponding entropy-potential vector} defined in \eqref{eq:entropy_potential}. Subscripts $L$ and $R$ denote evaluation at left and right states. In this work, we adopt Chandrashekar's entropy-conservative flux \cite{chandrashekar2013kinetic}; see also \cite[Eq.~(40)]{YangFu26}.

On each element $\Omega_e$, the solution is represented in nodal form as
\[
\bm U^e(\xi,\eta) \approx \bm U_N^e(\xi,\eta)
=
\sum_{i=0}^N \sum_{j=0}^N \bm U^e_{ij}\,\phi_i(\xi)\phi_j(\eta),
\]
and all fluxes are evaluated at the nodal states $\bm U^e_{ij}$.

The contravariant entropy-conservative volume fluxes are constructed using symmetric averages of the geometric coefficients. For example,
\begin{equation}
\begin{aligned}
(\widetilde{\bm f}^{\,e,\#})_{(i,m),j}
&=
\{\!\{y_\eta^e\}\!\}_{(i,m),j}\,
\bm f^\#\!\left(\bm U^e_{ij},\bm U^e_{mj}\right)
\\
&\quad
-
\{\!\{x_\eta^e\}\!\}_{(i,m),j}\,
\bm g^\#\!\left(\bm U^e_{ij},\bm U^e_{mj}\right),\\
{(\widetilde{\bm g}^{\,e,\#})_{i,(j,n)}}
&=
-
{\{\!\{y_\xi^e\}\!\}_{i,(j,n)}\,
\bm f^\#\!\left(\bm U^e_{ij},\bm U^e_{in}\right)}
\\
&\quad
+
{\{\!\{x_\xi^e\}\!\}_{i,(j,n)}\,
\bm g^\#\!\left(\bm U^e_{ij},\bm U^e_{in}\right).}
\end{aligned}
\end{equation}
Here $\{\!\{\cdot\}\!\}$ denotes the arithmetic average, e.g.,
\[
\{\!\{x_\eta^e\}\!\}_{(i,m),j}
=
\frac{1}{2}\Bigl((x_\eta^e)_{ij} + (x_\eta^e)_{mj}\Bigr),
\]
and similarly for the other metric terms.

Using tensor-product LGL quadrature, the entropy-stable DGSEM on each element $\Omega_e$ takes the nodal vector form
\begin{equation}
\begin{aligned}
&w_i w_j \Bigg(
\mathcal J^e_{ij} \partial_t \bm U^e_{ij}
+
2 \sum_{m=0}^N D_{i m} (\widetilde{\bm f}^{\,e,\#})_{(i,m),j}
+
{2 \sum_{n=0}^N D_{j n}
(\widetilde{\bm g}^{\,e,\#})_{i,(j,n)}}
\Bigg)
\\
&\quad
+
w_i \Bigl(
(\Delta \widetilde{\bm g}^{\,e})_{i0}\,\delta_{j0}
-
(\Delta \widetilde{\bm g}^{\,e})_{iN}\,\delta_{jN}
\Bigr)\\
&\quad+
w_j \Bigl(
(\Delta \widetilde{\bm f}^{\,e})_{0j}\,\delta_{i0}
-
(\Delta \widetilde{\bm f}^{\,e})_{Nj}\,\delta_{iN}
\Bigr)
= \bm 0,
\end{aligned}
\label{eq:es_dgsem}
\end{equation}
for all $0\le i,j\le N$, where
\[
\Delta \widetilde{\bm f}^{\,e}
=
\widetilde{\bm f}^{\,e}-\widetilde{\bm f}^{\,e,*},
\qquad
\Delta \widetilde{\bm g}^{\,e}
=
\widetilde{\bm g}^{\,e}-\widetilde{\bm g}^{\,e,*}.
\]

The interface numerical fluxes $\widetilde{\bm f}^{\,e,*}$ and $\widetilde{\bm g}^{\,e,*}$ define the coupling between neighboring elements through the boundary terms of \eqref{eq:es_dgsem}. At conforming interfaces, the numerical flux is defined in the physical normal direction. Let $E=\partial\Omega_e\cap\partial\Omega_{\mathrm{nbr}}$ be an interface shared by element $\Omega_e$ and its neighboring element $\Omega_{\mathrm{nbr}}$, and let
\[
\bm n = \mathcal J^e (\mathcal G^e)^{-T}\hat{\bm n}
\]
denote the outward metric-scaled normal vector on $\partial\Omega_e$, where $\mathcal G^e$ is the Jacobian matrix of the mapping from the reference element to $\Omega_e$,
\[
\mathcal G^e =
\begin{pmatrix}
x_\xi^e & x_\eta^e \\
y_\xi^e & y_\eta^e
\end{pmatrix},
\]
and $\hat{\bm n}$ is the outward unit normal on the reference element.

Let $\bm U^e$ and $\bm U^{e,\mathrm{nbr}}$ denote the traces of the solution on the two sides of $E$. The physical numerical flux is defined by the local Lax--Friedrichs flux
\begin{equation}
\label{lax}
\bm F^*(\bm U^e,\bm U^{e,\mathrm{nbr}})\cdot \bm n
=
\frac12\Bigl(
\bm F(\bm U^e)\cdot\bm n + \bm F(\bm U^{e,\mathrm{nbr}})\cdot\bm n
\Bigr)
-
\frac{\alpha\|\bm n\|}{2}\bigl(\bm U^{e,\mathrm{nbr}}-\bm U^e\bigr),
\end{equation}
where $\alpha$ is a suitable maximum wave-speed estimate.  {We
distinguish the following two wave-speed estimates. The conventional
endpoint maximum-characteristic-speed estimate is
\begin{equation}
\label{eq:alpha_char}
\alpha_{\rm char}
=
\max\bigl\{
|\boldsymbol u_L\cdot\bm n_u|+c_L,\,
|\boldsymbol u_R\cdot\bm n_u|+c_R
\bigr\},
\qquad \bm n_u=\bm n/\|\bm n\|.
\end{equation}
For the entropy analysis, let
\[
\alpha_{\rm GP}
:=\Lambda^{\mathrm{GP}}(\bm U_L,\bm U_R;\bm n_u)
\]
denote the Guermond--Popov guaranteed Riemann-fan wave-speed bound, namely the
computable upper bound of Guermond and
Popov~\cite{GuermondPopov16} for the maximum absolute wave speed of the
one-dimensional Euler Riemann problem in the direction $\bm n_u$.
To satisfy simultaneously the wave-speed
conditions used in the positivity and entropy analyses, we choose
\begin{equation}
\label{alpha}
\alpha
=
\max\!\left\{\alpha_{\rm char},\alpha_{\rm GP}\right\}.
\end{equation}
There is no general ordering between $\alpha_{\rm char}$ and
$\alpha_{\rm GP}$. Indeed, Guermond and Popov~\cite{GuermondPopov16}
give examples in which the conventional endpoint estimate is substantially
larger than the exact maximum Riemann-fan speed and others in which it is
substantially smaller.
The maximum-characteristic-speed estimate $\alpha_{\rm char}$ provides the standard LF admissibility condition used
in the positivity argument, while the Guermond--Popov bound contains the
exact Riemann fan and provides the condition used in the entropy-stability
argument.}

The contravariant numerical fluxes in \eqref{eq:es_dgsem} are defined so that their contraction with the reference normal reproduces the (metric-scaled) physical normal flux, i.e.,
\begin{equation}
\label{conforming_flux}
(\widetilde{\bm f}^{\,e,*},\widetilde{\bm g}^{\,e,*})\cdot \hat{\bm n}
=
\bm F^*(\bm U^e,\bm U^{e,\mathrm{nbr}})\cdot \bm n.
\end{equation}

The following lemma summarizes the fundamental elementwise properties of the entropy-stable DGSEM. For completeness, we state the result explicitly; see \cite{YangFu26} for the detailed derivation.

\begin{lemma}[Single-element conservation and entropy balance]
\label{lma:1}
Let $\{\bm U_N^e\}_e$ denote the numerical solution of the entropy-stable DGSEM~\eqref{eq:es_dgsem}.  
On a single element $\Omega_e$, the scheme satisfies the following properties.

\medskip
\noindent
\textbf{(i) Primary conservation.}
\begin{equation}
\label{eq_cc}
\begin{aligned}
\sum_{i,j=0}^N
w_i w_j\,
\mathcal{J}_{ij}^e\,\partial_t \bm U^e_{ij}
= {}
\sum_{i=0}^N w_i
\Bigl(
\widetilde{\bm g}^{\,e,*}_{i0}
-
\widetilde{\bm g}^{\,e,*}_{iN}
\Bigr)
+
\sum_{j=0}^N w_j
\Bigl(
\widetilde{\bm f}^{\,e,*}_{0j}
-
\widetilde{\bm f}^{\,e,*}_{Nj}
\Bigr).
\end{aligned}
\end{equation}

\medskip
\noindent
\textbf{(ii) Entropy balance.}
\begin{equation}
\label{eq_ss}
\begin{aligned}
\sum_{i,j=0}^N
w_i w_j\,
\mathcal{J}_{ij}^e\,
\partial_t \eta(\bm U^e_{ij})
= {} &
\sum_{i=0}^N w_i
\Bigl(
\widetilde{\bm g}^{\,e,*}_{i0} \cdot \bm V^e_{i0}
-
{\widetilde{\psi}^{\,e,g}_{i0}}
\Bigr)
\\
&\; -
\sum_{i=0}^N w_i
\Bigl(
\widetilde{\bm g}^{\,e,*}_{iN} \cdot \bm V^e_{iN}
-
{\widetilde{\psi}^{\,e,g}_{iN}}
\Bigr)
\\
&\; +
\sum_{j=0}^N w_j
\Bigl(
\widetilde{\bm f}^{\,e,*}_{0j} \cdot \bm V^e_{0j}
-
{\widetilde{\psi}^{\,e,f}_{0j}}
\Bigr)
\\
&\; -
\sum_{j=0}^N w_j
\Bigl(
\widetilde{\bm f}^{\,e,*}_{Nj} \cdot \bm V^e_{Nj}
-
{\widetilde{\psi}^{\,e,f}_{Nj}}
\Bigr).
\end{aligned}
\end{equation}

Here $\bm V^e_{ij} = (\partial \eta / \partial \bm U)(\bm U^e_{ij})$ denotes the entropy variables evaluated at the nodal states, and the contravariant entropy potential fluxes are defined by
\[
{\widetilde{\psi}^{\,e,f}_{ij}}
=
(y_\eta^e)_{ij}\,{\psi^f_{ij}}
-
(x_\eta^e)_{ij}\,{\psi^g_{ij}},
\qquad
{\widetilde{\psi}^{\,e,g}_{ij}}
=
-(y_\xi^e)_{ij}\,{\psi^f_{ij}}
+
(x_\xi^e)_{ij}\,{\psi^g_{ij}}.
\]
\end{lemma}

\begin{remark}
\label{rk1}
The above analysis is purely elementwise and is independent of the specific choice of interface numerical fluxes. In particular, it relies only on the summation-by-parts property and the entropy-conservative structure of the volume fluxes $\bm F^\#$.

As a consequence, these properties extend naturally to nonconforming AMR meshes. The key remaining challenge is the construction of interface numerical fluxes that ensure global conservation and entropy stability; see, e.g., \cite{FriedrichEtAl2018,DelReyFernandezEtAl2020, ChanEtAl2021}. This issue is addressed in the next section.
\end{remark}

\section{\texorpdfstring{{Positivity-preserving and entropy-stable flux on nonconforming meshes}}{Positivity-preserving and entropy-stable flux on nonconforming meshes}}
\label{sec:amr}

In this section, we extend the entropy-stable DGSEM to nonconforming curvilinear meshes. Our goal is to construct numerical fluxes for nonconforming interfaces that preserve both entropy stability and positivity. {This comes with an important limitation: the nonlocal coarse-side projection introduces first-order pairwise jumps even for \(N=1\), and for \(N\ge 2\) the sign changes in the interpolation operators cause an additional loss of interpolation cancellation. The resulting ES coupling therefore does not retain formal high-order accuracy; the convergence rates measured in Section~\ref{subsec:num_vortex} remain between first and second order and, for \(N\ge2\), are close to first order.} {On the other hand, the classical mortar flux recovers high-order consistency for sufficiently smooth solutions and is positivity preserving, but its entropy stability is not established by the present analysis. We review the mortar-based flux in the next section.}

As discussed in Remark~\ref{rk1}, the volume discretization and its associated elementwise conservation and entropy properties remain unchanged. Consequently, the extension reduces to the construction of appropriate numerical fluxes at nonconforming interfaces.

We restrict attention to two-dimensional quadrilateral meshes under a \(2{:}1\) refinement constraint, allowing a single level of hanging nodes. It is sufficient to consider a representative nonconforming interface configuration consisting of two fine elements adjacent to one coarse element, as illustrated in Figure~\ref{fig:nc_interface}. The goal is to construct numerical fluxes on this interface that preserve the conservation, entropy stability, and positivity properties of the underlying DGSEM.

\begin{figure}[t]
\centering
\begin{tikzpicture}[scale=1.0,line cap=round,line join=round]

\coordinate (A) at (0.0,0.0);
\coordinate (B) at (0.2,2.9);

\coordinate (D) at (4.2,-0.02);
\coordinate (C) at (4.55,2.95);

\coordinate (F) at (7.7,-0.28);
\coordinate (E) at (8.9,2.65);

\draw[thick] (A) .. controls (0.8,0.1) and (2.0,0.0) .. (D);
\draw[thick] (D) .. controls (5.5,-0.1) and (6.7,-0.2) .. (F);

\draw[thick] (B) .. controls (1.9,3.15) and (3.5,3.15) .. (C);
\draw[thick] (C) .. controls (6.2,2.8) and (7.4,2.7) .. (E);

\draw[thick] (A) .. controls (0.8,0.9) and (0.9,2.0) .. (B);
\draw[thick] (E) .. controls (9.0,1.7) and (8.8,0.5) .. (F);

\draw[thick] (D) .. controls (4.9,0.7) and (4.9,2.15) .. (C);

\coordinate (M) at (0.68,1.5);
\coordinate (H) at (4.75,1.42);
\draw[thick] (M) .. controls (3.4,1.55) .. (H);

\node at (2.25,2.0) {$\Omega_{F_2}$};
\node at (2.05,0.82) {$\Omega_{F_1}$};
\node at (6.9,1.35) {$\Omega_C$};

\fill[blue] (4.45,2.78) circle (3.2pt);
\fill[blue] (4.58,2.2) circle (3.2pt);
\fill[blue] (4.62,1.60) circle (3.2pt);
\fill[blue] (4.62,1.28) circle (3.2pt);
\fill[blue] (4.47,0.68) circle (3.2pt);
\fill[blue] (4.15,0.14) circle (3.2pt);

\fill[red] (4.80,2.78) circle (3.2pt);
\fill[red] (4.93,1.40) circle (3.2pt);
\fill[red] (4.50,0.14) circle (3.2pt);

\end{tikzpicture}
\caption{Representative nonconforming interface under $2{:}1$ refinement with polynomial degree $N=2$. The two fine elements on the left have three nodal
collocation points on each fine interface (blue), while the coarse element on the right has three nodal collocation points on the coarse interface (red).}
\label{fig:nc_interface}
\end{figure}

\subsection{Nonconforming interface flux construction}
\label{subsec:nc_flux}

We consider a representative nonconforming interface configuration under a \(2{:}1\) refinement, where two fine elements \(\Omega_{F_1}\) and \(\Omega_{F_2}\) are adjacent to a coarse element \(\Omega_C\); see Figure~\ref{fig:nc_interface}. Let
\[
E_{F_1} = \partial \Omega_{F_1} \cap \partial \Omega_C, 
\qquad
E_{F_2} = \partial \Omega_{F_2} \cap \partial \Omega_C,
\]
denote the two fine edges, and define the corresponding coarse edge
\[
E_C = E_{F_1} \cup E_{F_2}.
\]

We denote by \(\{\xi_i^{F_k}\}_{i=0}^N\) the collocation points on the fine edge \(E_{F_k}\) for \(k=1,2\), and by \(\{\xi_j^C\}_{j=0}^N\) the collocation points on the coarse edge \(E_C\). As shown in Figure~\ref{fig:nc_interface}, these collocation points generally do not coincide in physical space, which prevents the direct application of a standard two-point flux of the form \eqref{lax}.

We now construct the interface fluxes on the nonconforming edges \(E_{F_1}\), \(E_{F_2}\), and \(E_C\). For notational convenience, we denote by \((\bm U^{F_k}_i)_{i=0}^N\) the nodal values of the solution on the fine edge \(E_{F_k}\), and by \((\bm U^C_j)_{j=0}^N\) the nodal values on the coarse edge \(E_C\). Let \(\bm n^{F_k}_i\) denote the outward metric-scaled normal at the fine node \(\xi_i^{F_k}\), and \(\bm n^C_j\) the outward metric-scaled normal at the coarse node \(\xi_j^C\).

To construct the nonconforming interface fluxes, we introduce interpolation and projection operators defined on a reference nonconforming interface. Let
\[
E_C^{\mathrm{ref}}=[-1,1]
\]
denote the reference coarse edge, and split it into the two reference fine edges
\[
E_{F_1}^{\mathrm{ref}}=[-1,0],
\qquad
E_{F_2}^{\mathrm{ref}}=[0,1].
\]
Let \(\{\xi_j\}_{j=0}^N\) and \(\{\phi_j\}_{j=0}^N\) denote the LGL nodes and associated Lagrange basis functions introduced in Section~\ref{subsec:es_dgsem}. The nodal sets on the fine reference edges are obtained by affine mappings:
\[
\xi_i^{\,1}=\frac{\xi_i-1}{2}\in[-1,0],
\qquad
\xi_i^{\,2}=\frac{\xi_i+1}{2}\in[0,1],
\qquad i=0,\dots,N.
\]

For each \(k=1,2\), we define the interpolation matrix from the coarse edge to the fine edge by
\begin{equation}
\label{eq:P_C_to_Fk}
\bigl(P_{C\to F_k}\bigr)_{ij}
=
\phi_j(\xi_i^{\,k}),
\qquad i,j=0,\dots,N.
\end{equation}

The projection operators from the fine edges to the coarse edge are defined through the discrete compatibility condition
\begin{equation}
\label{eq:proj_compat}
P_{C\to F_k}^{T} M_{F_k}
=
M_C P_{F_k\to C},
\qquad k=1,2,
\end{equation}
or equivalently,
\begin{equation}
\label{eq:P_Fk_to_C}
P_{F_k\to C}
=
M_C^{-1} P_{C\to F_k}^{T} M_{F_k},
\qquad k=1,2,
\end{equation}
where {\(M_C=\mathrm{diag}(w_i)\)} and \(M_{F_k} = \frac12 M_C\) are the diagonal mass matrices associated with the LGL quadrature rules on \(E_C^{\mathrm{ref}}\) and \(E_{F_k}^{\mathrm{ref}}\), respectively.

These operators depend only on the reference nonconforming configuration and are independent of the geometry. Their explicit forms for \(N=1,2,3\) are listed in \ref{app:interp_mats}.

Using these operators, we define the numerical fluxes on the fine elements by
\begin{equation}
\label{eq:fine_flux}
\bm {F}^{\,F_k,*}_{i}\cdot\bm n^{F_k}_i
:=
\sum_{j=0}^{N}
\big(P_{C \to F_k}\big)_{ij}
\,
\bm F^\star\big( \bm U^{F_k}_{i}, \bm U^C_{j} \big)\cdot
\bm n_{ij}^{F_k,C},
\qquad k=1,2,
\end{equation}
for all \(0 \le i \le N\), where
\[
\bm n_{ij}^{F_k,C}
:=
\frac{1}{2}\bigl(\bm n^{F_k}_{i}-\tfrac12\,\bm n^{C}_{j}\bigr)
\]
denotes the averaged normal direction between the fine and coarse nodes, oriented from the fine elements \(\Omega_{F_k}\) toward the coarse element \(\Omega_C\). The factor \(\tfrac12\) reflects the metric scaling across the \(2{:}1\) nonconforming interface: the coarse-edge metric-scaled normal has twice the magnitude of the corresponding fine-edge normal.

{Throughout the entropy analysis of a fine--coarse pair, we use the same left-to-right convention as in the conforming formulation: the fine state is the left state, the coarse state is the right state, and \(\bm n_{ij}^{F_k,C}\) points from left to right.  Thus,
\[
\bm U_L=\bm U_i^{F_k},\qquad \bm U_R=\bm U_j^C,
\]
and every entropy jump is defined as right minus left, e.g.,
\[
\llbracket\bm V\rrbracket_{ij}^{F_k,C}
:=\bm V_j^C-\bm V_i^{F_k},
\qquad
{\llbracket\boldsymbol\psi\rrbracket_{ij}^{F_k,C}}
:={\boldsymbol\psi_j^C-\boldsymbol\psi_i^{F_k}}.
\]}

On the coarse element, the numerical flux is defined by combining contributions from both fine elements:
\begin{equation}
\label{eq:coarse_flux}
\bm {F}^{\,C,*}_{j}\cdot\bm n^{C}_j
:=
-2\sum_{k=1}^2 \sum_{i=0}^{N}
\big(P_{F_k \to C}\big)_{ji}
\,
\bm F^\star\big(\bm U^{F_k}_{i}, \bm U^C_{j}\big)\cdot
\bm n_{ij}^{F_k,C},
\end{equation}
for all \(0 \le j \le N\).

We define \(\bm F^\star\) using the following modified local Lax--Friedrichs flux:
\begin{equation}
\label{lxf}    
\begin{aligned}
\bm F^\star\big(\bm U^{F_k}_{i}, \bm U^C_{j} \big)\cdot \bm n_{ij}^{F_k,C}
:=\;&
\frac12\Bigl(
\bm F(\bm U^{F_k}_{i}) + \bm F(\bm U^C_{j})
\Bigr)\cdot \bm n_{ij}^{F_k,C}
\\
&-
\mathrm{sign}\!\left((P_{C \to F_k})_{ij}\right)
\frac{\alpha_i^{F_k}\|\bm n_i^{F_k}\|}{2}
\bigl(\bm U^C_{j}-\bm U^{F_k}_{i}\bigr){.}
\end{aligned}
\end{equation}
{
The wave-speed parameter must satisfy both the positivity and entropy
conditions in the direction of every pairwise normal.  Let
\(\widehat{\bm n}_{ij}^{F_k,C}:=
\bm n_{ij}^{F_k,C}/\|\bm n_{ij}^{F_k,C}\|\), and let
\(\Lambda_{ij}^{\mathrm{GP},F_k,C}\) denote the Guermond--Popov upper
bound for the Riemann data ordered according to the sign of
\((P_{C\to F_k})_{ij}\): the ordering is
\((\bm U_i^{F_k},\bm U_j^C)\) for a positive entry and
\((\bm U_j^C,\bm U_i^{F_k})\) for a negative entry.  We require
\begin{equation}
\label{eq:alpha_pp_nc}
\frac{\alpha_i^{F_k}\|\bm n_i^{F_k}\|}
     {\|\bm n_{ij}^{F_k,C}\|}
\geq
\max\!\left\{
\left|\bm u_i^{F_k}\!\cdot\widehat{\bm n}_{ij}^{F_k,C}\right|+c_i^{F_k},
\left|\bm u_j^{C}\!\cdot\widehat{\bm n}_{ij}^{F_k,C}\right|+c_j^{C},
\Lambda_{ij}^{\mathrm{GP},F_k,C}
\right\}
\end{equation}
for every pair \((i,j)\) with
\((P_{C\to F_k})_{ij}\ne0\) and
\(\bm n_{ij}^{F_k,C}\ne\bm0\).
Thus, \(\alpha_i^{F_k}\) is selected large enough to satisfy
\eqref{eq:alpha_pp_nc} for all \(j=0,\ldots,N\).  This pairwise directional
condition supplies the endpoint acoustic-speed bound used with
Lemma~\ref{lem:LF_admissible} in the positivity proof and the
Guermond--Popov Riemann-fan bound used in the entropy proof.  The sign
factor is the key ingredient that restores the nonnegative coefficients
required by both arguments.
}

\begin{remark}[Sign changes, stability, and loss of formal high-order consistency]
\label{rem:sign_consistency}
For polynomial degree \(N\ge 2\), the interpolation matrices \(P_{C\to F_k}\) generally contain negative entries. This creates a fundamental tension between high-order consistency and the requirements of entropy stability and positivity preservation.

{To describe the consistency directly at the flux level, fix a fine-side
node \(i\) and extend the coarse state and scaled normal smoothly along the
coarse edge. Define
\[
\bm n_i^k(\xi)
:=
\frac12\left(\bm n_i^{F_k}-\frac12\bm n^C(\xi)\right)
\]
and split the standard, unmodified pairwise LLF flux into its central and
dissipative parts,
\begin{align*}
\bm{\mathcal G}_i^k(\xi)
&:=\bm{\mathcal H}_i^k(\xi)-\bm{\mathcal D}_i^k(\xi),\\
\bm{\mathcal H}_i^k(\xi)
&:=\frac12\left(\bm F(\bm U_i^{F_k})+\bm F(\bm U^C(\xi))\right)
\cdot\bm n_i^k(\xi),\\
\bm{\mathcal D}_i^k(\xi)
&:=\frac{\alpha_i^{F_k}\|\bm n_i^{F_k}\|}{2}
\left(\bm U^C(\xi)-\bm U_i^{F_k}\right).
\end{align*}
Without the sign correction, the interface construction first evaluates this
pairwise numerical flux and then applies the signed interpolation operator.
For sufficiently smooth solution traces and geometry,
\[
\sum_{j=0}^N(P_{C\to F_k})_{ij}
\bm{\mathcal G}_i^k(\xi_j^C)
=
\bm{\mathcal G}_i^k(\xi_i^k)+\mathcal O(h^{N+1}).
\]
Thus, the signs of the interpolation coefficients do not by themselves
destroy high-order consistency; their nonnegativity is instead required by
the entropy and positivity arguments below.}

However, due to the presence of negative entries, the unmodified flux cannot
be used to establish entropy stability or positivity preservation. Indeed, the
stability arguments require the dissipation to appear with nonnegative
coefficients. With the sign correction in \eqref{lxf}, the assembled flux at
the fine node takes the form
{
\[
\sum_{j=0}^N(P_{C\to F_k})_{ij}\bm{\mathcal H}_i^k(\xi_j^C)
-
\sum_{j=0}^N\bigl|(P_{C\to F_k})_{ij}\bigr|
\bm{\mathcal D}_i^k(\xi_j^C).
\]}
{More explicitly, each pairwise contribution in the assembled interface balance is weighted by $(P_{C\to F_k})_{ij}$.  Without the sign correction, the corresponding LF dissipation would therefore carry the coefficient $(P_{C\to F_k})_{ij}\alpha_i^{F_k}$, which is negative whenever $(P_{C\to F_k})_{ij}<0$ and cannot define either a nonpositive entropy-production term or a convex positivity decomposition.  Inserting $\operatorname{sign}((P_{C\to F_k})_{ij})$ in the pairwise dissipative jump changes this assembled coefficient to
\[
(P_{C\to F_k})_{ij}\operatorname{sign}((P_{C\to F_k})_{ij})\alpha_i^{F_k}
=\bigl|(P_{C\to F_k})_{ij}\bigr|\alpha_i^{F_k}\geq0.
\]
This nonnegativity is precisely the step used in the entropy proof of Theorem~\ref{thm:global_es_nc} (see \ref{app:proof_global_es}) and in the positivity proof (see \ref{app:proof_pp}); hence the sign correction is required by both arguments rather than introduced as an empirical stabilization.}
{This modification destroys the interpolation property of the dissipative
flux. More precisely, signed interpolation gives
\[
\sum_{j=0}^N(P_{C\to F_k})_{ij}\bm{\mathcal D}_i^k(\xi_j^C)
=
\bm{\mathcal D}_i^k(\xi_i^k)+\mathcal O(h^{N+1}),
\]
whereas, in general,
\[
\sum_{j=0}^N\bigl|(P_{C\to F_k})_{ij}\bigr|
\bm{\mathcal D}_i^k(\xi_j^C)
\neq
\bm{\mathcal D}_i^k(\xi_i^k).
\]
At the matching physical location,
\(\bm U^C(\xi_i^k)-\bm U_i^{F_k}=\mathcal O(h^{N+1})\) for smooth
traces, whereas the pairwise jumps between nodes at distinct physical
locations are generally \(\mathcal O(h)\). The absolute-value weights prevent
the cancellations of these pairwise dissipative fluxes and therefore leave,
in general, a first-order consistency error. Consequently, for \(N\ge2\),
the numerical flux \eqref{lxf} should be regarded as at most first-order
accurate in a formal truncation-error sense.}

In summary, there is an inherent trade-off: without the sign correction, the flux remains formally high-order accurate but lacks provable entropy stability and positivity preservation; with the sign correction, these stability properties are recovered, but at the cost of a reduction in formal accuracy.

{For \(N=1\), the matrices \(P_{C\to F_k}\) are nonnegative, so the sign correction does not alter the fine-side interpolation. This does not, however, make the complete two-sided coupling formally second-order accurate. On the coarse side, \(P_{F_k\to C}\) couples a coarse node to fine nodes at distinct physical locations, including the hanging node, and the corresponding smooth-state differences are generally only \(\mathcal O(h)\). Consistent with this mechanism, the measured ES convergence rates in Section~\ref{subsec:num_vortex} lie between first and second order; the $N=1$ rates are around $1.5$, while most rates for $N\ge2$ are closer to first order.}
\end{remark}
With the nonconforming interface fluxes defined above, the entropy-stable DGSEM on a nonconforming mesh \(\{\Omega_e\}_{e=1}^{N_e}\) is given by \eqref{eq:es_dgsem}. On each element \(\Omega_e\), the volume terms are unchanged, while the interface numerical fluxes on \(\partial\Omega_e\) are specified as follows: on conforming interfaces, the flux \eqref{conforming_flux} with the local Lax--Friedrichs two-point flux \eqref{lax} is used; on nonconforming interfaces, the flux \eqref{eq:fine_flux} or \eqref{eq:coarse_flux}, together with the modified local Lax--Friedrichs two-point flux \eqref{lxf}, is employed depending on whether the edge corresponds to a fine or coarse side.

Next, we state the conservation and entropy properties of the proposed scheme. The proof is given in \ref{app:proof_global_es}.
The analysis follows the general framework developed in \cite{FriedrichEtAl2018,DelReyFernandezEtAl2020,ChanEtAl2021}, adapted to the present setting and notation.

We make the following assumptions on the nonconforming mesh. First, the mesh is \emph{watertight}, meaning that the unit normal vectors are equal and opposite across each shared interface between neighboring elements when evaluated at the same physical location. Second, the discretization is isoparametric, namely
\[
(x^e,y^e)\in \mathbb Q_N(\Omega^{\mathrm{ref}})
\qquad\text{for each element }\Omega_e.
\]
Under this assumption, the scaled normal vector restricted to any edge \(E\subset\partial\Omega_e\) is a polynomial of degree at most \(N-1\), i.e.,
\[
\bm n \in {\mathbb P_{N-1}(E^{\mathrm{ref}})},
\qquad
E^{\mathrm{ref}}=[-1,1].
\]
{For a refined curved interface, isoparametric approximation alone is not sufficient: we additionally require geometric compatibility in the sense of Kopriva et al.~\cite{KoprivaEtAl2019}.  Namely, each fine-edge mapping is the restriction of the parent coarse-edge polynomial (with consistent orientation), and the scaled normals on both sides are computed from this common polynomial geometry.  In two dimensions the edge metric is then represented exactly in the degree-$N$ trace space, so interpolation from the coarse edge to either child reproduces the same physical scaled normal; the affine half-edge reparameterization supplies the factor $2$ below.  Under these geometry conditions, we obtain the curvilinear nonconforming normal identity}
\begin{align}
\label{nrm1}
\sum_{j=0}^N (P_{C\to F_k})_{ij}\,\bm n_j^C
=
\bm n^C(\xi_i^{\,k})
=
-2\bm n_i^{F_k},
\qquad k=1,2.
\end{align}
Moreover, the partition-of-unity property of the Lagrange basis implies
\begin{align}
\label{nrm2}
\sum_{j=0}^N (P_{C\to F_k})_{ij}=1.
\end{align}

\begin{theorem}[Global conservation and entropy stability]
\label{thm:global_es_nc}
Consider the Euler system \eqref{eq:euler} with periodic boundary conditions, and let \(\{\bm U_N^e\}_e\) denote the solution of the entropy-stable DGSEM \eqref{eq:es_dgsem}, with interface fluxes \eqref{conforming_flux} on conforming interfaces and \eqref{eq:fine_flux}--\eqref{eq:coarse_flux} on nonconforming interfaces.

{Assume that {$\mathcal J_{ij}^e>0$ at every volume
node and every nodal state belongs to the admissible set $\mathcal G$,} the geometric compatibility conditions
\eqref{nrm1}--\eqref{nrm2} hold and that all pairwise metric-scaled
normals are nondegenerate.  On conforming interfaces, choose the LF
parameter according to \eqref{alpha}; on nonconforming interfaces, choose
\(\alpha_i^{F_k}\) to satisfy \eqref{eq:alpha_pp_nc} for every nonzero
pairwise contribution.  Then the conforming LF flux satisfies
\begin{equation}
\label{eq:esf_conf}
(\bm V_R-\bm V_L)\cdot\bigl(\bm F^*(\bm U_L,\bm U_R)\cdot\bm n\bigr)
-
({\boldsymbol\psi_R-\boldsymbol\psi_L})\cdot\bm n
\le 0,
\end{equation}
where \(\bm n\) is the metric-scaled normal vector oriented from the left state to the right state.

Moreover, by Lemma~\ref{lem:modified_lf_entropy}, the modified LF flux
satisfies, for all \(0\le i,j\le N\) and \(k=1,2\),
\begin{equation}
\label{eq:esf_nc}
(P_{C\to F_k})_{ij}
\left[
(\bm V_j^C-\bm V_i^{F_k})\cdot
\bigl(\bm F^\star(\bm U_i^{F_k},\bm U_j^C)\cdot\bm n_{ij}^{F_k,C}\bigr)
-
({\boldsymbol\psi_j^C-\boldsymbol\psi_i^{F_k}})\cdot\bm n_{ij}^{F_k,C}
\right]
\le 0.
\end{equation}

Consequently, the scheme is globally conservative and satisfies the semi-discrete entropy inequality}
\begin{equation}
\label{eq:global_cons}
\sum_{e=1}^{N_e}\sum_{i,j=0}^N
w_i w_j\,\mathcal J_{ij}^e\,\partial_t \bm U_{ij}^e
=
\bm 0,
\end{equation}
and
\begin{equation}
\label{eq:global_entropy}
\sum_{e=1}^{N_e}\sum_{i,j=0}^N
w_i w_j\,\mathcal J_{ij}^e\,\partial_t \eta(\bm U_{ij}^e)
\le 0.
\end{equation}
\end{theorem}

\begin{proof}
{See \ref{app:proof_global_es}.}
\end{proof}

\subsection{Positivity preservation}
\label{subsec:pp}

We now establish positivity preservation for the proposed scheme on nonconforming meshes. The goal is to prove that, under forward Euler time stepping and a suitable CFL condition, the cell-average state on each element remains physically admissible, i.e., has positive density and pressure. This, in turn, allows the Zhang--Shu scaling limiter \cite{zhang2010positivity} to be applied, yielding positivity of both density and pressure at all nodal solution points for the fully discrete scheme.

For each element \(\Omega_e\), define the cell average
\begin{equation}
\label{eq:cell_average}
\overline{\bm U}^e
:=
\frac{1}{|\Omega_e|}
\sum_{i,j=0}^N w_i w_j \mathcal J_{ij}^e \bm U_{ij}^{e},
\qquad
|\Omega_e|
:=
\sum_{i,j=0}^N w_i w_j \mathcal J_{ij}^e .
\end{equation}

Using the single-element conservation identity \eqref{eq_cc}, the forward Euler update may be written in the compact form
\begin{equation}
\label{eq:compact_FE}
\overline{\bm U}^{e,\,n+1}
=
\overline{\bm U}^{e,\,n}
-
\frac{\Delta t}{|\Omega_e|}
\sum_{E\subset\partial\Omega_e}\sum_{r=0}^N w_r
\bigl(\bm F_r^{E,*,n}\cdot \bm n_r^E\bigr),
\end{equation}
where the sum ranges over all edges \(E\) of \(\Omega_e\), and \(\bm n_r^E\) denotes the metric-scaled outward normal at the quadrature node \(r\) on \(E\). Depending on the type of edge \(E\), the numerical flux \(\bm F_r^{E,*,n}\cdot \bm n_r^E\) is defined as follows: for conforming interfaces it is given by the local Lax--Friedrichs flux \eqref{lax}; for fine nonconforming interfaces it is given by the projected flux \eqref{eq:fine_flux}; and for coarse nonconforming interfaces it is given by \eqref{eq:coarse_flux}.

We recall the admissible set for the Euler equations,
\begin{equation}
\label{eq:admissible_set}
\mathcal G
:=
\left\{
\bm U=
\begin{pmatrix}
\rho\\
\rho u\\
\rho v\\
E
\end{pmatrix}
:\;
\begin{aligned}
&\rho>0,\\
&p(\bm U)
=
(\gamma-1)\left(E-\frac{(\rho u)^2+(\rho v)^2}{2\rho}\right)>0
\end{aligned}
\right\}.
\end{equation}
The set \(\mathcal G\) is convex and positively homogeneous, i.e.,
\[
\bm U\in\mathcal G
\quad\Longrightarrow\quad
\lambda \bm U\in\mathcal G
\qquad
\forall\,\lambda>0.
\]

We also use the standard admissibility property associated with the local Lax--Friedrichs flux.

\begin{lemma}
\label{lem:LF_admissible}
Let \(\bm U\in\mathcal G\), let \(\bm \nu_u\in\mathbb R^2\) be any unit vector, and let
\[
\alpha_0 \ge |\bm u\cdot \bm \nu_u| + c,
\]
where \(c=\sqrt{\gamma p/\rho}\) denotes the sound speed. Then
\begin{equation}
\label{eq:LF_admissible}
\bm U \pm \frac{1}{\alpha_0}\bm F(\bm U)\cdot \bm \nu_u \in \mathcal G.
\end{equation}
\end{lemma}

\begin{proof}
This is the standard Lax--Friedrichs admissibility property used in positivity analyses for the Euler equations; see, e.g., \cite[Remark 2.4]{zhang2010positivity}.
\end{proof}

We can now state the cell-average positivity result. Its proof is deferred to
{\ref{app:proof_pp}}.

\begin{theorem}[Positivity of the cell average]
\label{thm:cell_average_pp}
Assume that, at time level \(t^n\), all nodal states are admissible:
\[
\bm U_{ij}^{e,n}\in\mathcal G,
\qquad
0\le i,j\le N,
\qquad
\forall\,\Omega_e\in\Omega_h.
\]
{Assume also that the conforming-interface LF parameters satisfy
\eqref{alpha} and that every nonconforming pairwise parameter satisfies
\eqref{eq:alpha_pp_nc}.}
Then the forward Euler update preserves positivity of the cell average,
\[
\overline{\bm U}^{e,\,n+1}\in\mathcal G,
\qquad
\forall\,\Omega_e\in\Omega_h,
\]
provided that, at every boundary node of each element, the total dissipative contribution from all incident edges does not exceed the corresponding cell-average weight. More precisely, for every boundary node \((m,n)\in\mathcal B_e\),
\begin{equation}
\label{eq:CFL_pp}
\Delta t \left(
\sum_{E\ni(m,n)} \beta_r^E
+\sum_{k:\,E_{F_k}\ni(m,n)} \beta_i^{F_k}
+\sum_{E_C\ni(m,n)} \beta_j^C
\right)
\le
w_m w_n \mathcal J_{mn}^e,
\end{equation}
where \(\beta_r^E\), \(\beta_i^{F_k}\), and \(\beta_j^C\) are the nonnegative edgewise dissipation coefficients defined in {\ref{app:proof_pp}}.
\end{theorem}

\begin{proof}
{See \ref{app:proof_pp}.}
\end{proof}

Theorem~\ref{thm:cell_average_pp} guarantees positivity of the cell-average density and pressure under a forward Euler step. The standard Zhang--Shu positivity-preserving scaling limiter can then be applied elementwise to enforce positivity at all nodal solution points while preserving the elementwise cell average. Consequently, the fully discrete forward Euler scheme equipped with this limiter is conservative and {positivity preserving.}

Moreover, since the admissible set \(\mathcal G\) is convex, any strong-stability-preserving Runge--Kutta method that can be written as a convex combination of forward Euler steps inherits the same positivity property. In particular, the SSP-RK3 scheme used in this work preserves positivity of both density and pressure, provided that each forward Euler stage satisfies \eqref{eq:CFL_pp}.

\section{\texorpdfstring{{Positivity-preserving mortar-based flux on nonconforming meshes}}{Positivity-preserving mortar-based flux on nonconforming meshes}}
\label{sec:pp_nc}

{In this section, we review the classical mortar-based flux, which retains high-order consistency for sufficiently smooth solutions. We rigorously prove its positivity-preserving property for curvilinear nonconforming interfaces; however, to the best of our knowledge, a proof of entropy stability for this classical coupling is not presently available.} The construction follows a classical mortar-type approach based on two-point flux evaluations; see, e.g., \cite{kopriva2002electromagnetic}, with the fine edges serving as the mortars.

Instead of using the entropy-stable flux \eqref{eq:fine_flux}--\eqref{eq:coarse_flux}, we interpolate the solution from the coarse edge to the fine-edge nodes and then evaluate a standard two-point local Lax--Friedrichs flux at the fine-edge nodes. {This avoids the sign modification introduced in Section~\ref{sec:amr} and thereby retains the formal high-order consistency of the interface dissipation for sufficiently smooth solutions.}

\subsection{Mortar-based flux construction}
\label{subsec:mortar_flux}

We consider the same nonconforming interface configuration described in Section~\ref{subsec:nc_flux}, where two fine elements \(\Omega_{F_1}\), \(\Omega_{F_2}\) are adjacent to a coarse element \(\Omega_C\). The interpolation operators \(P_{C\to F_k}\) and \(P_{F_k\to C}\) are defined as in \eqref{eq:P_C_to_Fk}--\eqref{eq:P_Fk_to_C}.

\medskip

\noindent
\textbf{Fine-side flux.}
At each fine-edge node \(\xi_i^{F_k}\), we first interpolate the coarse trace to the fine-edge nodes:
\begin{equation}
\label{eq:coarse_interp}
\widetilde{\bm U}_i^{C,k}
:= 
\sum_{j=0}^N (P_{C\to F_k})_{ij}\,\bm U_j^C.
\end{equation}
Since \(\bm U^C\) is represented in the {edge trace space \(\mathbb P_N\)}, the interpolation is exact in the reference coordinate, and thus
\[
\widetilde{\bm U}_i^{C,k}
=
\bm U^C(\xi_i^{\,k}).
\]
We then define the numerical flux by a standard two-point local Lax--Friedrichs flux:
\begin{equation}
\label{eq:fine_flux_mortar}
\bm F_i^{F_k,*}\cdot \bm n_i^{F_k}
:=
\bm F^{*}\bigl(\bm U_i^{F_k}, \widetilde{\bm U}_i^{C,k}\bigr)\cdot \bm n_i^{F_k},
\qquad i=0,\dots,N,\quad k=1,2,
\end{equation}
where \(\bm F^{*}\) denotes the local Lax--Friedrichs flux defined in \eqref{lax}.

\medskip

\noindent
\textbf{Coarse-side flux.}
Since the fine edges are chosen as the mortars, the interface flux is evaluated at the fine-edge nodes and then projected back to the coarse trace space. Accordingly, on the coarse element we define
\begin{equation}
\label{eq:coarse_flux_mortar}
\bm F_j^{C,*}\cdot \bm n_j^C
:=
-2\sum_{k=1}^2 \sum_{i=0}^N
(P_{F_k\to C})_{ji}\,
\bm F^{*}\bigl(\bm U_i^{F_k}, \widetilde{\bm U}_i^{C,k}\bigr)\cdot \bm n_i^{F_k},
\; j=0,\dots,N.
\end{equation}
The factor \(2\) accounts for the metric scaling across the nonconforming interface, as in \eqref{eq:coarse_flux}.

\subsection{Discussion of properties}

{We summarize the key properties of the classical mortar-based flux.}

\medskip
\noindent
\textbf{Conservation.}
The construction \eqref{eq:fine_flux_mortar}--\eqref{eq:coarse_flux_mortar} is conservative by design. Using the compatibility condition \eqref{eq:proj_compat} together with the partition-of-unity property of the interpolation operator, one can verify that the interface contributions cancel exactly between neighboring elements, yielding global conservation in the sense of \eqref{eq:global_cons}.

\medskip
\noindent
\textbf{High-order consistency.}
{Assume that the exact solution and the element mappings are sufficiently smooth in a neighborhood of the interface. Since the interpolated state \(\widetilde{\bm U}_i^{C,k}\) coincides with the coarse polynomial trace evaluated at the same reference location, the jump term
\[
\widetilde{\bm U}_i^{C,k} - \bm U_i^{F_k}
\]
represents the mismatch between two high-order trace approximations at the same physical point. Consequently, the dissipative term in \eqref{eq:fine_flux_mortar} is high-order small in smooth regions. The convergence study in Section~\ref{sec:num} numerically verifies the corresponding high-order behavior on the smooth isentropic-vortex problem.}

\medskip
\noindent
{\textbf{Status of entropy stability.}}
{To the best of our knowledge, a discrete entropy-stability proof for the classical mortar coupling \eqref{eq:fine_flux_mortar}--\eqref{eq:coarse_flux_mortar} has not been established for nonlinear systems in the present setting. Interpolating the conservative state before evaluating the nonlinear two-point flux does not preserve the pairwise algebraic structure used in the entropy proof of Section~\ref{sec:amr}; in particular, the resulting interface contribution cannot be reduced by our argument to a sum of sign-definite pairwise entropy productions. Consistent with this algebraic obstruction, the two entropy tests in Section~\ref{subsec:entropy_test} exhibit actual entropy violations for the mortar coupling: the deterministic local \(1{:}2\) interface test produces positive interface entropy production for admissible states, and the periodic isentropic-vortex simulation exhibits increases in the integrated numerical entropy. In both tests, the proposed entropy-stable flux satisfies the required non-positive entropy behavior under the stated wave-speed conditions.}

\medskip
\noindent
\textbf{Positivity preservation.}
{Despite the absence of an entropy-stability proof for this coupling, the scheme remains positivity preserving under the explicit CFL condition stated below. A key requirement is that the interpolated states \(\widetilde{\bm U}_i^{C,k}\) remain admissible, i.e.,}
\begin{equation}
\label{eq:interp_admissible}
\widetilde{\bm U}_i^{C,k} \in \mathcal G,
\qquad \forall\, i=0,\dots,N,\quad k=1,2.
\end{equation}
{At a mortar node, let \(\alpha_i^{F_k,M}\) be the LF parameter used in \eqref{eq:fine_flux_mortar}, and define
\[
\widehat{\bm n}_i^{F_k}:=
\frac{\bm n_i^{F_k}}{\|\bm n_i^{F_k}\|},
\qquad
a_i^{F_k,M}:=\alpha_i^{F_k,M}\|\bm n_i^{F_k}\|.
\]
We require
\begin{equation}
\label{eq:alpha_pp_mortar}
\alpha_i^{F_k,M}
\geq
\max\!\left\{
\left|\bm u_i^{F_k}\cdot\widehat{\bm n}_i^{F_k}\right|+c_i^{F_k},
\left|\widetilde{\bm u}_i^{C,k}\cdot\widehat{\bm n}_i^{F_k}\right|
+\widetilde c_i^{C,k}
\right\}.
\end{equation}
The nonconforming mortar-edge coefficients are
\begin{equation}
\label{eq:beta_mortar}
\beta_i^{F_k,M}:=w_i a_i^{F_k,M},
\qquad
\beta_j^{C,M}:=
\sum_{k=1}^2\sum_{i=0}^N
w_i a_i^{F_k,M}(P_{C\to F_k})_{ij}.
\end{equation}
The fine-side coefficient \(\beta_i^{F_k,M}\) is nonnegative, whereas
\(\beta_j^{C,M}\) can have either sign because the interpolation matrix can
contain negative entries. We therefore write
\((x)_+:=\max\{x,0\}\).}

{
\begin{theorem}[Cell-average positivity of the mortar coupling]
\label{thm:mortar_pp}
Assume that all nodal states and all interpolated states
\eqref{eq:interp_admissible} are admissible, and that
\eqref{eq:alpha_pp_mortar} holds. Then the forward Euler update of
\eqref{eq:es_dgsem}, with the mortar interface fluxes
\eqref{eq:fine_flux_mortar}--\eqref{eq:coarse_flux_mortar}, satisfies
\(\overline{\bm U}^{e,n+1}\in\mathcal G\), provided that, for every boundary
node \((m,n)\in\mathcal B_e\),
\begin{equation}
\label{eq:CFL_pp_mortar}
\Delta t\left(
\sum_{\substack{E\ni(m,n)\\ E\ {\mathrm{conforming}}}}\beta_r^E
+
\sum_{E_{F_k}\ni(m,n)}\beta_i^{F_k,M}
+
\sum_{E_C\ni(m,n)}(\beta_j^{C,M})_+
\right)
\leq
w_mw_n\mathcal J_{mn}^e.
\end{equation}
Consequently, the fully discrete scheme equipped with the extended
Zhang--Shu limiter remains positivity preserving, provided that every forward
Euler stage satisfies \eqref{eq:CFL_pp_mortar}.
\end{theorem}

\begin{proof}
See \ref{app:proof_mortar_pp}.
\end{proof}
}

\begin{remark}
The admissibility condition \eqref{eq:interp_admissible} is nontrivial due to the presence of negative entries in the interpolation matrix \(P_{C\to F_k}\). In particular, interpolation from the coarse edge to the fine-edge nodes does not, in general, preserve admissibility.
In practice, this issue can be addressed by applying the Zhang--Shu scaling limiter to the coarse-edge solution prior to interpolation, so that the resulting interpolated states at the fine-edge nodes remain admissible while preserving the coarse cell average. A similar strategy for data preparation in adaptive mesh refinement is described in Section~\ref{ppdt}.
{In the implementation, this is realized by enlarging the positivity-checking point set of a coarse nonconforming element: in addition to its LGL solution points, the set contains all child/mortar interpolation points obtained by restricting the coarse polynomial to the fine edges.  A single Zhang--Shu scaling factor is computed from this extended set and applied to the entire coarse polynomial.  Because the scaling is toward the unchanged coarse cell average, every state subsequently supplied to \eqref{eq:fine_flux_mortar} is admissible while conservation is preserved.}
\end{remark}

{In summary, for sufficiently smooth solutions the mortar-based flux \eqref{eq:fine_flux_mortar}--\eqref{eq:coarse_flux_mortar} provides a high-order-consistent and positivity-preserving nonconforming interface treatment. A discrete entropy-stability proof for this classical mortar construction is not presently available.}
{Moreover, entropy violations for this coupling are observed in the numerical experiments of Section~\ref{subsec:entropy_test}.}

\section{Oscillation elimination and adaptive mesh refinement}
\label{sec:oe_amr}

In this section, we present additional components necessary for a robust and practical implementation of positivity-preserving DGSEM on adaptively refined meshes. These include an oscillation elimination procedure on nonconforming meshes and adaptive mesh refinement (AMR) together with data transfer between mesh levels. We conclude the section with a summary of the overall algorithm.

\subsection{Oscillation elimination procedure}
\label{subsec:oe}

As discussed in the introduction, high-order DG approximations may
develop spurious oscillations in the vicinity of discontinuities. To
suppress these oscillations while retaining high-order accuracy in
smooth regions, we employ the oscillation-eliminating (OE) procedure
introduced in \cite{peng2025oedg} as an element-local post-processing
step. In the present setting, the OE operator is applied on
curvilinear meshes with nonconforming interfaces.

The construction follows the conforming-grid case detailed in
\cite[Section 4]{YangFu26} at the element level. The only modification
concerns the face-based jump indicators entering the damping
coefficients: on nonconforming interfaces, the traces on the two sides
are represented on different nodal sets and must therefore be
transferred to a common edge representation before the jump is
evaluated.

For each element $\Omega_e$, let
\[
\mathbb Q_N(\Omega_e)
:=
\{v\in L^2(\Omega_e): v\circ \bm\Phi^e \in \mathbb Q_N(\Omega_{\mathrm{ref}})\}
\]
denote the mapped tensor-product polynomial space of degree $N$. We define an element-local linear
operator
\[
\mathcal F_{\hat t}^e:[\mathbb Q_N(\Omega_e)]^4\to[\mathbb Q_N(\Omega_e)]^4,
\]
which acts componentwise on the DG solution. We write
\[
\bm U_\sigma^{e}(\hat t)=\mathcal F_{\hat t}^e(\bm U_N^{e}),
\]
where $\hat t$ denotes the pseudo time.

The OE operator is defined through the following initial value problem:
find $\bm U_\sigma^{e}(\hat t)\in[\mathbb Q_N(\Omega_e)]^4$ such that
\begin{equation}
\label{eq:oe_ode_nc}
\left\{
\begin{aligned}
&\frac{d}{d\hat t}
\int_{\Omega_e} \bm U_\sigma^{e}\cdot \bm v\,d\bm x
+
\sum_{m=0}^{N}s\,\delta_m^{(e)}(\bm U_N^e)
\int_{\Omega_e}
\bigl(\bm U_\sigma^{e}-P^{m-1}\bm U_\sigma^{e}\bigr)\cdot\bm v\,d\bm x
=0,\\
&\bm U_\sigma^{e}(0)=\bm U_N^{e},
\end{aligned}
\right.
\end{equation}
for all $\bm v\in [\mathbb Q_N(\Omega_e)]^4$,
where $s\in(0,1]$ is a scaling parameter that controls the strength of the OE procedure; see \cite[Remark~4.1]{YangFu26},
$\delta_m^{(e)}$ is the damping coefficient to be specified later, and  $P^m:\mathbb Q_N(\Omega_e)\to \mathbb Q_m(\Omega_e)$ denotes the
scalar $L^2$ projection,
\begin{equation}
\label{eq:oe_projection_nc}
\int_{\Omega_e}(P^m w-w)\,v\,d\bm x=0,
\qquad \forall v\in\mathbb Q_m(\Omega_e),
\end{equation}
extended componentwise to vector-valued functions. For notational
convenience, we set $P^{-1}:=P^0$.

Using the hierarchical decomposition
\[
u=P^0u+\sum_{k=1}^N\bigl(P^ku-P^{k-1}u\bigr),
\qquad u\in \mathbb Q_N(\Omega_e),
\]
applied componentwise to $\bm U_N^e$, one obtains the explicit OE
operator
\begin{equation}
\label{eq:oe_curve_apply_damp_nc}
\bm U_{\sigma}^{e}
=
P^0(\bm U_N^{e})
+
\sum_{k=1}^N
\exp\!\Bigl(
-s\Delta t\sum_{m=0}^{k}\delta_m^{(e)}(\bm U_N^e)
\Bigr)\,
\bigl(P^k(\bm U_N^{e})-P^{k-1}(\bm U_N^{e})\bigr).
\end{equation}
That is, each hierarchical increment
\(
P^k(\bm U_N^{e})-P^{k-1}(\bm U_N^{e})
\)
is damped by the factor
$$
\exp\!\bigl(-s\Delta t\sum_{m=0}^{k}\delta_m^{(e)}(\bm U_N^e)\bigr),
$$
while the constant mode \(P^0(\bm U_N^e)\) remains unchanged. In particular, the OE procedure preserves the cell average of each conserved variable. 

\subsubsection{Damping coefficients}

The damping coefficients are defined by
\begin{equation}
\label{eq:oe_delta_nc}
\delta_m^{(e)}
:=
\frac{\beta_e}{h_e}
{\sum_{f\in\mathcal F_e}}
\sigma_m^{(f)}(\bm U_N^e),
\end{equation}
where {$\beta_e$ is a cell-average-based estimate of the local
characteristic speed} on $\Omega_e$, and $h_e$ is a characteristic element
length scale.
{Here $\mathcal F_e$ contains every conforming face segment and every
fine-scale segment of a nonconforming coarse face incident on $\Omega_e$; in
particular, the two children of a $1{:}2$ coarse face contribute separately.}
We define
\[
h_e := \max_{0\le i,j\le N} \sqrt{\mathcal J_{ij}^e},
\qquad
{\beta_e := \|\bm u(\overline{\bm U}^{e})\|
+ c(\overline{\bm U}^{e}),}
\]
{where $\overline{\bm U}^{e}$ is the cell-average conservative
state.}

For the vector-valued DG solution
\[
\bm U_N^e = (U_N^{e,1},\,U_N^{e,2},\,U_N^{e,3},\,U_N^{e,4})^T,
\]
the face indicator is defined componentwise by
\begin{equation}
\label{eq:oe_sigma_face_nc}
\sigma_m^{(f)}(\bm U_N^e)
=
\max_{1\le r\le 4}\sigma_m^{(f)}(U_N^{e,r}).
\end{equation}

To express the indicator in a unified form, we introduce the face-average operator
\begin{equation}
\label{eq:theta_def}
\Theta_f(w)
:=
\frac{1}{|f|}
\int_f w \, dS,
\end{equation}
for any integrable function $w$ defined on the face $f$.

Using \eqref{eq:theta_def}, the face indicator can be written as
\begin{equation}
\label{eq:oe_sigma_unified}
\sigma_m^{(f)}(U_N^{e,r})
=
\begin{cases}
0,
& {U_N^{r}=\overline{U}_N^{r}},\\[1.2ex]
\dfrac{(2m+1)\,h_e^m}{2(2N-1)\,m!}
{\displaystyle\sum_{|\bm\alpha|=m}}
\dfrac{
\Theta_f\!\left(
\left|
\llbracket \partial^{\bm\alpha} U_N^{e,r} \rrbracket_f
\right|
\right)
}{
{\|U_N^{r}-\overline{U}_N^{r}\|_{\infty,h}}
},
& \text{otherwise},
\end{cases}
\end{equation}
{Here $U_N^r$ is the piecewise polynomial formed by the $r$th
component over the full mesh, and
\[
\overline U_N^r
:=
\frac{\displaystyle\sum_e\sum_{i,j=0}^N
w_iw_j\mathcal J_{ij}^e U_{N,ij}^{e,r}}
{\displaystyle\sum_e\sum_{i,j=0}^N w_iw_j\mathcal J_{ij}^e},
\qquad
\|U_N^r-\overline U_N^r\|_{\infty,h}
:=
\max_{e,i,j}|U_{N,ij}^{e,r}-\overline U_N^r|.
\]
These are the discrete global mean and nodal maximum deviation used in the
implementation.} The vector $\bm\alpha=(\alpha_1,\alpha_2)$ is a multi-index with
$|\bm\alpha|=\alpha_1+\alpha_2$, and
\[
\partial^{\bm\alpha}U
=
\frac{\partial^{|\bm\alpha|}U}
{\partial x^{\alpha_1}\partial y^{\alpha_2}}.
\]

The remaining ingredient is the definition of the jump operator
$\llbracket\cdot\rrbracket_f$, which depends on the type of interface.

\medskip
\noindent
\textbf{Conforming faces.}
If $f$ is a conforming interface between $\Omega_e$ and a neighboring
element $\Omega_{e'}$, the jump is defined by
\[
\llbracket w \rrbracket_f := w|_{\Omega_e} - w|_{\Omega_{e'}}.
\]

\noindent
\textbf{Nonconforming faces.}
On nonconforming interfaces, the traces on the two sides are represented
on different nodal sets and must be transferred to a common fine face
representation using the projection operator
\eqref{eq:P_C_to_Fk}.

For a fine nonconforming face $f=E_{F_k}\subset\partial\Omega_e$, the jump is
evaluated on the fine nodal set as
\[
\llbracket w \rrbracket_f
:=
w|_{\Omega_e}
-
\mathcal I_{C\to F_k}(w|_{\Omega_{e'}}),
\]
where $\Omega_{e'}$ is the neighboring coarse element, and
$\mathcal I_{C\to F_k}$ is the nodal interpolation operator that maps
data from the coarse face to the fine face $F_k$.

For a coarse nonconforming face $E_C\subset\partial\Omega_e$,
{the jump is defined separately on each fine-scale segment
$f=E_{F_k}$, $k=1,2$, by}
\[
{\llbracket w \rrbracket_{E_{F_k}}}
:=
{\mathcal I_{C\to F_k}(w|_{\Omega_e})
- w|_{\Omega_{e_k}},}
\]
where $\Omega_{e_k}$ are the neighboring fine elements corresponding to
$E_{F_k}$. {Each segment contributes its own face average
$\Theta_{E_{F_k}}$ to the sum in \eqref{eq:oe_delta_nc}.}

With these definitions, \eqref{eq:oe_sigma_unified} applies uniformly to
both conforming and nonconforming interfaces.

\subsubsection{Shock indicator and selective application of OE}
As in the conforming-grid case, evaluating all damping coefficients on
every element is computationally expensive. We therefore apply the OE
operator only to \emph{troubled elements} identified by a low-cost shock
indicator.

For each element $\Omega_e$, we define
\begin{equation}
\label{eq:oe_indicator_nc}
\mathcal I^{(e)}(\bm U_N^e)
=
{\sum_{f\in\mathcal F_e}}
\sigma_0^{(f)}(\bm U_N^e),
\end{equation}
where $\sigma_0^{(f)}$ is computed using the unified definition
\eqref{eq:oe_sigma_unified}, with the jump operator interpreted according
to the type of face $f$.

An element $\Omega_e$ is marked as troubled if
\begin{equation}
\label{eq:oe_threshold_nc}
\mathcal I^{(e)}(\bm U_N^e)>C_{\mathrm{oe}},
\end{equation}
where $C_{\mathrm{oe}}>0$ is a user-prescribed threshold. For troubled elements, the
full set of damping coefficients $\{\delta_m^{(e)}\}_{m=0}^N$ is computed
from \eqref{eq:oe_delta_nc}, and the OE update
\eqref{eq:oe_curve_apply_damp_nc} is applied. Otherwise, the solution is
left unchanged.

\begin{algorithm}[H]\small
  \DontPrintSemicolon
  \SetKwInOut{Input}{Input}\SetKwInOut{Output}{Output}

  \Input{DG solution $\{\bm U_N^e\}_e$; threshold constant $C_{\mathrm{oe}}>0$.}
  \Output{OE-processed solution $\{\bm U_\sigma^e\}_e$.}

  \ForEach{element $\Omega_e$}{
    Compute $\mathcal I^{(e)}(\bm U_N^e)$ from \eqref{eq:oe_indicator_nc}\;
    \eIf{$\mathcal I^{(e)}(\bm U_N^e)>C_{\mathrm{oe}}$}{
      Compute $\delta_m^{(e)}$ from \eqref{eq:oe_delta_nc}\;
      Apply \eqref{eq:oe_curve_apply_damp_nc} to obtain $\bm U_\sigma^e$\;
    }{
      Set $\bm U_\sigma^e \leftarrow \bm U_N^e$\;
    }
  }
  \caption{Selective OE procedure on curvilinear nonconforming meshes.}
  \label{alg:oe_nc}
\end{algorithm}

\subsection{Adaptive mesh refinement}
\label{subsec:amr}

We now describe the adaptive mesh refinement (AMR) strategy. As noted in
the introduction, we employ a $2{:}1$ refinement procedure driven by the
shock indicator \eqref{eq:oe_indicator_nc}. The same indicator used in
the OE procedure is reused here to detect under-resolved regions,
providing a unified mechanism for oscillation control and mesh
adaptation.

\subsubsection{Refinement and coarsening criteria}

For each element $\Omega_e$, we evaluate the indicator
$\mathcal I^{(e)}(\bm U_N^e)$, which is exactly the same indicator used for {shock detection}. Given two thresholds
$C_{\mathrm{ref}}$ and $C_{\mathrm{crs}}$ with
$C_{\mathrm{ref}}\geq C_{\mathrm{crs}}>0$, we apply the following {criteria}:
\begin{itemize}
\item \emph{Refinement}: if $\mathcal I^{(e)}(\bm U_N^e) > C_{\mathrm{ref}}$,
the element $\Omega_e$ is marked for refinement.
\item \emph{Coarsening}: if $\mathcal I^{(e)}(\bm U_N^e) < C_{\mathrm{crs}}$,
the element is marked for coarsening.
\end{itemize}

We employ the AMR infrastructure provided by \texttt{MFEM} \cite{anderson2021mfem}, which consists of a refinement stage followed by a coarsening stage. The mesh is maintained under a $2{:}1$ balance constraint throughout both stages. In the coarsening stage, an element $\Omega_e$ is coarsened only if all of its four children $\{\Omega_{e_k}\}_{k=1}^4$ are marked for coarsening. Parallel load balancing is performed after mesh adaptation to maintain computational efficiency. In addition, a maximum refinement level $L$ is enforced, so that once an element has been subdivided $L$ times, it is no longer refined, thereby imposing a lower bound on the mesh size.

{After every refinement, coarsening, and parallel rebalancing operation, geometric data are not transferred by projecting the metric field from the old mesh.  Instead, the Jacobians, inverse Jacobians, contravariant metric terms, metric-scaled face normals, and element volumes are recomputed from the updated element transformations using the same isoparametric construction as in Section~\ref{sec:conforming_esdgsem}.  Thus the subsequent DG operator and transfer quadrature use geometry associated with the current AMR mesh.}

{
\subsubsection{Admissible data transfer}
\label{ppdt}

Mesh adaptation changes the finite element space and therefore requires the
degrees of freedom of the DG solution to be reconstructed on the new mesh.
Both conservation and positivity preservation are essential during this
operation. A conservative transfer prevents refinement and coarsening from
artificially creating or removing mass, momentum, or total energy, which could
otherwise alter shock speeds and strengths or contaminate resolved flow
structures. A positivity-preserving transfer keeps the reconstructed states
inside $\mathcal G$, so that the positivity argument remains valid and the
Zhang--Shu limiter can be initialized from an admissible average on the new
mesh.

Conservative and positivity-preserving AMR transfers have already been used in
high-order adaptive solvers. In particular, the tree-based AMR machinery in
\texttt{Trixi.jl} transfers the Jacobian-weighted conservative variables
$\mathcal J\bm U$ between refinement levels: $\mathcal J\bm U$ is prolonged
to the children during refinement and projected back to the parent during
coarsening, after which division by the new Jacobian recovers $\bm U$ and
positivity limiting is applied \cite{RanochaEtAl2022}. Here we take a different
route. During refinement, the parent polynomial is restricted to the child
nodes; during coarsening, {the piecewise child solution is evaluated
at the parent nodes using the deterministic one-sided convention defined
below}. In both cases,
we then add a patchwise rank-one correction computed from the actual discrete
physical measures before and after adaptation. Consequently, the proposed
transfer does not require a separate transfer of $\mathcal J$ or
$\mathcal J\bm U$, or a pointwise interpolation identity between the parent
and child Jacobians.

For each element $\Omega_e$, let $\mathcal S_e$ denote the positivity-checking
set associated with the current discretization. It contains the element LGL
nodes and, when required by a nonconforming interface treatment, the child or
mortar interpolation points at which states are supplied to the numerical
flux. We define the discrete admissible set on a mesh $\mathcal T_h$ by
\begin{equation}
\label{eq:discrete_admissible_set}
\mathcal G_h(\mathcal T_h)
:=
\left\{
\bm U_h:
\overline{\bm U}^{e}\in\mathcal G
\ \text{and}\ 
\bm U_h^{e}(\bm x)\in\mathcal G
\quad
\forall\,\bm x\in\mathcal S_e,\quad
\forall\,\Omega_e\in\mathcal T_h
\right\}.
\end{equation}

\begin{definition}[Admissible AMR data transfer]
\label{def:admissible_transfer}
Let $\mathcal T_h$ and $\mathcal T_h'$ denote the meshes immediately before
and after an AMR operation. A data-transfer operator
\[
\mathfrak T_{h\to h'}:\bm U_h\longmapsto\bm U_h'
\]
is called an \emph{admissible data transfer} if it satisfies both of the
following properties for every refinement or coarsening patch $\mathcal P$:
\begin{enumerate}
\item patchwise conservation,
\begin{equation}
\label{eq:transfer_conservation}
\sum_{\substack{\Omega_{e'}\subset\mathcal P\\
                  \Omega_{e'}\in\mathcal T_h'}}
|\Omega_{e'}|\,\overline{\bm U}^{e'}
=
\sum_{\substack{\Omega_e\subset\mathcal P\\
                  \Omega_e\in\mathcal T_h}}
|\Omega_e|\,\overline{\bm U}^{e}.
\end{equation}
\item preservation of the discrete admissible structure,
\begin{equation}
\label{eq:transfer_admissibility}
\bm U_h\in\mathcal G_h(\mathcal T_h)
\quad\Longrightarrow\quad
\bm U_h'\in\mathcal G_h(\mathcal T_h').
\end{equation}
\end{enumerate}
Thus, throughout this work, the word \emph{admissible} applied to an AMR
transfer includes both conservation and positivity preservation.
\end{definition}

We next introduce notation common to refinement and coarsening. Let
$\mathcal I_{\mathcal P}^{-}$ and $\mathcal I_{\mathcal P}^{+}$ denote all
LGL degrees of freedom in the patch before and after adaptation, respectively.
The corresponding discrete physical weights and patch moments are
\begin{equation}
\label{eq:amr_patch_moments}
\mu_q^{\pm}=w_q^{\pm}\mathcal J_q^{\pm}>0,
\qquad
V_{\mathcal P}^{\pm}
=\sum_{q\in\mathcal I_{\mathcal P}^{\pm}}\mu_q^{\pm},
\qquad
\bm M_{\mathcal P}^{\pm}(\bm W)
=\sum_{q\in\mathcal I_{\mathcal P}^{\pm}}
 \mu_q^{\pm}\bm W_q .
\end{equation}
Here $w_q^{\pm}$ denotes the appropriate tensor-product LGL weight; the
element index is included in $q$ when the patch contains several elements.
{Throughout this subsection, the superscripts $-$ and $+$ refer to
the meshes immediately before and after the AMR operation, respectively. We
distinguish the following four patchwise nodal state vectors:}
\begin{equation}
\label{eq:amr_state_notation}
{
\begin{aligned}
\bm U_{\mathcal P}^{-}
&:=\bigl(\bm U_q^{-}\bigr)_{q\in\mathcal I_{\mathcal P}^{-}},
&&\text{the old nodal state},\\
\widehat{\bm U}_{\mathcal P}^{+}
&:=\bigl(\widehat{\bm U}_q^{+}\bigr)_{q\in\mathcal I_{\mathcal P}^{+}},
&&\text{the provisional state produced by nodal transfer},\\
\widetilde{\bm U}_{\mathcal P}^{+}
&:=\bigl(\widetilde{\bm U}_q^{+}\bigr)_{q\in\mathcal I_{\mathcal P}^{+}},
&&\text{the state after the rank-one conservation correction},\\
\bm U_{\mathcal P}^{+}
&:=\bigl(\bm U_q^{+}\bigr)_{q\in\mathcal I_{\mathcal P}^{+}},
&&\text{the final state after joint positivity scaling}.
\end{aligned}}
\end{equation}
{The elementwise quantities used below, such as
$\widehat{\bm U}_{ij}^{e_k}$ and $\widetilde{\bm U}_{ij}^{e_k}$, are the
corresponding entries of these patchwise vectors. The functional
$\bm M_{\mathcal P}^{\pm}$ is always applied to a nodal vector indexed by
$\mathcal I_{\mathcal P}^{\pm}$.}

\paragraph{Refinement transfer.}
Suppose that a parent element $\Omega_e$ is divided into four children
$\{\Omega_{e_k}\}_{k=1}^4$. We first evaluate the parent polynomial directly
at the child LGL nodes,
\begin{equation}
\label{eq:amr_raw_refinement}
\widehat{\bm U}_{ij}^{e_k}
=\bm U_N^e\bigl(\chi_k(\widehat{\bm x}_{ij})\bigr),
\qquad k=1,\ldots,4,
\end{equation}
where $\chi_k:\widehat\Omega\to\widehat\Omega_k$ maps the child reference
element onto subdomain $k$ of the parent reference element. Using the
actual old and new Jacobian-weighted quadrature measures, define
\begin{equation}
\label{eq:refinement_rank_one}
\bm\delta_{\mathrm{ref}}
=\frac{
 \bm M_{\mathcal P}^{-}(\bm U_{\mathcal P}^-)
 -\bm M_{\mathcal P}^{+}(\widehat{\bm U}_{\mathcal P}^{+})}
 {V_{\mathcal P}^{+}},
\qquad
\widetilde{\bm U}_{ij}^{e_k}
=\widehat{\bm U}_{ij}^{e_k}+\bm\delta_{\mathrm{ref}}.
\end{equation}
This is a single constant-vector correction shared by all four children. Set
\begin{equation}
\label{eq:amr_target_average}
\overline{\bm U}_{\mathcal P}
:=\frac{\bm M_{\mathcal P}^{-}(\bm U_{\mathcal P}^-)}{V_{\mathcal P}^{+}}.
\end{equation}
Let $\mathcal S_{\mathcal P}^{+}$ be the union of all child positivity sets,
including the child LGL nodes and every child/mortar interpolation point
required by the new nonconforming interfaces. A single Zhang--Shu parameter
$\theta_{\mathrm{ref}}\in[0,1]$ is selected over this entire set, and the
final child polynomials are
\begin{equation}
\label{eq:refinement_joint_scaling}
\bm U_N^{e_k}(\bm x)
=\overline{\bm U}_{\mathcal P}
+\theta_{\mathrm{ref}}
 \left(
  \widetilde{\bm U}_N^{e_k}(\bm x)
  -\overline{\bm U}_{\mathcal P}
 \right),
\quad
\bm U_N^{e_k}(\bm x)\in\mathcal G
\quad\forall\bm x\in\mathcal S_{\mathcal P}^{+}.
\end{equation}

\paragraph{Coarsening transfer.}
Suppose that four children $\{\Omega_{e_k}\}_{k=1}^4$ are merged into their
parent element $\Omega_e$. For $\widehat{\bm x}\in\widehat\Omega_k$, define
the piecewise child polynomial on the parent reference element by
\begin{equation}
\label{eq:piecewise_child_reference_state}
\bm U_{\mathrm{pw}}(\widehat{\bm x})
:=\bm U_N^{e_k}\bigl(\chi_k^{-1}(\widehat{\bm x})\bigr),
\qquad \widehat{\bm x}\in\widehat\Omega_k,
\end{equation}
where the interiors of the four subdomains $\widehat\Omega_k$ are disjoint and
their closures cover $\widehat\Omega$. {Fix once and for all a
priority ordering of the four children, denoted by
$\pi_1<\pi_2<\pi_3<\pi_4$. For every parent LGL node
$\widehat{\bm x}_{ij}$, define the set of children whose closures contain that
node by}
\begin{equation}
\label{eq:coarsening_owner_set}
{
\mathcal K_{ij}
:=
\left\{k\in\{1,2,3,4\}:\widehat{\bm x}_{ij}
\in\overline{\widehat\Omega_k}\right\},
\qquad
\kappa(i,j)
:=
\operatorname*{arg\,max}_{k\in\mathcal K_{ij}}\pi_k.}
\end{equation}
{Thus an interior node of a child is assigned to that child, while
a node on a common child edge or at the four-child junction is assigned
uniquely by the fixed priority ordering. The provisional parent nodal values
are obtained by evaluating the selected child polynomial:}
\begin{equation}
\label{eq:amr_raw_coarsening}
{
\widehat{\bm U}_{ij}^{e}
=
\bm U_N^{e_{\kappa(i,j)}}
\!\left(\chi_{\kappa(i,j)}^{-1}(\widehat{\bm x}_{ij})\right),
\qquad 0\leq i,j\leq N.}
\end{equation}
{In operator form, this provisional transfer is}
\begin{equation}
\label{eq:amr_raw_coarsening_matrix}
{
\widehat{\bm U}_{\mathcal P}^{+}
=\mathcal R_{\rm nod}\bm U_{\mathcal P}^{-},}
\end{equation}
{where $\mathcal R_{\rm nod}$ is completely specified by
\eqref{eq:coarsening_owner_set}--\eqref{eq:amr_raw_coarsening}. It reproduces
constant states for any priority ordering, but it is a nodal restriction
rather than a physical-Jacobian-weighted projection. It therefore need not
preserve the physical patch moment on a curved mesh. The priority ordering can
affect the provisional values at child interfaces and hence the final
corrected polynomial. However, the conservation and positivity conclusions
proved below hold for every fixed ordering.}
Its possibly nonzero conservation defect is removed by
\begin{equation}
\label{eq:coarsening_rank_one}
\bm\delta_{\mathrm{crs}}
=\frac{
 \bm M_{\mathcal P}^{-}(\bm U_{\mathcal P}^-)
 -\bm M_{\mathcal P}^{+}(\widehat{\bm U}_{\mathcal P}^{+})}
 {V_{\mathcal P}^{+}},
\qquad
\widetilde{\bm U}_{ij}^{e}
=\widehat{\bm U}_{ij}^{e}+\bm\delta_{\mathrm{crs}}.
\end{equation}
With the same target average \eqref{eq:amr_target_average}, a common
$\theta_{\mathrm{crs}}\in[0,1]$ is then computed over the new parent checking
set $\mathcal S_e$, including any mortar interpolation points, and
\begin{equation}
\label{eq:coarsening_joint_scaling}
\bm U_N^e(\bm x)
=\overline{\bm U}_{\mathcal P}
+\theta_{\mathrm{crs}}
 \left(
  \widetilde{\bm U}_N^e(\bm x)
  -\overline{\bm U}_{\mathcal P}
 \right)
\in\mathcal G,
\qquad \forall\bm x\in\mathcal S_e.
\end{equation}

\begin{theorem}[Conservation of the corrected AMR transfer]
\label{thm:conservative_amr_transfer}
Assume $\mu_q^+>0$ and hence $V_{\mathcal P}^+>0$. For either refinement or
coarsening, the rank-one correction
\eqref{eq:refinement_rank_one} or \eqref{eq:coarsening_rank_one}, followed by
the corresponding joint scaling \eqref{eq:refinement_joint_scaling} or
\eqref{eq:coarsening_joint_scaling}, satisfies the patchwise conservation
identity \eqref{eq:transfer_conservation}.
\end{theorem}

\begin{proof}
Both corrections have the common form
\[
\widetilde{\bm U}_q^+
=\widehat{\bm U}_q^+
+\frac{\bm M_{\mathcal P}^{-}(\bm U_{\mathcal P}^-)
       -\bm M_{\mathcal P}^{+}(\widehat{\bm U}_{\mathcal P}^+)}
      {V_{\mathcal P}^{+}}.
\]
Therefore,
\[
\bm M_{\mathcal P}^{+}(\widetilde{\bm U}_{\mathcal P}^+)
=\bm M_{\mathcal P}^{+}(\widehat{\bm U}_{\mathcal P}^+)
+\frac{\sum_q\mu_q^+}{V_{\mathcal P}^{+}}
 \left[
  \bm M_{\mathcal P}^{-}(\bm U_{\mathcal P}^-)
  -\bm M_{\mathcal P}^{+}(\widehat{\bm U}_{\mathcal P}^+)
 \right]
=\bm M_{\mathcal P}^{-}(\bm U_{\mathcal P}^-).
\]
The target average consequently satisfies
$V_{\mathcal P}^+\overline{\bm U}_{\mathcal P}
=\bm M_{\mathcal P}^{+}(\widetilde{\bm U}_{\mathcal P}^+)$. {For any value of the}
common scaling parameter $\theta$,
\[
\bm M_{\mathcal P}^{+}(\bm U_{\mathcal P}^+)
=V_{\mathcal P}^+\overline{\bm U}_{\mathcal P}
+\theta\left[
 \bm M_{\mathcal P}^{+}(\widetilde{\bm U}_{\mathcal P}^+)
 -V_{\mathcal P}^+\overline{\bm U}_{\mathcal P}
\right]
=\bm M_{\mathcal P}^{-}(\bm U_{\mathcal P}^-),
\]
which proves \eqref{eq:transfer_conservation} for both AMR operations.
\end{proof}

\begin{theorem}[Positivity of the corrected AMR transfer]
\label{thm:positive_amr_transfer}
Assume that all old and new Jacobian-weighted quadrature weights are positive
and that $\bm U_h\in\mathcal G_h(\mathcal T_h)$. Then the common parameters
$\theta_{\mathrm{ref}}$ and $\theta_{\mathrm{crs}}$ in
\eqref{eq:refinement_joint_scaling} and
\eqref{eq:coarsening_joint_scaling} can be chosen by the Zhang--Shu
construction so that the transferred solution belongs to
$\mathcal G_h(\mathcal T_h')$. Hence both corrected transfers are admissible
in the sense of Definition~\ref{def:admissible_transfer}.
\end{theorem}

\begin{proof}
Because the old LGL weights are positive and the old nodal states are in the
convex set $\mathcal G$, the old patch average
\[
\overline{\bm U}_{\mathcal P}^{-}
=\frac{\bm M_{\mathcal P}^{-}(\bm U_{\mathcal P}^-)}{V_{\mathcal P}^{-}}
\]
belongs to $\mathcal G$. Moreover,
\[
\overline{\bm U}_{\mathcal P}
=\frac{V_{\mathcal P}^{-}}{V_{\mathcal P}^{+}}
 \overline{\bm U}_{\mathcal P}^{-}\in\mathcal G,
\]
because $V_{\mathcal P}^{\pm}>0$ and the Euler admissible set is positively
homogeneous. Thus the target state about which the new polynomial is scaled
is admissible. Convexity of $\mathcal G$ and the standard Zhang--Shu argument
then provide a single parameter $\theta\in[0,1]$ that places every state in
the finite new checking set in $\mathcal G$. Since that set contains all new
LGL nodes, every new element average is a positive weighted combination of
admissible states and also lies in $\mathcal G$. The conservation part of
Definition~\ref{def:admissible_transfer} follows from
Theorem~\ref{thm:conservative_amr_transfer}.
\end{proof}

\begin{remark}[Geometric requirements and flexibility]
{The two theorems use the actual old and new Jacobian values only
through the positive scalar weights in \eqref{eq:amr_patch_moments}; therefore,
their conservation and positivity conclusions do not require
$V_{\mathcal P}^-=V_{\mathcal P}^+$ or a pointwise parent--child Jacobian
transfer identity. Exact constant-state, and hence free-stream, preservation
is an additional property under compatible isoparametric refinement. More
precisely, if each child mapping is the exact affine restriction of the parent
mapping, then $\mathcal J^{\rm parent}\in
\mathbb Q_{2N-1,2N-1}$, the child Jacobians are the corresponding restricted
Jacobians with the affine scaling included, and the tensor-product LGL rule is
exact in each coordinate through degree $2N-1$. Consequently,
$V_{\mathcal P}^-=V_{\mathcal P}^+$, and the corrected restriction preserves a
constant state exactly. The same conclusion holds for the inverse coarsening
operation on this compatible refinement hierarchy. If the child geometry is
instead re-fitted or independently projected, volume compatibility must be
checked separately; even in that case, Theorems~\ref{thm:conservative_amr_transfer}
and~\ref{thm:positive_amr_transfer} remain unchanged.}

Although the formulation above fixes the provisional refinement and
coarsening operators to parent-to-child nodal restriction and the
child-to-parent nodal restriction
\eqref{eq:coarsening_owner_set}--\eqref{eq:amr_raw_coarsening}, respectively,
the rank-one correction and the two proofs only use the
provisional nodal states and their new physical patch moment. The same
correction can therefore, in principle, be placed after another
constant-preserving transfer. In this limited sense, the construction is more
flexible than a transfer tied directly to a particular representation of
$\mathcal J\bm U$; other provisional transfer operators are not investigated
in this work.
\end{remark}
}

\subsection{Summary of the full algorithm}
\label{subsec:algorithm}

{We summarize the complete practical algorithm on nonconforming
curvilinear meshes with positivity limiting, oscillation elimination,
and adaptive mesh refinement.  The term entropy stable below refers to
the underlying semi-discrete spatial operator equipped with the
entropy-stable interface flux; it does not assert entropy stability of
the entire fully discrete OE--AMR algorithm.}

Given an initial mesh $\{\Omega_e\}_e$ and the DG solution
$\{\bm U_N^e\}_e$ at time $t^n$, one time step proceeds as follows.

\begin{enumerate}

\item \textbf{Time integration.}
Advance the DG solution using an SSP Runge--Kutta method (SSP-RK3 is
used in the numerical experiments), under a proper CFL condition
for time step size. At each stage, the semi-discrete system
\eqref{eq:es_dgsem} is evaluated using
\begin{itemize}
\item the local Lax--Friedrichs flux \eqref{lax} on conforming interfaces,
\item {either the entropy-stable nonconforming fluxes
\eqref{eq:fine_flux}--\eqref{eq:coarse_flux} with the modified flux
\eqref{lxf}, or the mortar fluxes
\eqref{eq:fine_flux_mortar}--\eqref{eq:coarse_flux_mortar}, on
nonconforming interfaces.}
\end{itemize}

{At every Runge--Kutta stage, the cell averages are checked before
OE and positivity limiting. If any rank detects a non-admissible or non-finite
cell average, or if positivity limiting fails, the entire time step is
rejected, $(\bm U,t)$ is restored to $(\bm U^n,t^n)$, and the step is retried
with $\Delta t/2$. The computation terminates with an explicit diagnostic if
the retried step size falls below $\Delta t_{\min}=10^{-9}$. The sharp
positivity CFL bound is generally difficult, and unnecessary, to determine
exactly in practical computations. The positivity result instead guarantees
the existence of a sufficiently small admissible time step; therefore, a
failed stage can be handled by reducing $\Delta t$ and restarting the complete
step. We refer to \cite{YangFu26} for the motivation and further details of
this restart strategy.}

\item \textbf{Oscillation elimination (OE).}
After each stage, compute the shock indicator
$\mathcal I^{(e)}(\bm U_N^e)$ using \eqref{eq:oe_indicator_nc} with
threshold $C_{\mathrm{oe}}>0$, and apply the OE procedure described in
Algorithm~\ref{alg:oe_nc}.

\item \textbf{Positivity enforcement.}
After OE, apply the Zhang--Shu positivity-preserving limiter on each
element to ensure that all nodal values remain in the admissible set
$\mathcal G$.

\item \textbf{Adaptive mesh refinement and data transfer.}
At prescribed intervals (every $N_{\mathrm{AMR}}$ time steps), perform
mesh adaptation using the same indicator
$\mathcal I^{(e)}(\bm U_N^e)$. Elements are marked for refinement and
coarsening based on thresholds $C_{\mathrm{ref}}$ and
$C_{\mathrm{crs}}$, with $C_{\mathrm{ref}}\geq C_{\mathrm{crs}}>0$.

\begin{itemize}
\item \emph{Refinement stage:}
{Apply the parent-to-child nodal restriction, restore the old
patch moment with the rank-one correction \eqref{eq:refinement_rank_one}, and
apply one joint Zhang--Shu scaling over all child and mortar checking points as
in \eqref{eq:refinement_joint_scaling}. Perform parallel load balancing after
refinement.}

\item \emph{Coarsening stage:}
{Coarsen an element only when all four children are marked. Apply
the deterministic child-to-parent nodal restriction
\eqref{eq:coarsening_owner_set}--\eqref{eq:amr_raw_coarsening}, restore the
total child moment with the rank-one
correction \eqref{eq:coarsening_rank_one}, and apply one Zhang--Shu scaling
over the new parent checking set as in \eqref{eq:coarsening_joint_scaling}.
Perform parallel load balancing after coarsening.}
\end{itemize}

\end{enumerate}

{
\begin{remark}[Scope of the entropy statement]
The entropy inequality proved in Theorem~\ref{thm:global_es_nc} is a
semi-discrete property of the spatial DGSEM operator.  The Zhang--Shu
limiter replaces every nodal state by a common convex scaling toward the
cell average.  By convexity of the mathematical entropy and the positive
quadrature weights, this scaling step does not increase the elementwise
discrete entropy; see Chen and Shu~\cite{ChenShu2017}.  This compatibility
does not, however, supply a fully discrete entropy proof for an explicit
high-order Runge--Kutta method.  Moreover, although the OE step damps the
nonconstant modal content and decreases the corresponding discrete
\(L^2\) energy, no proof is presently available that it is non-increasing
for the nonlinear Euler entropy; see also the discussion and numerical
evidence in~\cite{YangFu26}.  The AMR transfer is likewise not included in
the semi-discrete entropy theorem.  We therefore claim conservation and
positivity for the complete algorithm, and entropy stability only for the
semi-discrete DGSEM spatial discretization under the hypotheses of
Theorem~\ref{thm:global_es_nc}.
\end{remark}

The complete procedure is conservative and positivity preserving under
the stated CFL and admissibility assumptions, while enabling adaptive
resolution of localized features.  Its underlying spatial DGSEM operator
is semi-discretely entropy stable when the entropy-stable nonconforming
flux is used and its wave-speed conditions are satisfied.
}

\section{Numerical results}
\label{sec:num}
{In this section, we present a series of numerical examples to assess
the performance of the proposed method. Our goals are fourfold: first, to
directly examine the local entropy production of the nonconforming interface
fluxes; second, to examine their accuracy; third, to numerically verify that
the positivity-preserving mechanism remains effective in the AMR setting;
and fourth, to demonstrate that the proposed AMR strategy is effective on
both Cartesian and curvilinear meshes and can capture complicated flow
structures in challenging compressible-flow problems.} All simulations are
carried out within the open-source \texttt{MFEM} framework
\cite{anderson2021mfem,andrej2024high}.
{In every example that employs adaptive mesh refinement, the mesh
is adapted once every ten time steps.}
{Unless explicitly stated otherwise, every numerical flux in
this section uses both the characteristic-speed and Guermond--Popov bounds.
For a conforming or mortar two-state flux, this is the combined parameter
\eqref{alpha}. For the proposed entropy-stable nonconforming coupling, it
means the scaled pairwise condition \eqref{eq:alpha_pp_nc}: each
$\alpha_i^{F_k}$ is chosen to satisfy that condition for every coupled coarse
node $j$.}

Time integration is performed using the third-order
strong-stability-preserving Runge--Kutta method ({SSP-RK3}).
{For a time step that directly satisfies the sufficient positivity
condition proved above, define $\mathcal C_{mn}^e$ to be the complete
nonnegative boundary coefficient multiplying $\Delta t$ on the left-hand
side of \eqref{eq:CFL_pp} for the entropy-stable coupling, or of
\eqref{eq:CFL_pp_mortar} for the mortar coupling. The theoretical time-step
choice is then}
\begin{equation}
\label{eq:cfl_num}
{
\Delta t_{\rm PP}
=
\mathrm{CFL}
\min_{\Omega_e}
\min_{(m,n)\in\mathcal B_e}
\frac{w_mw_n\mathcal J_{mn}^e}{\mathcal C_{mn}^e},
\qquad 0<\mathrm{CFL}\leq1,}
\end{equation}
{where a vanishing denominator imposes no restriction. This is a
direct restatement of the proved nodal CFL conditions and is recomputed from
the current stage states and interface dissipation coefficients.}

{For large-scale shock calculations, this sufficient bound can be
considerably more restrictive than necessary. In the double Mach reflection,
astrophysical-jet, and circular-cylinder examples, we therefore use the
standard practical estimate}
\begin{equation}
\label{eq:cfl_char_practical}
{
\Delta t_{\rm char}
=
\frac{\mathrm{CFL}}{2N+1}\,
\min_{\Omega_e} \frac{ h_e}{\bigl(\|\bm u_e\|+c_e\bigr)}.}
\end{equation}
{Here $N$ is the polynomial degree, $h_e$ is a characteristic
local mesh size, and $\bm u_e$ and $c_e$ are the local velocity and sound
speed. The numerical flux still uses the wave-speed bounds specified above;
only the initial time-step estimate uses \eqref{eq:cfl_char_practical}. If a
Runge--Kutta stage violates cell-average admissibility, the restart mechanism
of Section~\ref{subsec:algorithm} rejects the step, halves $\Delta t$, and
retries it. This combines the inexpensive characteristic-speed estimate with
the robustness guaranteed by the existence of a sufficiently small
positivity-preserving time step. Unless otherwise stated, we set
$\mathrm{CFL}=0.8$.}
{The OE parameters $s$ and $C_{\mathrm{oe}}$ are
problem-dependent and are specified separately for each shock-dominated test
below. The OE procedure is disabled in the smooth-vortex convergence and
entropy-history tests.}
{The refinement and coarsening thresholds
$C_{\mathrm{ref}}$ and $C_{\mathrm{crs}}$ are likewise}
problem-dependent and will be specified separately for each test case.

\subsection{Isentropic vortex problem}
\label{subsec:num_vortex}

We begin with the isentropic Euler vortex problem~\cite{vincent2011insights,shu2006essentially}. The domain is the periodic square
$\Omega=[-L/2,L/2]^2$, with $L = 20$. The exact solution is given by
\begin{equation}
\label{eq:num_vortex_exact}
\begin{aligned}
u(x,y,t)
&= u_\infty
-\frac{{\varepsilon_{\rm v}}}{2\pi}\bigl(y-v_\infty t\bigr)
\exp\!\left(
\frac12\Bigl[1-(x-u_\infty t)^2-(y-v_\infty t)^2\Bigr]
\right),\\
v(x,y,t)
&= v_\infty
+\frac{{\varepsilon_{\rm v}}}{2\pi}\bigl(x-u_\infty t\bigr)
\exp\!\left(
\frac12\Bigl[1-(x-u_\infty t)^2-(y-v_\infty t)^2\Bigr]
\right),\\
T(x,y,t)
&=
1-\frac{(\gamma-1){\varepsilon_{\rm v}^2}}{8\gamma\pi^2}
\exp\!\left(
1-(x-u_\infty t)^2-(y-v_\infty t)^2
\right),\\
\rho(x,y,t)
&= T(x,y,t)^{\frac{1}{\gamma-1}},
\qquad
p(x,y,t)
= T(x,y,t)^{\frac{\gamma}{\gamma-1}},
\end{aligned}
\end{equation}
where $\gamma=1.4$, {$\varepsilon_{\rm v}=5$ is the vortex
strength}, and $(u_\infty,v_\infty)=(1,1)$.

To assess convergence on nonconforming curvilinear meshes, we begin with a uniform $n\times n$ Cartesian partition with $n=10$ and apply a checkerboard refinement pattern. The resulting nonconforming mesh is then curved using the mapping
\begin{align*}
\tilde{x} = x + L {\varepsilon_{\rm g}} \cos\!\left(\frac{2\pi}{L} y\right), \qquad
\tilde{y} = y + L {\varepsilon_{\rm g}} \cos\!\left(\frac{2\pi}{L} x\right),
\end{align*}
with {$\varepsilon_{\rm g}=0.05$}. This defines the baseline curvilinear nonconforming mesh (Figure~\ref{fig:vortex_nc_mesh}, left). Successively refined meshes are obtained via uniform refinement of this baseline mesh (Figure~\ref{fig:vortex_nc_mesh}, right).
\begin{figure}[H]
  \centering
  \begin{subfigure}{0.42\textwidth}
    \centering
    \includegraphics[width=\linewidth]{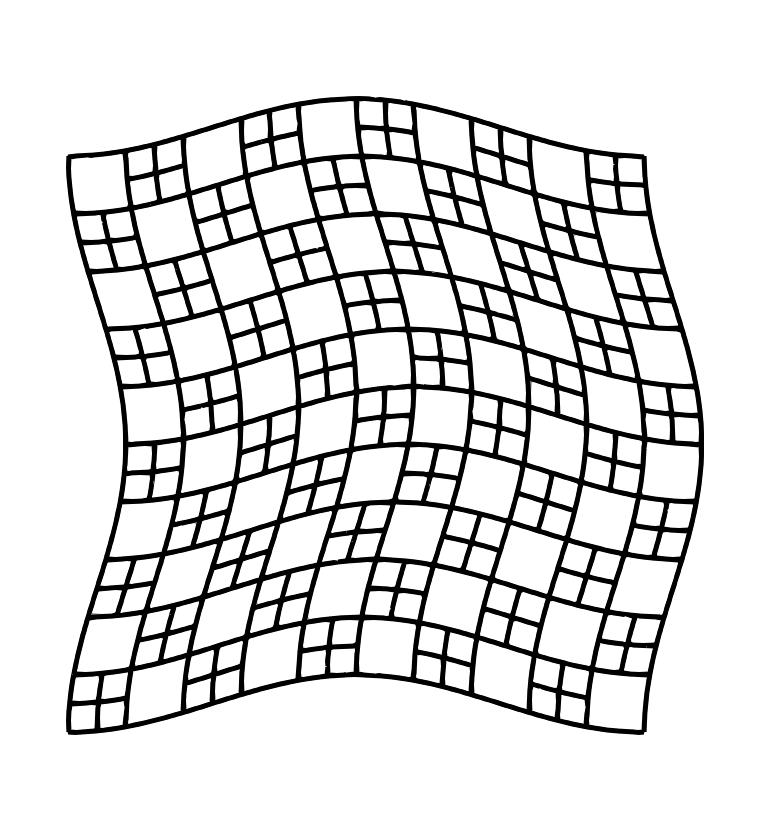}
    \caption{Initial curvilinear nonconforming mesh with $n=10$}
  \end{subfigure}\hfill
  \begin{subfigure}{0.42\textwidth}
    \centering
    \includegraphics[width=\linewidth]{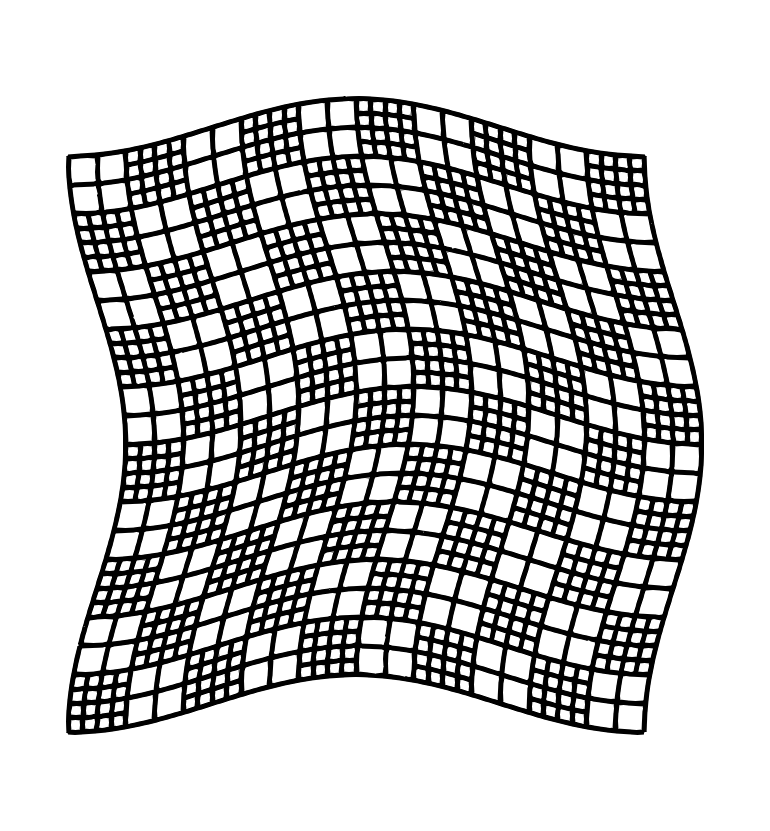}
    \caption{Uniformly refined nonconforming mesh}
  \end{subfigure}
  \caption{\textit{Isentropic Euler vortex problem}: initial checkerboard
 curvilinear nonconforming mesh and the mesh obtained after one additional uniform
  refinement.}
  \label{fig:vortex_nc_mesh}
\end{figure}

{In all simulations reported here, the OE procedure and the positivity limiter are turned off.
The solution is advanced to $t=0.2$, and errors are computed against the exact solution. Polynomial degrees $N=1,2,3,4$ are considered. We study spatial convergence using two numerical fluxes: the entropy-stable flux \eqref{eq:fine_flux}--\eqref{eq:coarse_flux} (denoted {\sf ES}) and the mortar-based flux \eqref{eq:fine_flux_mortar}--\eqref{eq:coarse_flux_mortar} (denoted {\sf mortar}). Time integration is performed using {SSP-RK3} with uniform time step $\Delta t = 0.2/(20\cdot 2^\ell)$, where $\ell$ denotes the refinement level ($\ell=0$ on the coarsest mesh). The squared-error integrals are evaluated using a Gauss quadrature rule of order $2N+3$.

Table~\ref{tab:vortex_conv} reports $L^2$ errors in density, momentum, and total energy. For $N=1$, both fluxes achieve rates near $1.5$. For $N=2,3,4$, the {\sf ES} flux exhibits reduced rates of approximately $1.0$--$1.5$, whereas the observed rates of the {\sf mortar} flux approach $N$.  We describe these as observed rates, rather than as a universal ``design order.''  In particular, they are below the classical $N+1$ rate commonly associated with degree-$N$ DG approximations of smooth solutions on conforming Cartesian meshes.  The present test combines curved element mappings, geometrically nonconforming interfaces, and LGL collocation, each of which can affect approximation and quadrature errors; related geometry- and interface-dependent accuracy behavior is discussed in \cite{Warburton2013,KoprivaEtAl2019,FriedrichEtAl2018,DelReyFernandezEtAl2020,ChanEtAl2021}, and a comparable near-$O(h^N)$ rate is visible in \cite[Fig.~7]{ChanEtAl2021}. {One contributing interface mechanism is that the SBP-compatible projection $P_{F_k\to C}=M_C^{-1}P_{C\to F_k}^TM_{F_k}$ uses the lumped LGL mass matrix. Because the LGL rule is not exact for degree-$2N$ products, the resulting composite fine-to-coarse projection does not reproduce the highest polynomial mode in general and has exact polynomial reproduction only through degree $N-1$. This is consistent with the near-$O(h^N)$ mortar behavior also reported in \cite[Fig.~7]{ChanEtAl2021}.} The lower rate of the {\sf ES} flux is additionally consistent with Remark~\ref{rem:sign_consistency}, which identifies both the nonlocal coarse-side dissipation and, for $N\ge2$, the loss of interpolation cancellation caused by the sign correction.}

{Although the {\sf ES} flux exhibits reduced convergence rates in this test, it performs comparably to the {\sf mortar} flux in the more realistic flow problems considered below. This behavior may be attributable, at least in part, to the particularly severe checkerboard curvilinear nonconforming mesh used in the present convergence study.}

\begin{table}[htbp]
\centering

\begin{subtable}{\textwidth}
\centering
\resizebox{\textwidth}{!}{%
\begin{tabular}{c c c c c c c c c c}
\hline
Flux & $\ell$  & $\|\rho-\rho_h\|_{L^2}$ & rate 
& $\|m_x-m_{x,h}\|_{L^2}$ & rate
& $\|m_y-m_{y,h}\|_{L^2}$ & rate&
$\|E-E_h\|_{L^2}$ & rate \\
\hline
\sf{ES} & 0 & $4.51\times 10^{-1}$ & -- & $5.93\times 10^{-1}$ & -- & $6.57\times 10^{-1}$ & -- & $2.03\times 10^{0}$ & -- \\
 & 1 & $1.31\times 10^{-1}$ & 1.78 & $2.25\times 10^{-1}$ & 1.40 & $2.35\times 10^{-1}$ & 1.48 & $6.32\times 10^{-1}$ & 1.69 \\
 & 2 & $4.48\times 10^{-2}$ & 1.55 & $7.97\times 10^{-2}$ & 1.50 & $7.87\times 10^{-2}$ & 1.58 & $2.20\times 10^{-1}$ & 1.52 \\
 & 3 & $1.39\times 10^{-2}$ & 1.69 & $2.46\times 10^{-2}$ & 1.70 & $2.43\times 10^{-2}$ & 1.69 & $6.78\times 10^{-2}$ & 1.70 \\
\hline
\sf{mortar} & 0 & $4.54\times 10^{-1}$ & -- & $5.61\times 10^{-1}$ & -- & $6.48\times 10^{-1}$ & -- & $2.04\times 10^{0}$ & -- \\
 & 1 & $1.31\times 10^{-1}$ & 1.80 & $2.18\times 10^{-1}$ & 1.36 & $2.29\times 10^{-1}$ & 1.50 & $6.23\times 10^{-1}$ & 1.71 \\
 & 2 & $4.73\times 10^{-2}$ & 1.46 & $8.24\times 10^{-2}$ & 1.40 & $8.07\times 10^{-2}$ & 1.50 & $2.36\times 10^{-1}$ & 1.40 \\
 & 3 & $1.62\times 10^{-2}$ & 1.55 & $2.78\times 10^{-2}$ & 1.57 & $2.82\times 10^{-2}$ & 1.52 & $8.22\times 10^{-2}$ & 1.52 \\
\hline
\end{tabular}}
\caption{$N=1$}
\end{subtable}

\vspace{0.6em}

\begin{subtable}{\textwidth}
\centering
\resizebox{\textwidth}{!}{%
\begin{tabular}{c c c c c c c c c c}
\hline
Flux & $\ell$  & $\|\rho-\rho_h\|_{L^2}$ & rate 
& $\|m_x-m_{x,h}\|_{L^2}$ & rate
& $\|m_y-m_{y,h}\|_{L^2}$ & rate&
$\|E-E_h\|_{L^2}$ & rate \\
\hline
\sf{ES}  & 0 & $1.00\times 10^{-1}$ & -- & $2.87\times 10^{-1}$ & -- & $2.74\times 10^{-1}$ & -- & $5.97\times 10^{-1}$ & -- \\
 & 1 & $5.18\times 10^{-2}$ & 0.96 & $1.02\times 10^{-1}$ & 1.49 & $9.92\times 10^{-2}$ & 1.47 & $2.66\times 10^{-1}$ & 1.16 \\
 & 2 & $1.93\times 10^{-2}$ & 1.43 & $3.57\times 10^{-2}$ & 1.51 & $3.50\times 10^{-2}$ & 1.50 & $9.47\times 10^{-2}$ & 1.49 \\
 & 3 & $7.95\times 10^{-3}$ & 1.28 & $1.28\times 10^{-2}$ & 1.48 & $1.30\times 10^{-2}$ & 1.42 & $3.50\times 10^{-2}$ & 1.44 \\
\hline
\sf{mortar} & 0 & $7.56\times 10^{-2}$ & -- & $2.03\times 10^{-1}$ & -- & $1.75\times 10^{-1}$ & -- & $4.57\times 10^{-1}$ & -- \\
 & 1 & $2.75\times 10^{-2}$ & 1.46 & $5.35\times 10^{-2}$ & 1.92 & $4.48\times 10^{-2}$ & 1.96 & $1.44\times 10^{-1}$ & 1.67 \\
 & 2 & $6.58\times 10^{-3}$ & 2.06 & $1.20\times 10^{-2}$ & 2.16 & $1.22\times 10^{-2}$ & 1.88 & $3.34\times 10^{-2}$ & 2.11 \\
 & 3 & $1.65\times 10^{-3}$ & 2.00 & $2.63\times 10^{-3}$ & 2.19 & $2.58\times 10^{-3}$ & 2.24 & $7.02\times 10^{-3}$ & 2.25 \\
\hline
\end{tabular}}
\caption{$N=2$}
\end{subtable}

\caption{\textit{Isentropic Euler vortex problem}: $L^2$ errors and
convergence rates on nonconforming curvilinear meshes for the
entropy-stable flux and the mortar-based flux. The four subtables
correspond to polynomial degrees $N=1,2,3,4$. {\sf ES} and {\sf mortar}
represent the entropy-stable and mortar-based fluxes, respectively.}
\label{tab:vortex_conv}
\end{table}

\begin{table}[htbp]
\ContinuedFloat
\centering

\begin{subtable}{\textwidth}
\centering
\resizebox{\textwidth}{!}{%
\begin{tabular}{c c c c c c c c c c}
\hline
Flux & $\ell$  & $\|\rho-\rho_h\|_{L^2}$ & rate 
& $\|m_x-m_{x,h}\|_{L^2}$ & rate
& $\|m_y-m_{y,h}\|_{L^2}$ & rate&
$\|E-E_h\|_{L^2}$ & rate \\
\hline
\sf{ES} & 0 & $1.03\times 10^{-1}$ & -- & $2.35\times 10^{-1}$ & -- & $2.22\times 10^{-1}$ & -- & $5.56\times 10^{-1}$ & -- \\
 & 1 & $4.89\times 10^{-2}$ & 1.08 & $9.55\times 10^{-2}$ & 1.30 & $9.40\times 10^{-2}$ & 1.24 & $2.46\times 10^{-1}$ & 1.18 \\
 & 2 & $2.30\times 10^{-2}$ & 1.09 & $3.93\times 10^{-2}$ & 1.28 & $3.90\times 10^{-2}$ & 1.27 & $1.08\times 10^{-1}$ & 1.19 \\
 & 3 & $1.01\times 10^{-2}$ & 1.18 & $1.54\times 10^{-2}$ & 1.35 & $1.56\times 10^{-2}$ & 1.32 & $4.35\times 10^{-2}$ & 1.31 \\
\hline
\sf{mortar} & 0 & $4.10\times 10^{-2}$ & -- & $7.50\times 10^{-2}$ & -- & $5.52\times 10^{-2}$ & -- & $2.07\times 10^{-1}$ & -- \\
 & 1 & $6.05\times 10^{-3}$ & 2.76 & $1.24\times 10^{-2}$ & 2.60 & $1.25\times 10^{-2}$ & 2.15 & $3.61\times 10^{-2}$ & 2.52 \\
 & 2 & $6.42\times 10^{-4}$ & 3.24 & $1.11\times 10^{-3}$ & 3.48 & $1.03\times 10^{-3}$ & 3.59 & $2.97\times 10^{-3}$ & 3.60 \\
 & 3 & $6.86\times 10^{-5}$ & 3.23 & $1.11\times 10^{-4}$ & 3.32 & $1.16\times 10^{-4}$ & 3.15 & $3.13\times 10^{-4}$ & 3.25 \\
\hline
\end{tabular}}
\caption{$N=3$}
\end{subtable}

\vspace{0.6em}

\begin{subtable}{\textwidth}
\centering
\resizebox{\textwidth}{!}{%
\begin{tabular}{c c c c c c c c c c}
\hline
Flux & $\ell$  & $\|\rho-\rho_h\|_{L^2}$ & rate
& $\|m_x-m_{x,h}\|_{L^2}$ & rate
& $\|m_y-m_{y,h}\|_{L^2}$ & rate&
$\|E-E_h\|_{L^2}$ & rate \\
\hline
\sf{ES} & 0 & $1.13\times 10^{-1}$ & -- & $2.17\times 10^{-1}$ & -- & $2.11\times 10^{-1}$ & -- & $5.56\times 10^{-1}$ & -- \\
 & 1 & $5.32\times 10^{-2}$ & 1.08 & $9.89\times 10^{-2}$ & 1.13 & $9.65\times 10^{-2}$ & 1.13 & $2.64\times 10^{-1}$ & 1.07 \\
 & 2 & $2.53\times 10^{-2}$ & 1.07 & $4.02\times 10^{-2}$ & 1.30 & $4.05\times 10^{-2}$ & 1.25 & $1.14\times 10^{-1}$ & 1.21 \\
 & 3 & $1.15\times 10^{-2}$ & 1.13 & $1.70\times 10^{-2}$ & 1.24 & $1.71\times 10^{-2}$ & 1.25 & $4.91\times 10^{-2}$ & 1.22 \\
\hline
\sf{mortar} & 0 & $1.40\times 10^{-2}$ & -- & $2.65\times 10^{-2}$ & -- & $2.74\times 10^{-2}$ & -- & $7.86\times 10^{-2}$ & -- \\
 & 1 & $1.07\times 10^{-3}$ & 3.71 & $1.96\times 10^{-3}$ & 3.76 & $1.61\times 10^{-3}$ & 4.09 & $5.29\times 10^{-3}$ & 3.89 \\
 & 2 & $6.73\times 10^{-5}$ & 3.99 & $1.07\times 10^{-4}$ & 4.19 & $9.53\times 10^{-5}$ & 4.08 & $2.87\times 10^{-4}$ & 4.20 \\
 & 3 & $4.01\times 10^{-6}$ & 4.07 & $6.01\times 10^{-6}$ & 4.16 & $5.83\times 10^{-6}$ & 4.03 & $1.65\times 10^{-5}$ & 4.12 \\
\hline
\end{tabular}}
\caption{$N=4$}
\end{subtable}

\caption[]{\textit{Isentropic Euler vortex problem}: $L^2$ errors and
convergence rates on nonconforming curvilinear meshes (continued).}
\end{table}

{
\subsection{Entropy test}
\label{subsec:entropy_test}

{The following two tests examine the entropy behavior of the
classical mortar-based flux. To the best of the authors' knowledge, an
entropy-stability proof for this coupling has not been established, and both
tests presented below exhibit entropy violations produced by the mortar
coupling. Thus, the mortar-based flux is not an entropy-stable flux in
general.}

\subsubsection{Local nonconforming-interface entropy production}
\label{subsec:local_entropy_test}

We first isolate the algebraic effect of the nonconforming interface flux
from time integration and volume discretization. We consider one Cartesian
$1{:}2$ interface with unit normal $\bm n=(1,0)^T$, as illustrated in
Figure~\ref{fig:local_entropy_mesh}, and evaluate the assembled interface
entropy production.

\begin{figure}[H]
  \centering
  \begin{tikzpicture}[
      x=1.45cm,
      y=1.25cm,
      every node/.style={text=black,font=\small}]
    \draw[black,thick] (0,0) rectangle (2,2);
    \draw[black,thick] (0,1) -- (2,1);
    \draw[black,thick] (2,0) rectangle (4.5,2);

    \draw[blue!65!black,very thick] (2,0) -- (2,2);
    \fill[blue!65!black] (2,1) circle (2.3pt);
    \draw[black,->,>=stealth,thick] (1.55,2.18) -- (2.55,2.18);

    \node at (1,0.5) {$\Omega_{F_1}$};
    \node at (1,1.5) {$\Omega_{F_2}$};
    \node at (3.25,0.62) {$\Omega_C$};
    \node[above] at (2.05,2.18) {$\bm n$};
    \node[anchor=west] (hnlabel) at (3.0,1.42) {hanging node};
    \draw[black,->,>=stealth] (hnlabel.west) -- (2.08,1.04);
  \end{tikzpicture}
  \caption{Cartesian $1{:}2$ nonconforming interface used in the local
  entropy-production test. Two fine elements $\Omega_{F_1}$ and
  $\Omega_{F_2}$ share the coarse edge of $\Omega_C$; the midpoint of the
  coarse edge is a hanging node, and $\bm n$ points from the fine side to the
  coarse side.}
  \label{fig:local_entropy_mesh}
\end{figure}

{Define}
\begin{align}
\mathcal D_{\rm M}
&=\frac12\sum_{k=1}^2\sum_{i=0}^N w_i
\left[
(\widetilde{\bm V}^{C,k}_i-\bm V^{F_k}_i)^T
\bm F^*(\bm U^{F_k}_i,\widetilde{\bm U}^{C,k}_i;\bm n)
-(\widetilde\psi^{C,k}_i-\psi^{F_k}_i)
\right],\label{eq:local_mortar_entropy_production}\\
\mathcal D_{\rm ES}
&=\frac12\sum_{k=1}^2\sum_{i,j=0}^N w_i
(P_{C\to F_k})_{ij}
\left[
(\bm V^C_j-\bm V^{F_k}_i)^T
\widehat{\bm F}_{ij}^{F_k,C}
-(\psi^C_j-\psi^{F_k}_i)
\right],\label{eq:local_es_entropy_production}
\end{align}
For the mortar contribution, $\widetilde{\bm U}^{C,k}_i$ is the coarse
conservative state interpolated to the $i$th node of child $F_k$, as defined
in \eqref{eq:coarse_interp}, while
\begin{equation}
\widetilde{\bm V}^{C,k}_i
=\sum_{j=0}^N(P_{C\to F_k})_{ij}\bm V_j^C,
\qquad
\widetilde{\psi}^{C,k}_i
=\sum_{j=0}^N(P_{C\to F_k})_{ij}\psi_j^C
\end{equation}
are the corresponding interpolations of the coarse-side nodal entropy
variables and normal entropy potential. In general,
$\widetilde{\bm V}^{C,k}_i\ne
\bm V(\widetilde{\bm U}^{C,k}_i)$ because the entropy-variable map is
nonlinear. For the ES contribution,
$\widehat{\bm F}_{ij}^{F_k,C}$ is the sign-corrected pairwise flux defined in
\eqref{lxf}. The factor $1/2$ is the child-to-parent edge-length
ratio. With the entropy convention of Section~\ref{subsec:euler_entropy},
entropy stability requires $\mathcal D\leq0$.

We use $\gamma=1.4$, polynomial degrees $N=1,2,3$, and a deterministic
$401\times401$ state lattice. Thus each degree has 160801 samples. Density and pressure range over
$[e^{-1.5},e^{1.5}]$, and both velocity components range over $[-4,4]$.
{For completeness, the lattice is defined as follows. Let
$G=401$, $L=1.5$, and $V=4$. For a lattice pair
$(r,c)\in\{0,\ldots,G-1\}^2$, state number $s$, and primitive component
$\ell\in\{0,1,2,3\}$, set $q=4s+\ell$ and
\begin{equation}
I_{s\ell}(r,c)
=\bigl[r(2q+1)+c(q^2+3q+1)+17q+11\bigr]\bmod G,
\qquad
z_{s\ell}=\frac{I_{s\ell}}{G-1}.
\label{eq:entropy_lattice}
\end{equation}
The corresponding primitive state is
\begin{equation}
(\rho_s,u_s,v_s,p_s)
=\left(e^{-L+2Lz_{s0}},-V+2Vz_{s1},-V+2Vz_{s2},
e^{-L+2Lz_{s3}}\right),
\label{eq:entropy_lattice_state}
\end{equation}
and is converted to conservative variables in the standard way. States
$s=0,\ldots,N$ are assigned to the coarse nodes, followed in order by the
$N+1$ nodes of $F_1$ and then those of $F_2$. The complete $N=1$
counterexample at $(r,c)=(247,100)$ is listed in
\ref{app:local_entropy_repro}.}
Only samples for which every nodal and interpolated mortar state has positive
density and pressure are retained. A production greater than $10^{-11}$ is
counted as a violation. {As an exception to the combined
default in this section, this diagnostic uses only the guaranteed
Guermond--Popov bound $\alpha_{\rm GP}$.} Both fluxes use the same
arithmetic-average central part.

\begin{table}[H]
\centering
\resizebox{\textwidth}{!}{%
\begin{tabular}{c|rr|r|rr}
\hline
$N$ & $\min P_{C\to F}$ & negative entries & admissible samples &
mortar-GP & ES-GP\\
\hline
1 & $0$ & 0 & 160801 & 1002 & 0\\
2 & $-1.2500\times10^{-1}$ & 2 & 112400 & 358 & 0\\
3 & $-1.8402\times10^{-1}$ & 10 & 57822 & 50 & 0\\
\hline
\end{tabular}}
\caption{\textit{Entropy Test}: positive nonconforming-interface entropy
productions obtained with the Guermond--Popov wave-speed estimate
$\alpha_{\rm GP}$.}
\label{tab:local_entropy_violations}
\end{table}

Table~\ref{tab:local_entropy_violations} shows that the mortar construction
admits positive interface entropy production, whereas the proposed ES flux
has no violation in any retained sample. The strongest $N=1$ counterexample
occurs at lattice index $(247,100)$; its mortar and ES productions are,
respectively, $5.79680\times10^2$ and $-1.77350\times10^2$.

The $N=1$ result is especially informative. Both interpolation matrices are
entrywise nonnegative and all 160801 interpolated states are admissible, yet
the mortar flux violates the entropy inequality in 1002 Guermond--Popov-speed
tests. Thus these counterexamples do not arise from negative interpolation
coefficients, loss of mortar-state admissibility, or an underestimated
Riemann-fan speed. They instead expose the loss
of the pairwise Tadmor structure caused by interpolating conservative states
before evaluating a nonlinear two-point flux. This local result does not
assert that every full mortar discretization increases total entropy; it
demonstrates that the classical mortar coupling does not satisfy a general
interface entropy inequality of the form proved for the proposed ES flux.

\subsubsection{Entropy history for the isentropic vortex}
\label{subsec:vortex_entropy_history}

We next examine whether the local entropy violations identified above remain
visible in a complete PDE simulation. We use the same periodic isentropic
vortex, exact solution, and curvilinear checkerboard mesh construction as in
Section~\ref{subsec:num_vortex}. For this test, the background mesh contains
$80\times80$ Cartesian elements. Refining every other background element once
in the checkerboard pattern produces a nonconforming mesh with 16000 elements,
which is subsequently curved with {$\varepsilon_{\rm g}=0.05$}. We consider polynomial
degrees $N=1$ and $N=2$.

The solution is advanced to $t=1$ by {SSP-RK3} using the fixed time step
$\Delta t=5\times10^{-4}$. The OE procedure and positivity limiter are both
disabled so that the measured entropy change is not affected by either
postprocessing operation. {The volume discretization uses the same
Tadmor entropy-conservative two-point flux in both runs. At element
interfaces, both the {\sf mortar} and {\sf ES} couplings instead use the same
arithmetic average of the physical fluxes as their central part, together
with scalar LF dissipation.} {Because this experiment is designed specifically to
test the entropy behavior and does not employ the positivity-preserving
limiter, the dissipation parameter is chosen as the guaranteed
Guermond--Popov bound $\alpha_{\rm GP}$ alone rather than the combined
default \eqref{alpha}.}
Thus the two computations differ only in the treatment of the nonconforming
interfaces. At every time step, we evaluate the integrated numerical entropy
\begin{equation}
\mathcal E_h(t)
=\sum_e\int_{\Omega_e}
\eta\bigl(\bm U_h(\bm x,t)\bigr)\,\mathrm d\bm x,
\qquad
\eta(\bm U)=-\frac{\rho s}{\gamma-1}.
\label{eq:vortex_integrated_entropy}
\end{equation}

The entropy histories are shown in Figure~\ref{fig:vortex_entropy_history}.
For both $N=1$ and $N=2$, the mortar solution exhibits positive increments
of $\mathcal E_h$ and a positive net entropy change over the simulation. In
contrast, the entropy produced by the proposed {\sf ES} coupling is
non-positive at every time step to the prescribed numerical tolerance. This
full isentropic-vortex calculation therefore confirms that the local entropy
violations of the mortar coupling can persist in a practical smooth-flow
simulation, while the proposed flux maintains the expected entropy-stable
behavior.

{For $N=1$, the ES entropy decreases by approximately
$2.44\times10^{-1}$ over $0\le t\le1$, whereas the mortar entropy has the much
smaller positive net change $1.73\times10^{-3}$. This marked difference is
consistent with Remark~\ref{rem:sign_consistency}: although the $N=1$
fine-side interpolation coefficients are nonnegative, the coarse-side ES
dissipation still couples states at separated physical points and is therefore
substantially more dissipative.}

\begin{figure}[H]
  \centering
  \begin{subfigure}{0.48\textwidth}
    \centering
    \includegraphics[width=\linewidth]{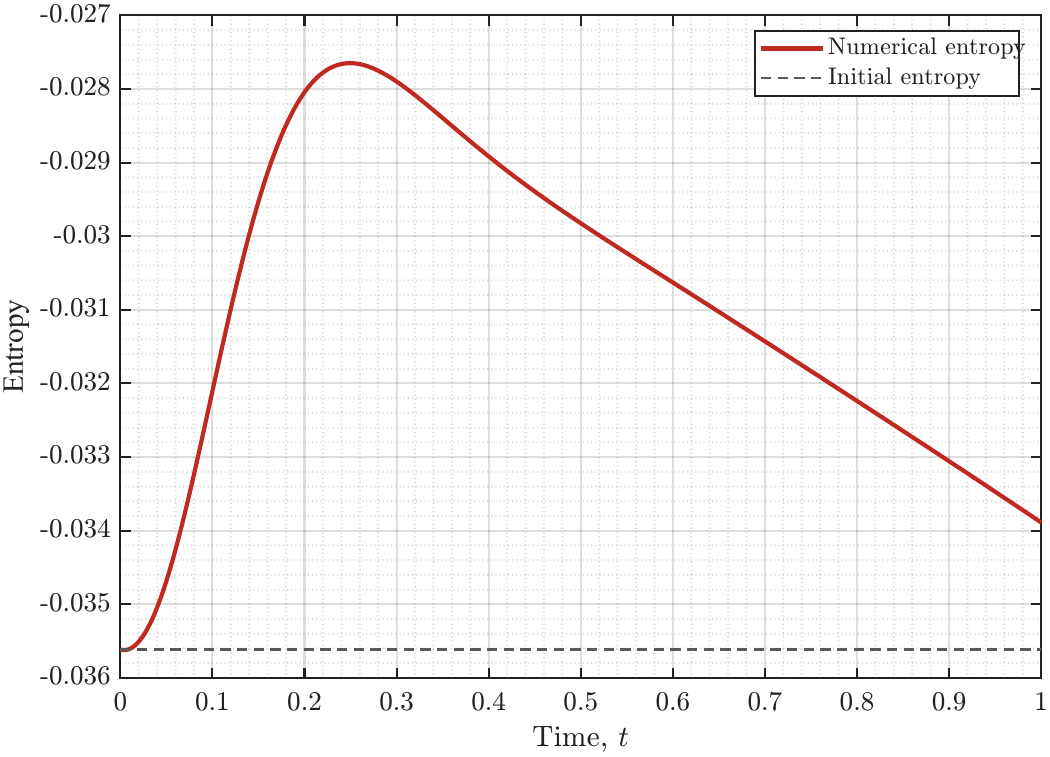}
    \caption{{\sf mortar}, $N=1$}
  \end{subfigure}\hfill
  \begin{subfigure}{0.48\textwidth}
    \centering
    \includegraphics[width=\linewidth]{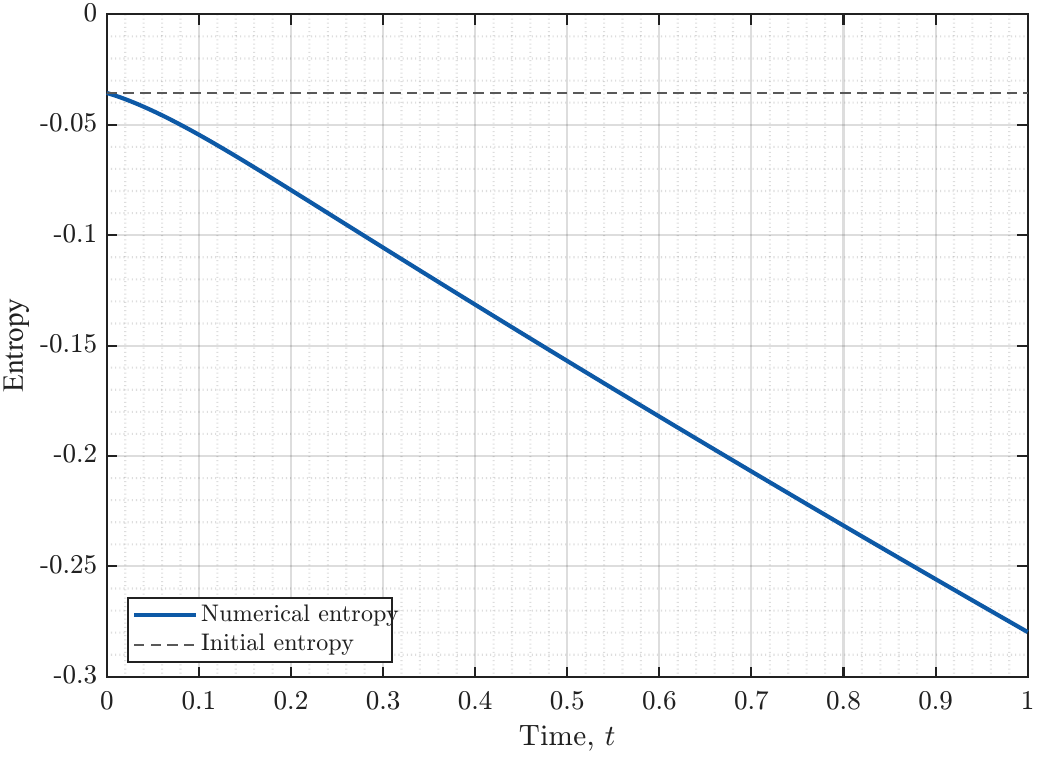}
    \caption{{\sf ES}, $N=1$}
  \end{subfigure}

  \medskip

  \begin{subfigure}{0.48\textwidth}
    \centering
    \includegraphics[width=\linewidth]{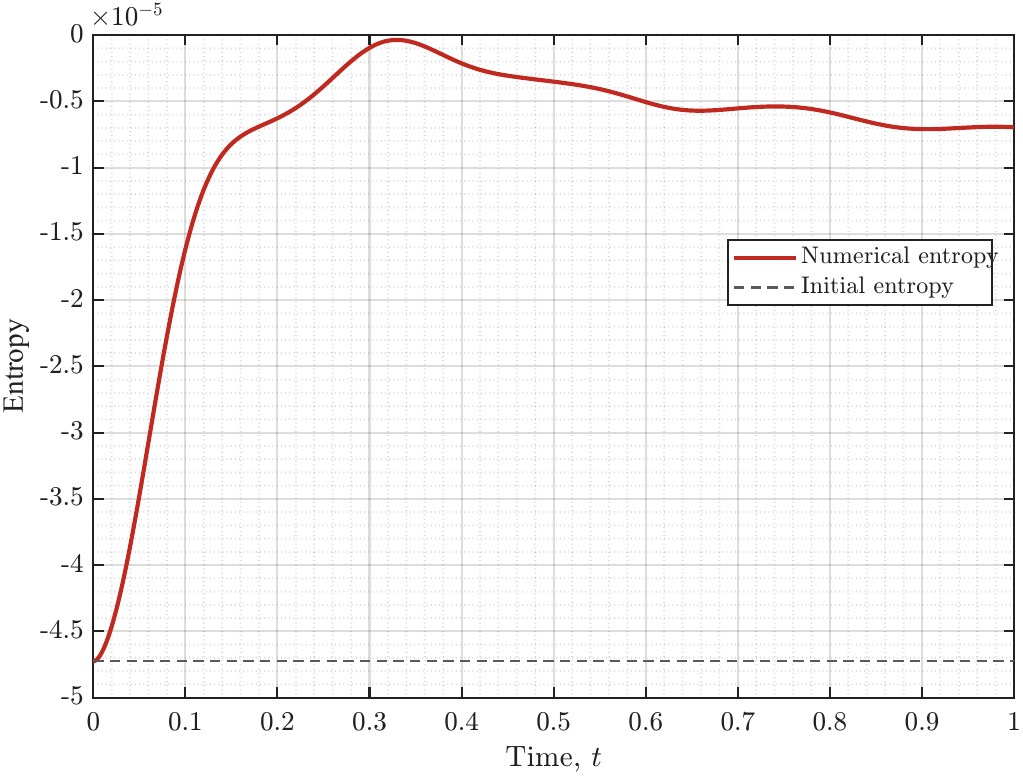}
    \caption{{\sf mortar}, $N=2$}
  \end{subfigure}\hfill
  \begin{subfigure}{0.48\textwidth}
    \centering
    \includegraphics[width=\linewidth]{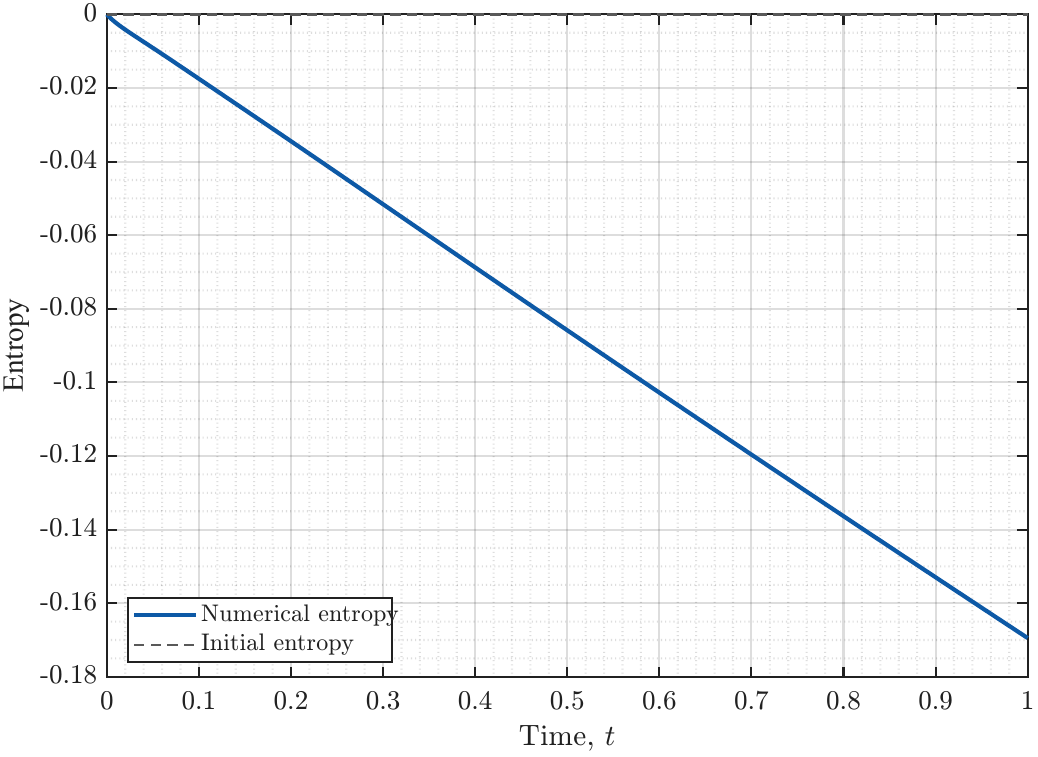}
    \caption{{\sf ES}, $N=2$}
  \end{subfigure}
  \caption{\textit{Entropy Test}: history of the integrated numerical entropy
  for the isentropic Euler vortex using the mortar and proposed entropy-stable
  nonconforming-interface fluxes with $N=1$ and $N=2$. The LF dissipation
  uses the Guermond--Popov wave-speed estimate $\alpha_{\rm GP}$.}
  \label{fig:vortex_entropy_history}
\end{figure}
}

\subsection{Double Mach reflection}
\label{subsec:num_dmr}

The double Mach reflection problem is a standard benchmark for
compressible flow solvers in the presence of strong shocks, multiple
reflections, and intricate small-scale structures~\cite{woodward1984numerical,shu2003high,zhang2017positivity}. It is particularly
useful for assessing the robustness of a numerical method as well as
its ability to resolve delicate flow features near slip lines and in
the downstream region. We consider the rectangular domain
$\Omega=[0,4]\times[0,1]$ and initialize a Mach-$10$ shock that makes
an angle of $60^\circ$ with the horizontal direction and initially
intersects the bottom boundary at $(x,y)=(1/6,0)$. The initial states
are given by
\begin{equation}
\label{eq:dmr_ic_main}
(\rho,u,v,p)=
\begin{cases}
(8,\;8.25\cos(\pi/6),\;-8.25\sin(\pi/6),\;116.5),
& x<\dfrac{1}{6}+\dfrac{y}{\sqrt{3}},\\[1ex]
(1.4,\;0,\;0,\;1),
& x>\dfrac{1}{6}+\dfrac{y}{\sqrt{3}}.
\end{cases}
\end{equation}
At the left boundary we impose the post-shock inflow state, while an
outflow condition is used on the right boundary. On the top boundary, a
time-dependent moving-shock boundary condition is prescribed. On the
bottom boundary, the post-shock state is imposed for $0<x<1/6$, and a
reflective wall boundary condition is applied elsewhere.

{To compare the two nonconforming interface treatments systematically,
we perform four AMR computations at $t=0.2$: the \textsf{ES} and mortar
fluxes are each tested with polynomial degrees $N=1$ and $N=2$.  All four
runs start from the same $64\times16$ Cartesian mesh (in the $x$- and
$y$-directions, respectively), for which the base mesh size is $h=1/16$.
Before time integration, the initial discontinuity is refined with threshold
$0.03$ through at most four levels.  During time integration, AMR is applied
every 10 time steps.  The refinement thresholds are $0.2$ for $N=1$ and
$0.05$ for $N=2$; {both degrees set
$C_{\mathrm{crs}}=C_{\mathrm{ref}}$ (equivalently,
$C_{\mathrm{crs}}/C_{\mathrm{ref}}=1$)} and use a maximum
of five refinement levels, so the smallest admissible cell size is $h=1/512$.
The degree-dependent thresholds account for the different scaling of the
indicator. {All four computations use the combined wave-speed
parameter \eqref{alpha},
$\alpha=\max\{\alpha_{\rm char},\alpha_{\rm GP}\}$.} The CFL number is $0.5$, and all other spatial, temporal,
positivity-preserving, and oscillation-eliminating parameters are held fixed.
{In particular, all four runs use
$s=0.2$ and $C_{\mathrm{oe}}=0.1$.}
Thus, within each polynomial degree, the only changed parameter is the
nonconforming interface flux.}

{Figure~\ref{fig:dmr_density} compares the four density fields on
$[0,3]\times[0,1]$ (left column) together with their corresponding close-ups
on $[2.2,2.8]\times[0,0.5]$ (right column), using the common range
$1\leq\rho\leq24$ and 30 uniformly spaced contour levels.  All four
computations remain robust and
capture the incident and reflected shocks, the Mach stem, and the principal
slip line.  At a fixed polynomial degree, the \textsf{ES} and mortar results
have closely aligned large-scale shock locations.  This verifies that the
positivity-preserving and entropy-stable nonconforming coupling remains
effective in a strong-shock calculation with repeated refinement,
coarsening, and parallel data redistribution.}

\begin{figure}[p]
  \centering
  \captionsetup[subfigure]{font=small,skip=1pt}
  \begin{subfigure}[b]{0.65\textwidth}
    \centering
    \includegraphics[height=0.165\textheight,keepaspectratio]{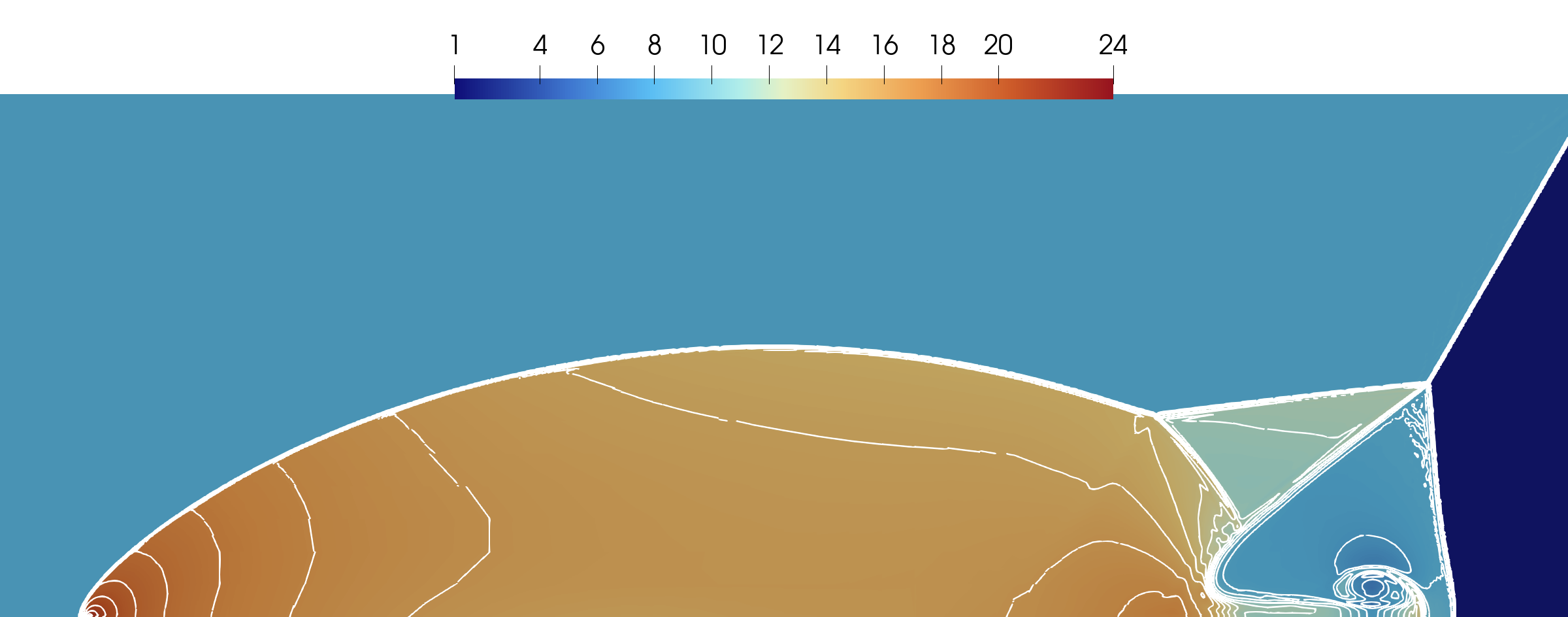}
    \caption{{$N=1$, \textsf{ES} flux: density profile}}
  \end{subfigure}
  \hfill
  \begin{subfigure}[b]{0.32\textwidth}
    \centering
    \includegraphics[height=0.165\textheight,keepaspectratio]{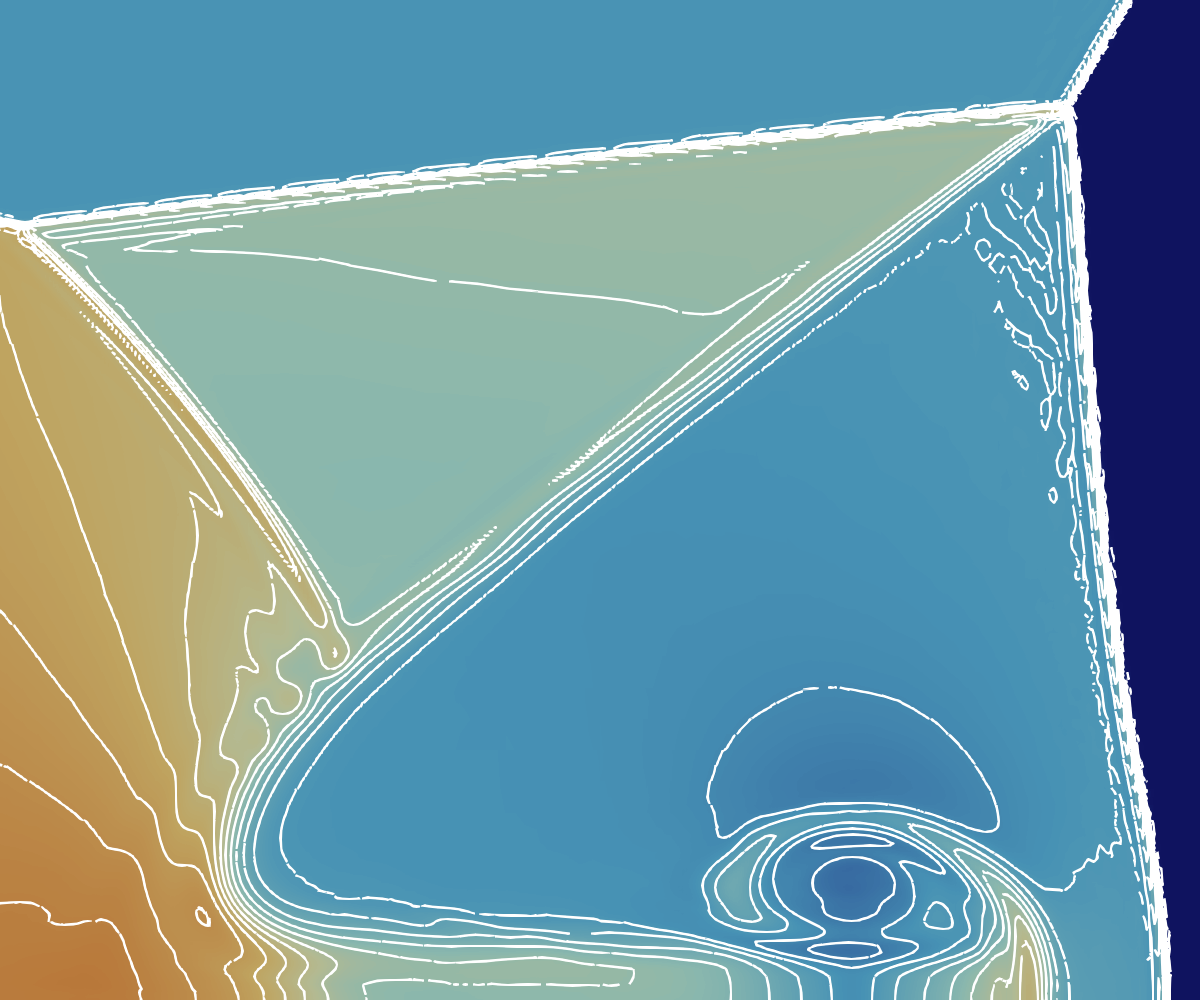}
    \caption{{close-up}}
  \end{subfigure}
  \par\vspace{0.1cm}
  \begin{subfigure}[b]{0.65\textwidth}
    \centering
    \includegraphics[height=0.165\textheight,keepaspectratio]{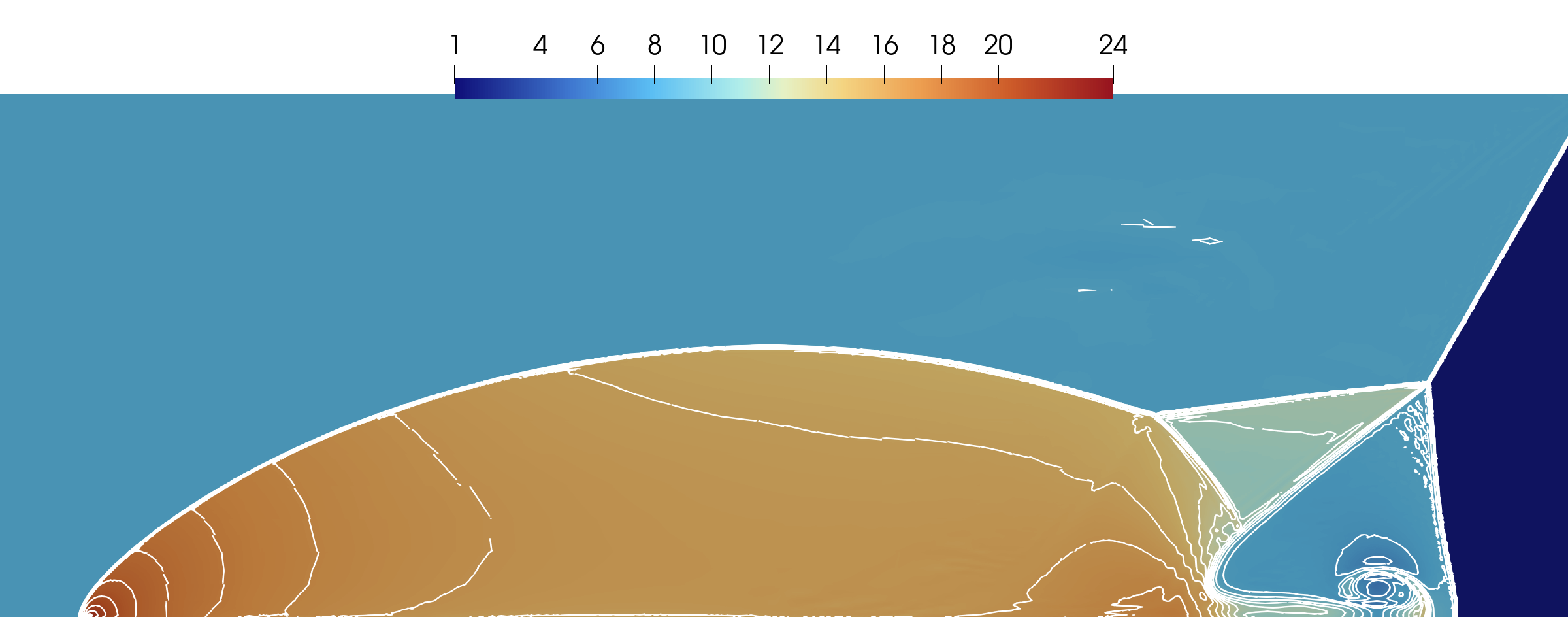}
    \caption{{$N=1$, mortar flux: density profile}}
  \end{subfigure}
  \hfill
  \begin{subfigure}[b]{0.32\textwidth}
    \centering
    \includegraphics[height=0.165\textheight,keepaspectratio]{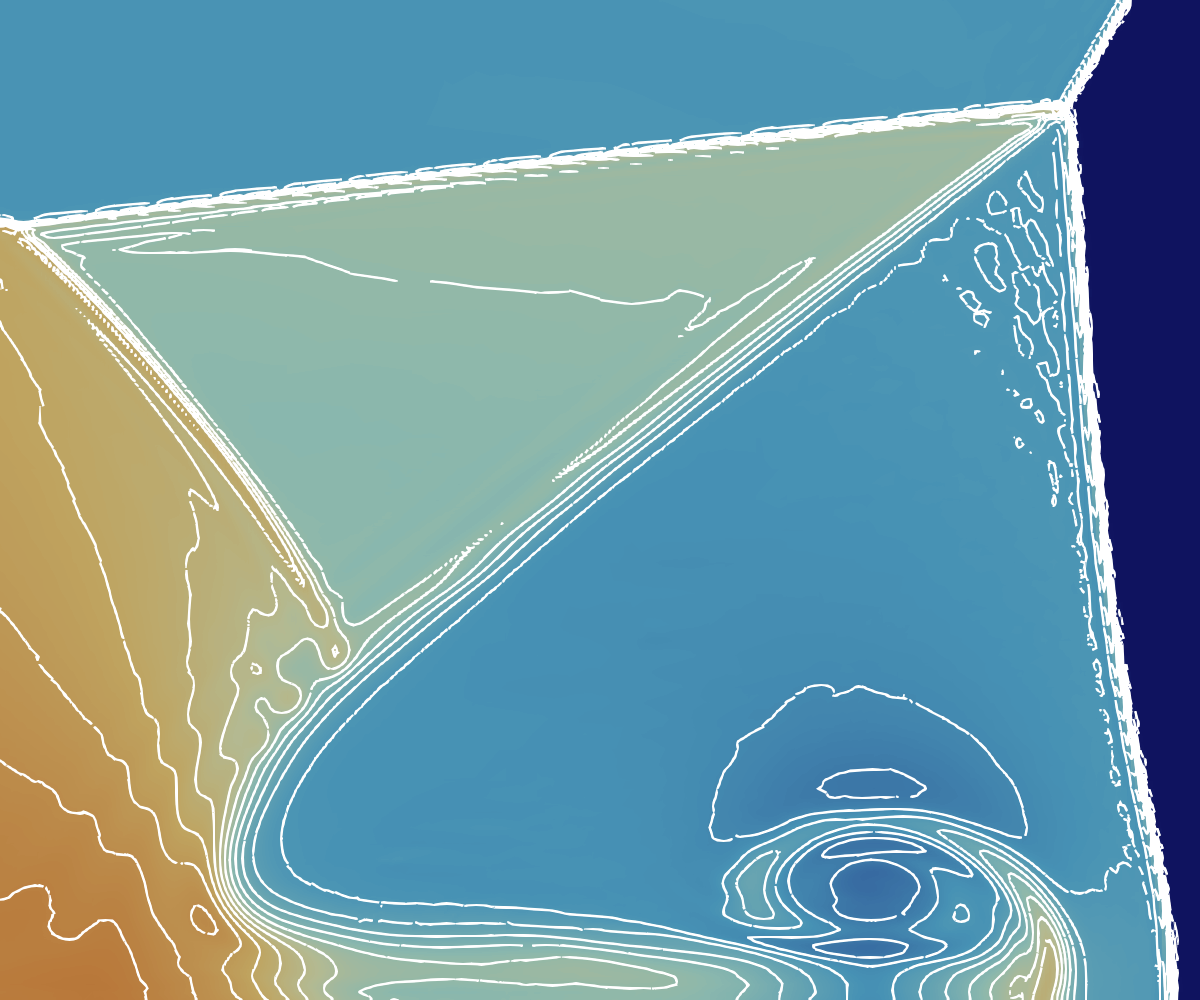}
    \caption{{close-up}}
  \end{subfigure}
  \par\vspace{0.1cm}
  \begin{subfigure}[b]{0.65\textwidth}
    \centering
    \includegraphics[height=0.165\textheight,keepaspectratio]{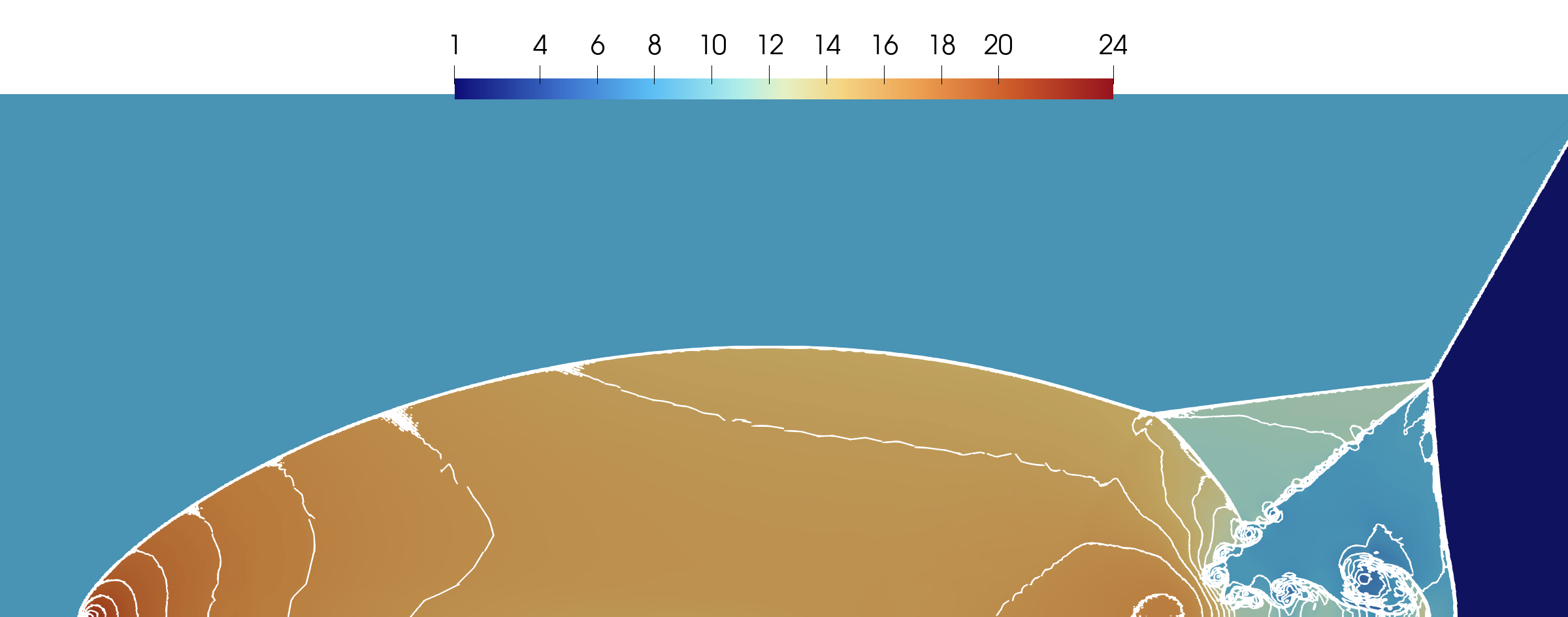}
    \caption{{$N=2$, \textsf{ES} flux: density profile}}
  \end{subfigure}
  \hfill
  \begin{subfigure}[b]{0.32\textwidth}
    \centering
    \includegraphics[height=0.165\textheight,keepaspectratio]{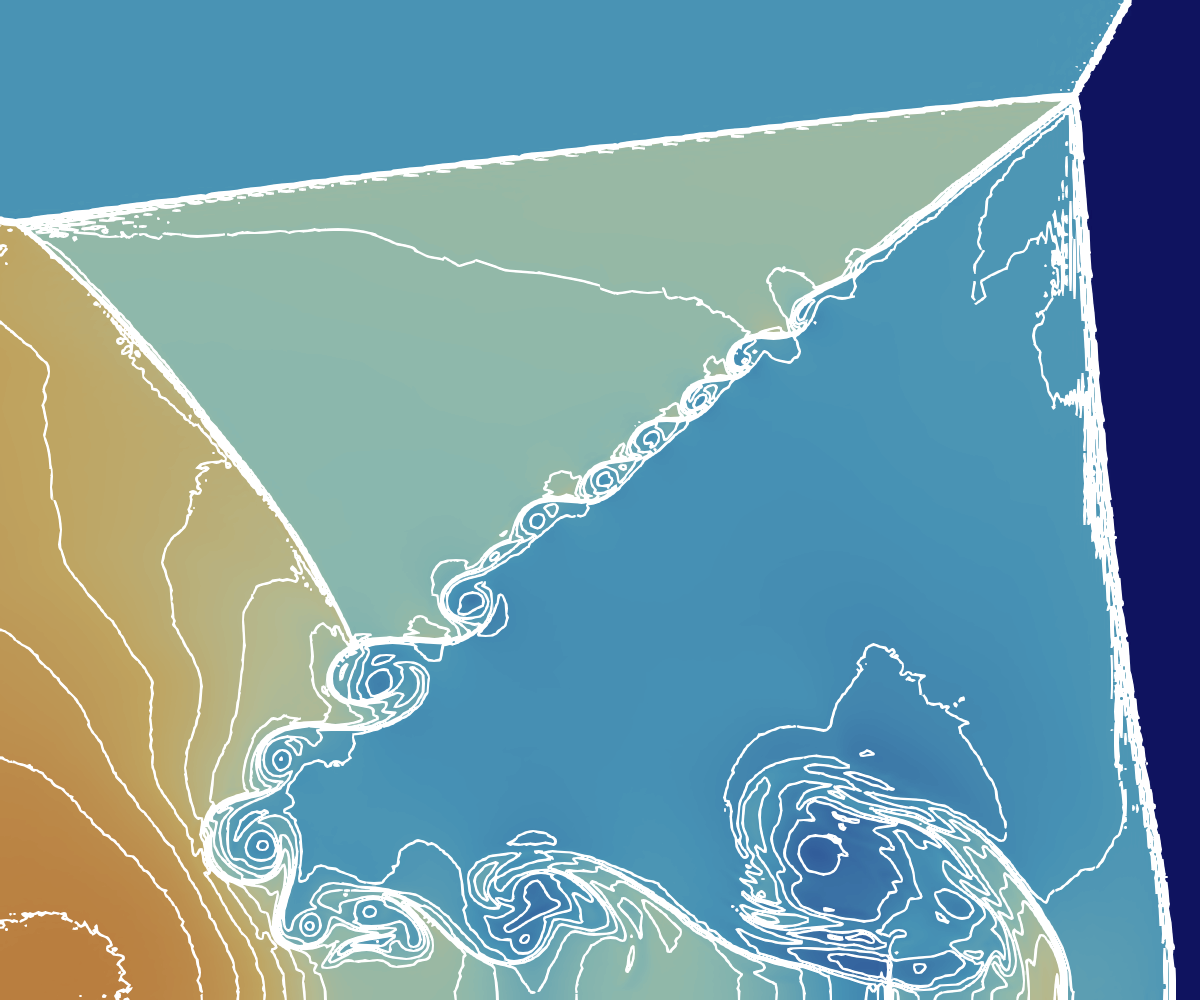}
    \caption{{close-up}}
  \end{subfigure}
  \par\vspace{0.1cm}
  \begin{subfigure}[b]{0.65\textwidth}
    \centering
    \includegraphics[height=0.165\textheight,keepaspectratio]{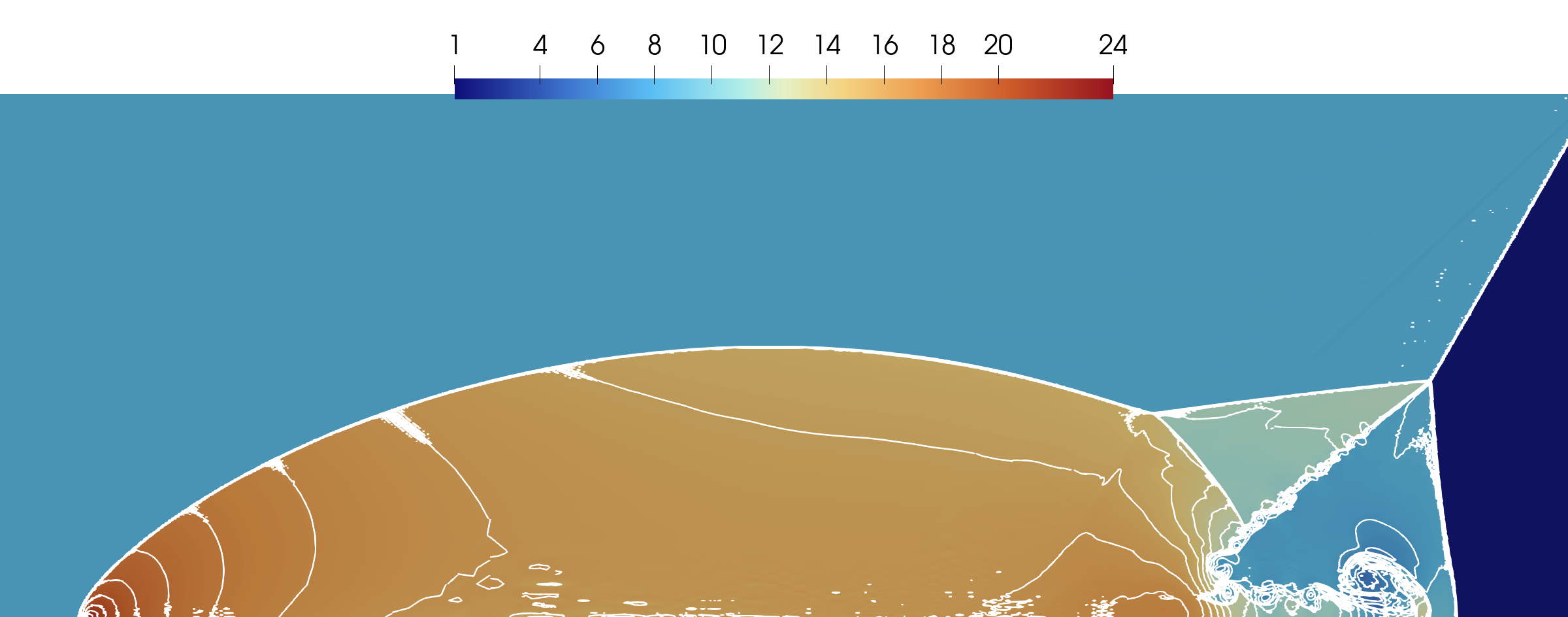}
    \caption{{$N=2$, mortar flux: density profile}}
  \end{subfigure}
  \hfill
  \begin{subfigure}[b]{0.32\textwidth}
    \centering
    \includegraphics[height=0.165\textheight,keepaspectratio]{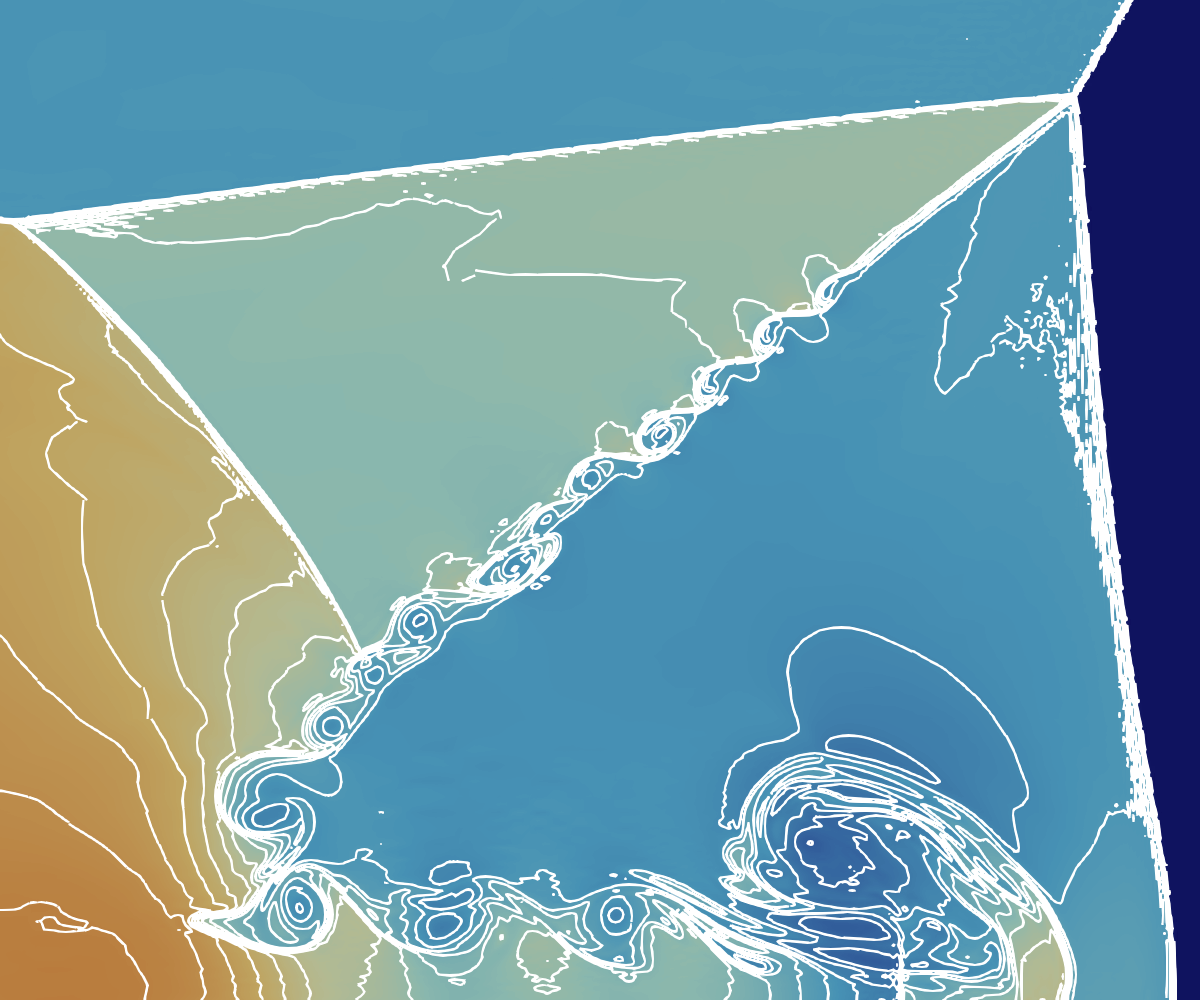}
    \caption{{close-up}}
  \end{subfigure}
  \caption{{\textit{Double Mach reflection}: density profiles and
  corresponding close-ups at $t=0.2$.  In each row, the left panel shows
  $[0,3]\times[0,1]$ and the right panel shows
  $[2.2,2.8]\times[0,0.5]$.  The rows correspond, from top to bottom, to
  $N=1$ \textsf{ES}, $N=1$ mortar, $N=2$ \textsf{ES}, and $N=2$ mortar.
  All panels use the common color range $1\leq\rho\leq24$, and the white
  curves denote the same 30 uniformly spaced density contours.}}
  \label{fig:dmr_density}
\end{figure}

{The close-ups in the right column of
Figure~\ref{fig:dmr_density} provide a more
discriminating comparison.  The two $N=1$ solutions resolve similar
large-scale features but contain relatively little roll-up along the slip
lines.  Increasing the polynomial degree to $N=2$ produces substantially
richer Kelvin--Helmholtz structures for both interface treatments.  In
particular, the $N=2$ \textsf{ES} result is much more detailed than its
$N=1$ counterpart and is qualitatively comparable to the $N=2$ mortar
result, even though the smooth-vortex experiment shows a reduced asymptotic
convergence rate for the \textsf{ES} interface coupling.  There is no
contradiction: increasing $N$ still improves the element-interior
representation and the traces supplied to the interface flux, whereas the
observed convergence rate characterizes asymptotic mesh refinement for a
smooth problem.}

{Figure~\ref{fig:dmr_mesh} shows that all four calculations place
the finest cells along shocks, slip lines, and downstream vortical
structures.  Because the indicator is evaluated from the evolving numerical
solution, the meshes within each same-degree pair need not coincide even
though their AMR parameters are identical.  The maximum numbers of elements
(solution unknowns) were {54751 (876016), 74221 (1187536),
62950 (2266200), and 64963 (2338668)} for
$N=1$ \textsf{ES}, $N=1$ mortar, $N=2$ \textsf{ES}, and $N=2$ mortar,
respectively.  These data also show that the two $N=2$ comparisons have
similar peak algebraic sizes.}

{Figure~\ref{fig:dmr_reference} reproduces the conforming-mesh reference
solution from \cite{YangFu26}, in which the secondary slip-line structure is
clearly visible.  Its comparison with the AMR solutions in
Figure~\ref{fig:dmr_density} exposes a limitation of the present adaptive
strategy.  Although the indicator successfully refines the wake region and
enables the AMR solutions to resolve markedly finer downstream structures, it
does not consistently identify the secondary slip-line region.  The mesh
therefore remains comparatively coarse in that region, and the secondary
slip-line structure is weaker in the AMR results than in the conforming
reference solution.  A more selective indicator capable of detecting this
feature without sacrificing the low cost shared with the OE procedure is
beyond the scope of this work and will be investigated in future work.}

\begin{figure}[htbp]
  \centering
  \includegraphics[width=0.96\textwidth]{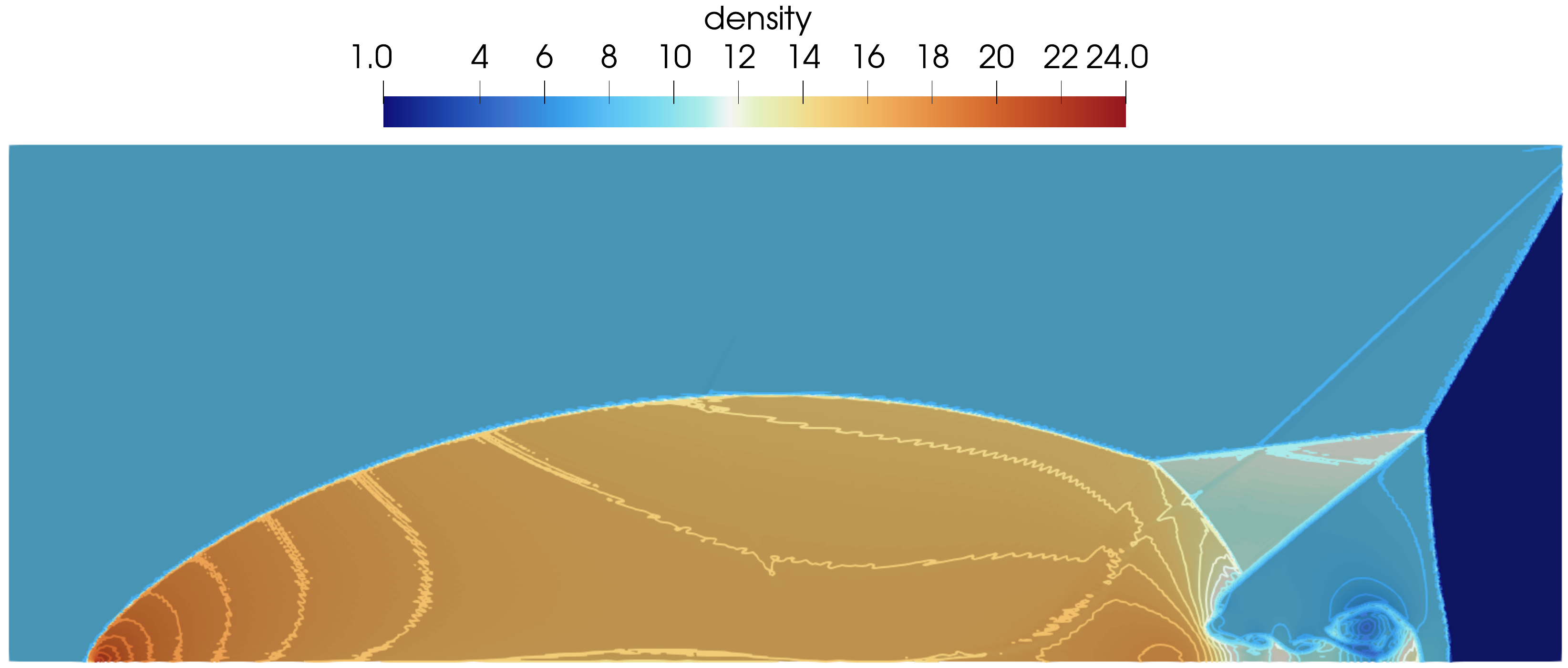}
  \caption{{\textit{Double Mach reflection}: conforming-mesh reference
  solution from \cite{YangFu26} at $t=0.2$, obtained with polynomial degree
  $N=2$ on a uniform $960\times240$ mesh in the $x$- and $y$-directions,
  respectively.  The density profile is shown on
  $[0,3]\times[0,1]$ with 30 uniformly spaced contour levels.  In contrast
  to the AMR results in Figure~\ref{fig:dmr_density}, the secondary slip-line
  structure is clearly visible.}}
  \label{fig:dmr_reference}
\end{figure}

\begin{figure}[htbp]
  \centering
  \begin{subfigure}{0.49\textwidth}
    \centering
    \includegraphics[width=\linewidth]{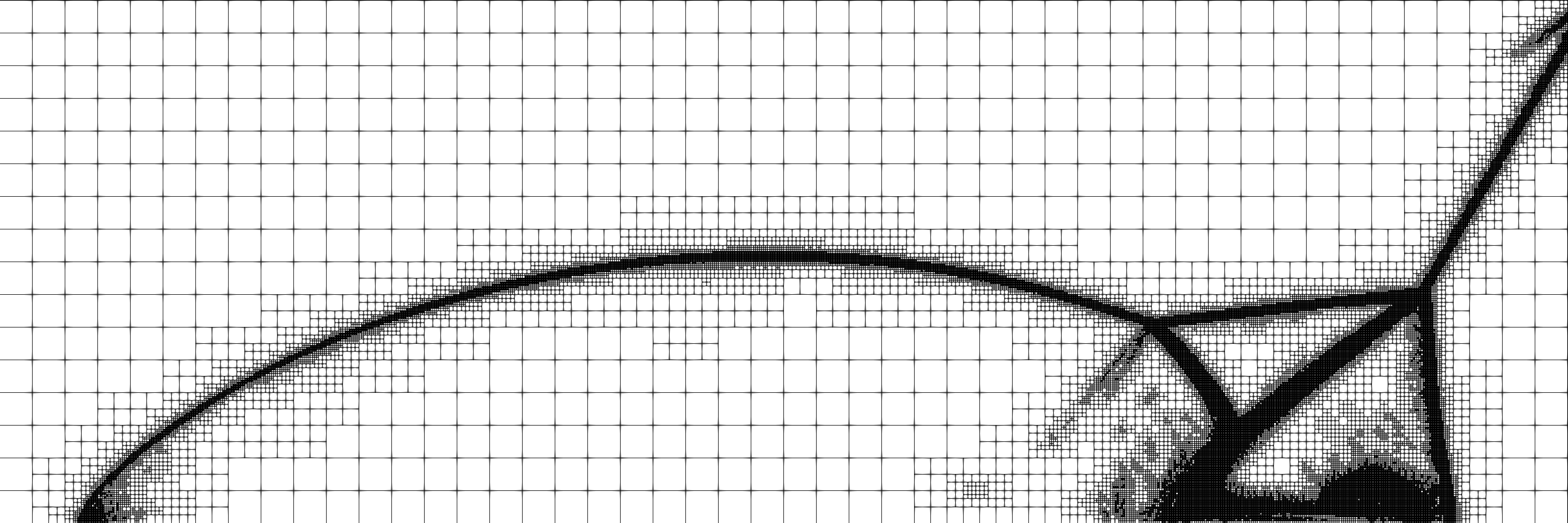}
    \caption{{$N=1$, \textsf{ES} flux}}
  \end{subfigure}
  \hfill
  \begin{subfigure}{0.49\textwidth}
    \centering
    \includegraphics[width=\linewidth]{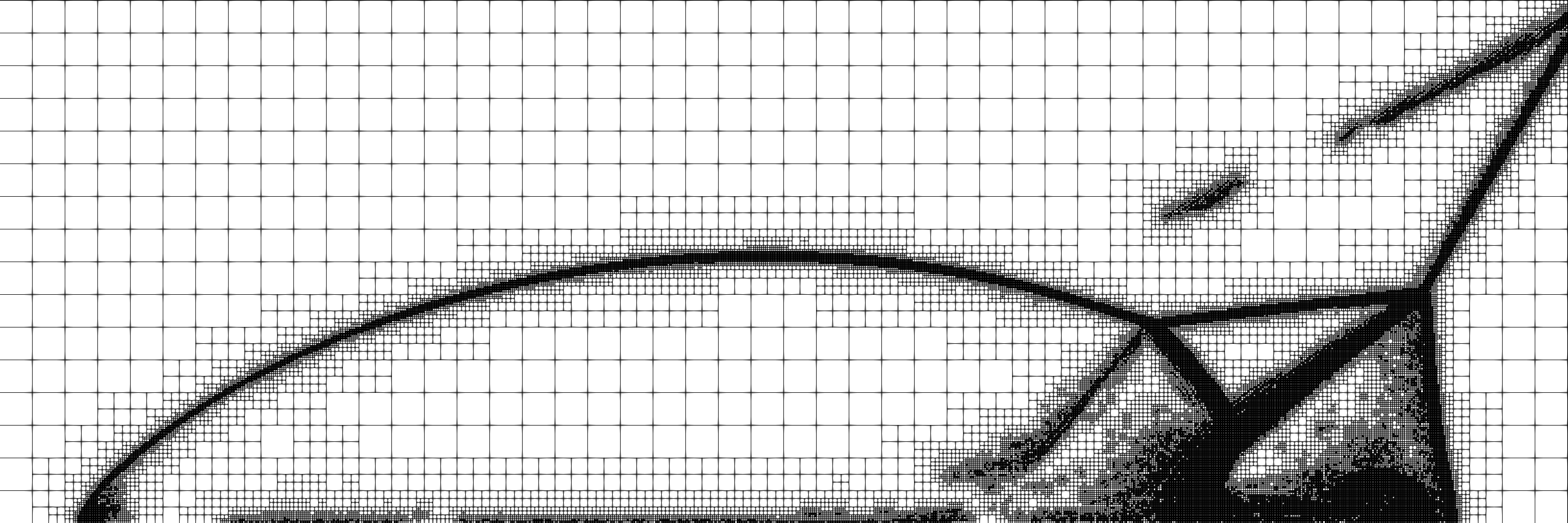}
    \caption{{$N=1$, mortar flux}}
  \end{subfigure}
  \begin{subfigure}{0.49\textwidth}
    \centering
    \includegraphics[width=\linewidth]{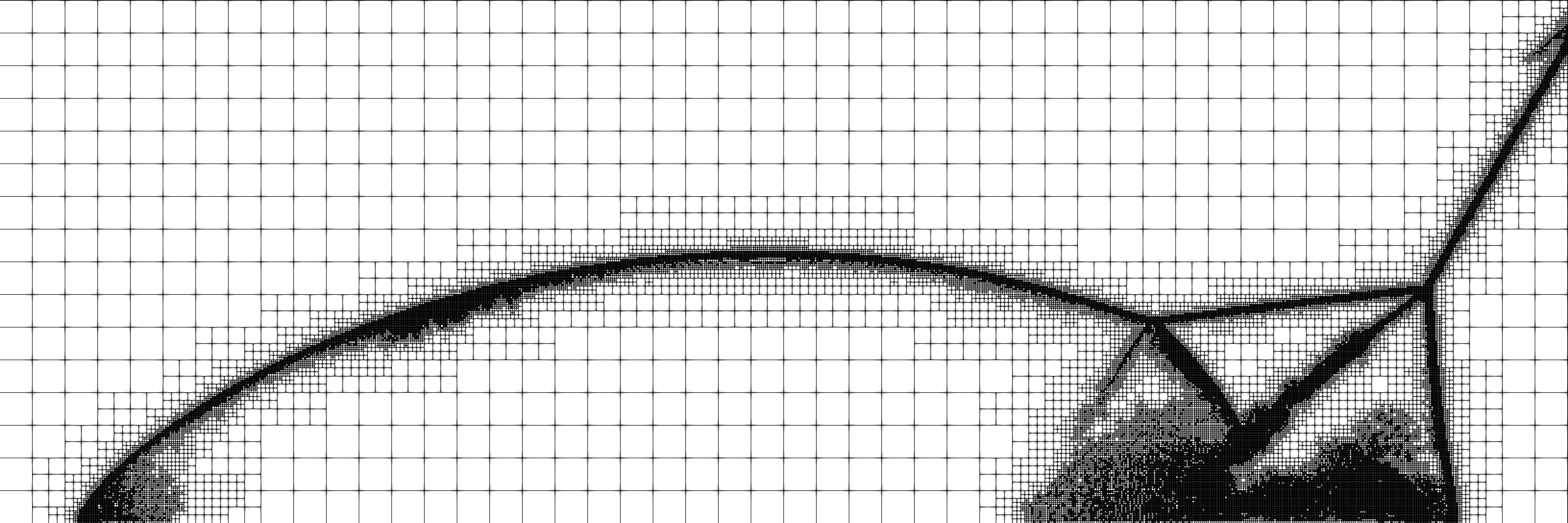}
    \caption{{$N=2$, \textsf{ES} flux}}
  \end{subfigure}
  \hfill
  \begin{subfigure}{0.49\textwidth}
    \centering
    \includegraphics[width=\linewidth]{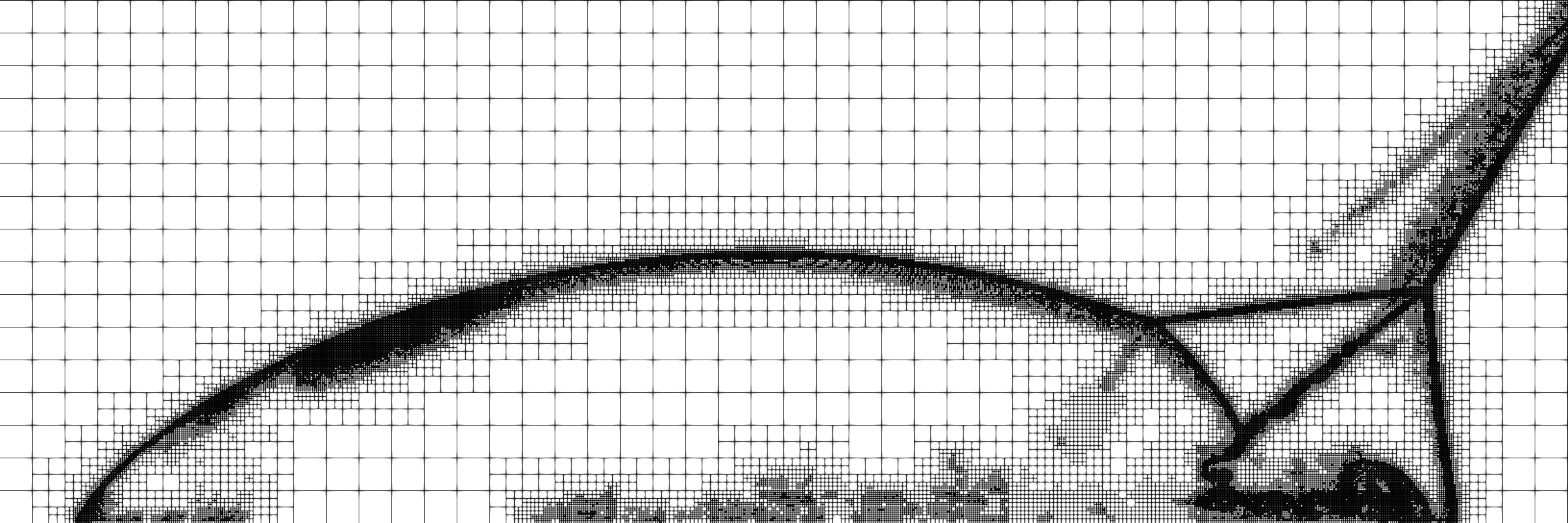}
    \caption{{$N=2$, mortar flux}}
  \end{subfigure}
  \caption{{\textit{Double Mach reflection}: adaptive meshes on
  $[0,3]\times[0,1]$ at $t=0.2$ corresponding to the four solutions in
  Figure~\ref{fig:dmr_density}.}}
  \label{fig:dmr_mesh}
\end{figure}
\FloatBarrier
\subsection{High Mach number astrophysical jet}
\label{subsec:num_machjet}

As a stringent test of the positivity-preserving property of the
proposed method under extreme conditions, we consider a high Mach number
astrophysical jet problem. This problem has been widely used as a severe benchmark for high-order schemes because it
contains very large velocity and density contrasts and generates
regions in which the density and pressure can become dangerously close
to zero~\cite{balsara2012self,zhang2010positivity,dzanic2022positivity,liu2024optimization}. {In the present work, this example is used to
compare the robustness and resolution of the mortar-based and
entropy-stable nonconforming fluxes and to verify that the Zhang--Shu
positivity-preserving limiter remains effective for both couplings on AMR
meshes.}

We consider the half-domain $\Omega=[0,1]\times[0,0.5]$. The ambient
gas is initialized with $(\rho, u_x, u_y, p )=(0.5,\; 0,\; 0, \;10^{-2})$. {On the left boundary $x=0$, we prescribe an inflow condition: for $y<0.05$, the inflow state is set to}
\[
(\rho, u_x, u_y, p ) = {(5,\;800,\;0,\;0.4127)},
\]
corresponding to {an approximately Mach 2350 jet}, while {for $y>0.05$} the inflow state is
given by the ambient initial condition. A symmetry boundary condition
is imposed on the bottom boundary, and outflow boundary conditions are
used on all remaining boundaries. {We perform two calculations
with a $\mathbb Q_3$ approximation, using the mortar-based and
entropy-stable nonconforming fluxes, respectively. Both runs use the
combined wave-speed parameter \eqref{alpha},
$\alpha=\max\{\alpha_{\rm char},\alpha_{\rm GP}\}$, a CFL number of
$0.8$, and identical OE and AMR parameters. The OE parameters are
$C_{\mathrm{oe}}=0.1$ and $s=0.4$. For AMR, both calculations start from the same
uniform $300\times150$ mesh, with
$\Delta x=\Delta y=1/300$, use
$C_{\mathrm{ref}}=C_{\mathrm{crs}}=0.1$, and allow up to three
refinement levels. Thus, the finest locally refined cells are eight times
smaller than the base cells. The simulations are run until
$T=0.001$.}

{Figure~\ref{fig:machjet} compares the density fields and adaptive
meshes at $T=0.001$. This problem is particularly demanding for the
positivity-preserving mechanism, since the jet evolution produces
near-vacuum regions in which even a small loss of admissibility may cause the
computation to fail. Both interface treatments complete the calculation
robustly and resolve the principal bow shock, jet head, reflected waves, and
downstream structures with comparable quality. The two solutions are not
pointwise identical: their small-scale roll-up patterns differ because the
interface dissipation changes the evolving solution and hence the AMR
decisions. Nevertheless, neither flux exhibits a systematic loss of the
dominant flow structures. At the final time, the mortar and entropy-stable
calculations contain $99{,}429$ and $99{,}693$ elements, respectively; their
maximum element counts are $106{,}218$ and $100{,}245$. Thus, the two
independently evolved adaptive meshes have closely comparable sizes as well
as comparable resolved flow content.}
\begin{figure}[htbp]
  \centering
  \begin{subfigure}{0.49\textwidth}
    \centering
    \includegraphics[width=\linewidth]{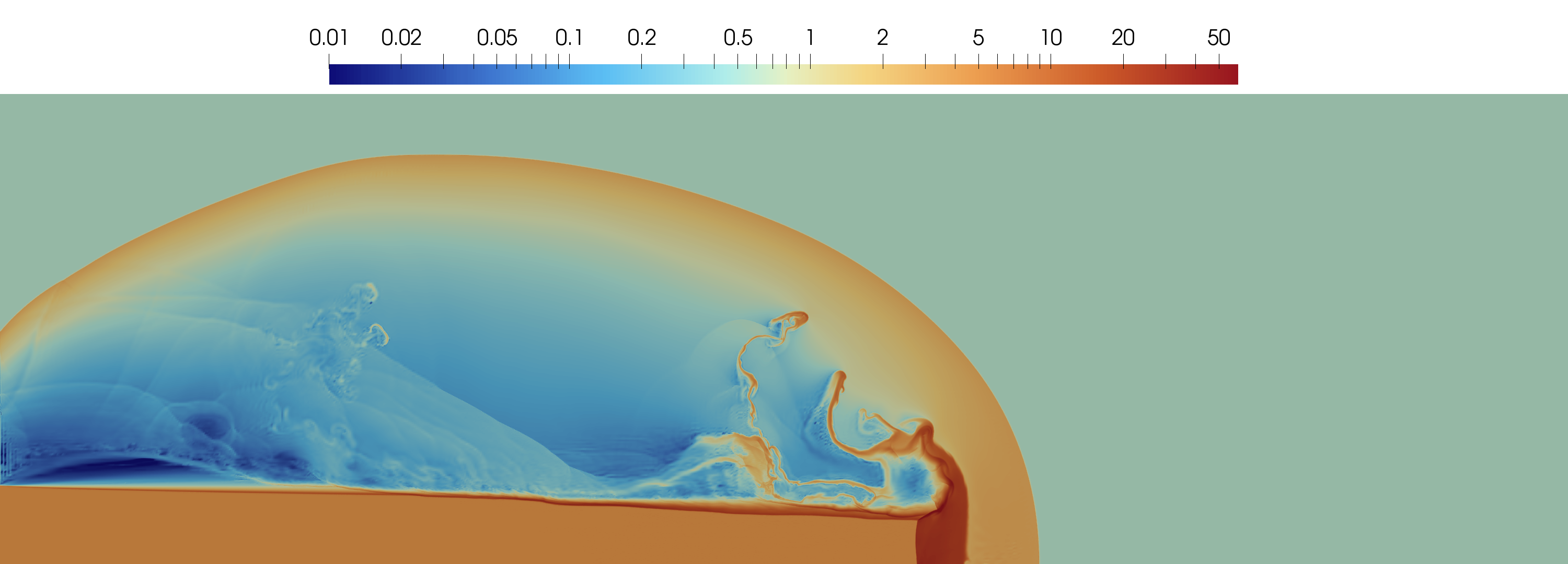}
    \caption{{mortar flux: density}}
  \end{subfigure}
  \hfill
  \begin{subfigure}{0.49\textwidth}
    \centering
    \includegraphics[width=\linewidth]{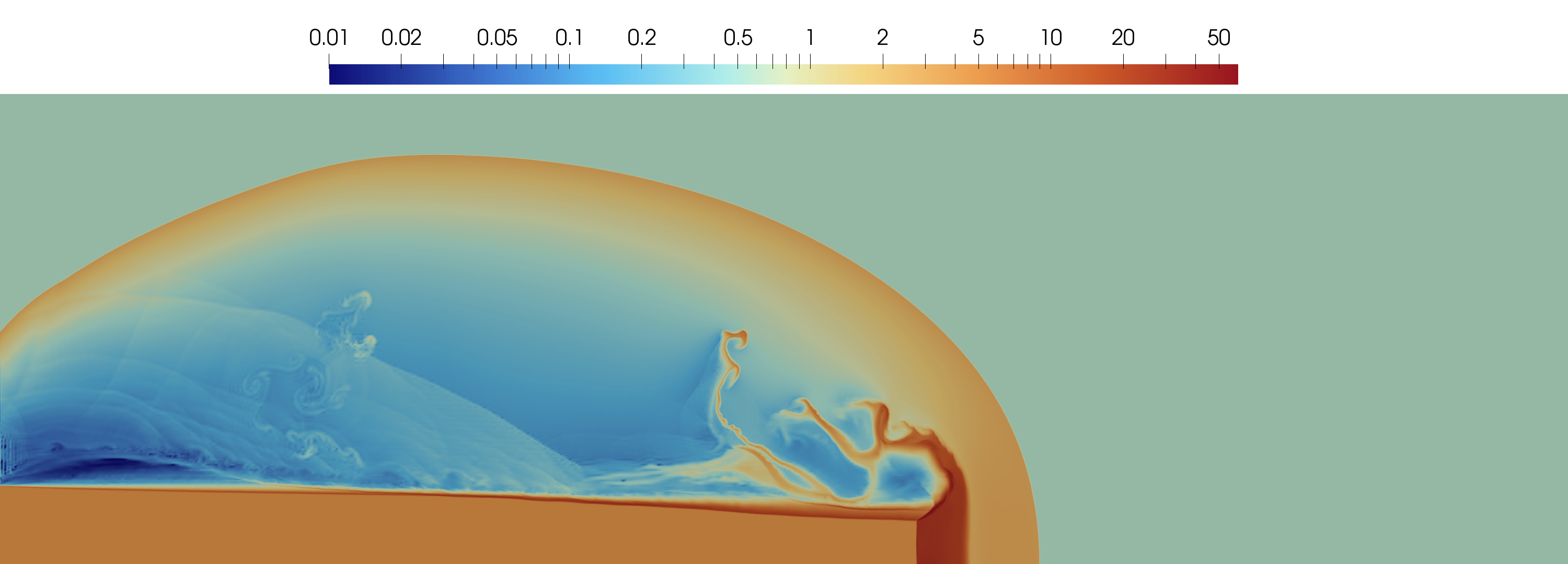}
    \caption{{\textsf{ES} flux: density}}
  \end{subfigure}

  \par\vspace{0.15cm}

  \begin{subfigure}{0.49\textwidth}
    \centering
    \includegraphics[width=\linewidth]{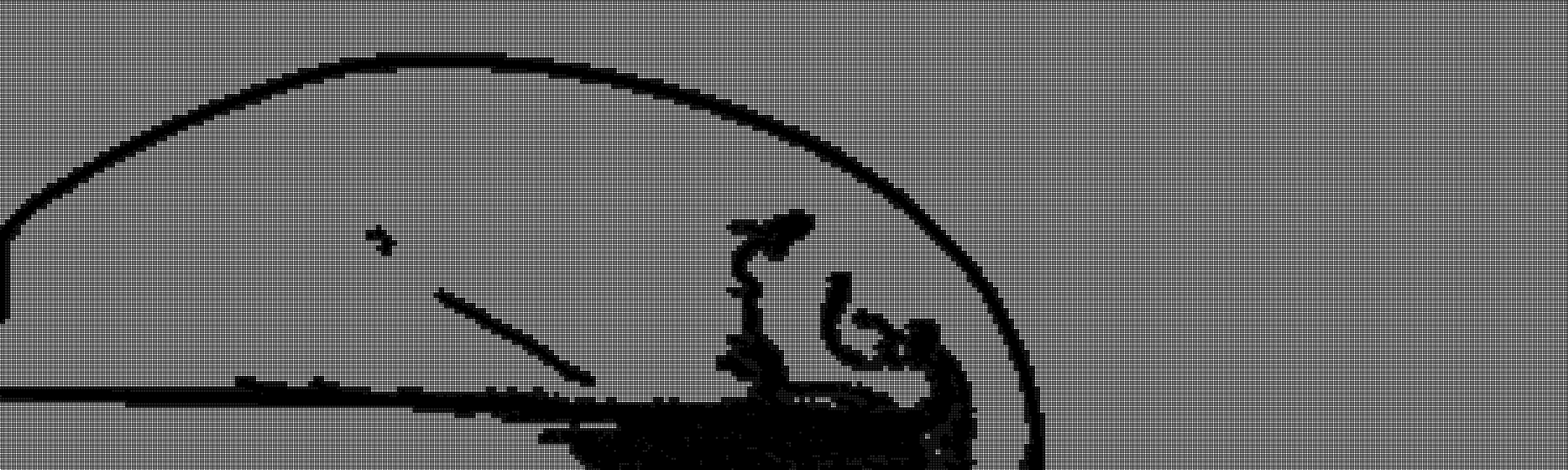}
    \caption{{mortar flux: adaptive mesh}}
  \end{subfigure}
  \hfill
  \begin{subfigure}{0.49\textwidth}
    \centering
    \includegraphics[width=\linewidth]{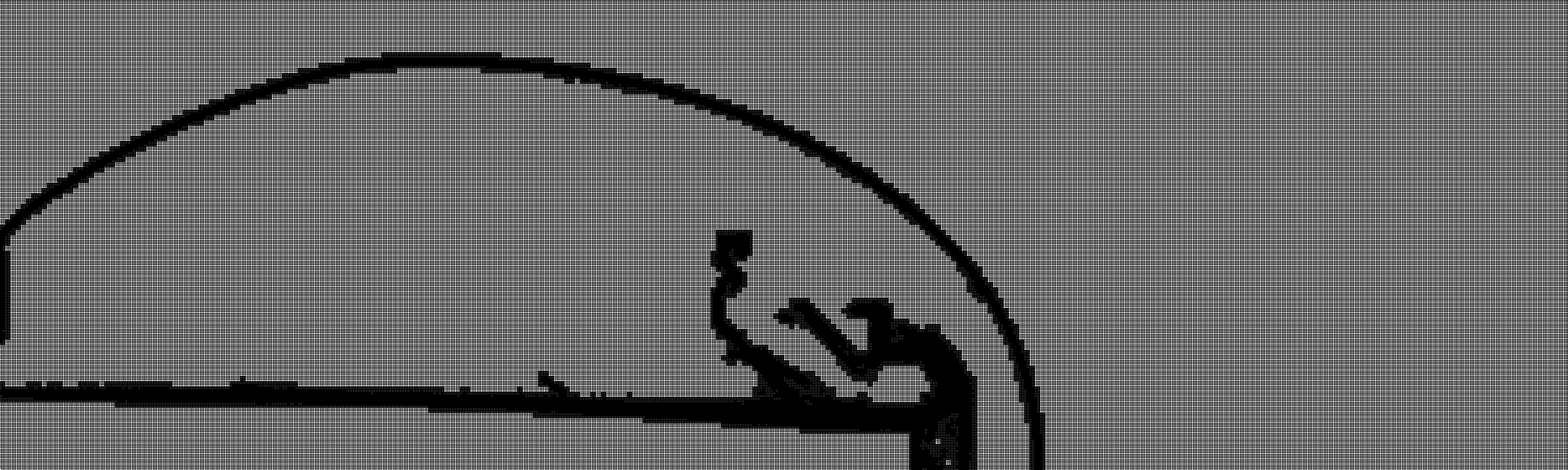}
    \caption{{\textsf{ES} flux: adaptive mesh}}
  \end{subfigure}
  \caption{{\textit{High Mach number astrophysical jet}: comparison
  of the mortar-based and entropy-stable nonconforming fluxes at
  $T=0.001$ with polynomial degree $N=3$. The top row shows density using
  the common logarithmic range $0.01\leq\rho\leq60$, and the bottom row
  shows the corresponding adaptive meshes. The base mesh is
  $300\times150$, up to three AMR levels are allowed, and all panels show
  $[0,1]\times[0,0.3]$.}}
  \label{fig:machjet}
\end{figure}

\subsection{Supersonic flow past a circular cylinder}
\label{subsec:num_cylinder}

We next consider the classical problem of supersonic flow past a
two-dimensional circular cylinder~\cite{guermond2018second,deng2023new,nazarov2017numerical,chen2025high,YangFu26}. This is a demanding benchmark on a
curvilinear domain, involving a strong bow shock, downstream shock
interactions, and the development of increasingly fine-scale structures such as Kelvin--Helmholtz instabilities. {In the present work, this example is used to compare the mortar-based and entropy-stable nonconforming fluxes in a curvilinear AMR calculation with sustained fine-scale flow development.}

The computational domain is
$\Omega=[-0.6,3.4]\times[-1,1]$ and contains a circular cylinder of
diameter $D=0.5$ centered at $(0,0)$. Following our previous work, the
mesh is generated by \texttt{Gmsh}
\cite{geuzaine2009gmsh,remacle2012blossom}. We use a nearly uniform
curvilinear quadrilateral mesh obtained with target mesh size $0.03$ in
\texttt{Gmsh}. The resulting initial mesh consists of $10{,}397$
third-order curvilinear quadrilateral elements. The initial condition
is a uniform Mach 3 inflow with $\rho=1$, $p=1$, and velocity
$\bm u=(3.55,0)$. Inflow and outflow boundary conditions are imposed on
the left and right boundaries, respectively, while slip wall boundary
conditions are used on the upper and lower boundaries as well as on the
cylinder surface.

\begin{figure}[!b]
  \centering
  \includegraphics[width=0.58\textwidth]{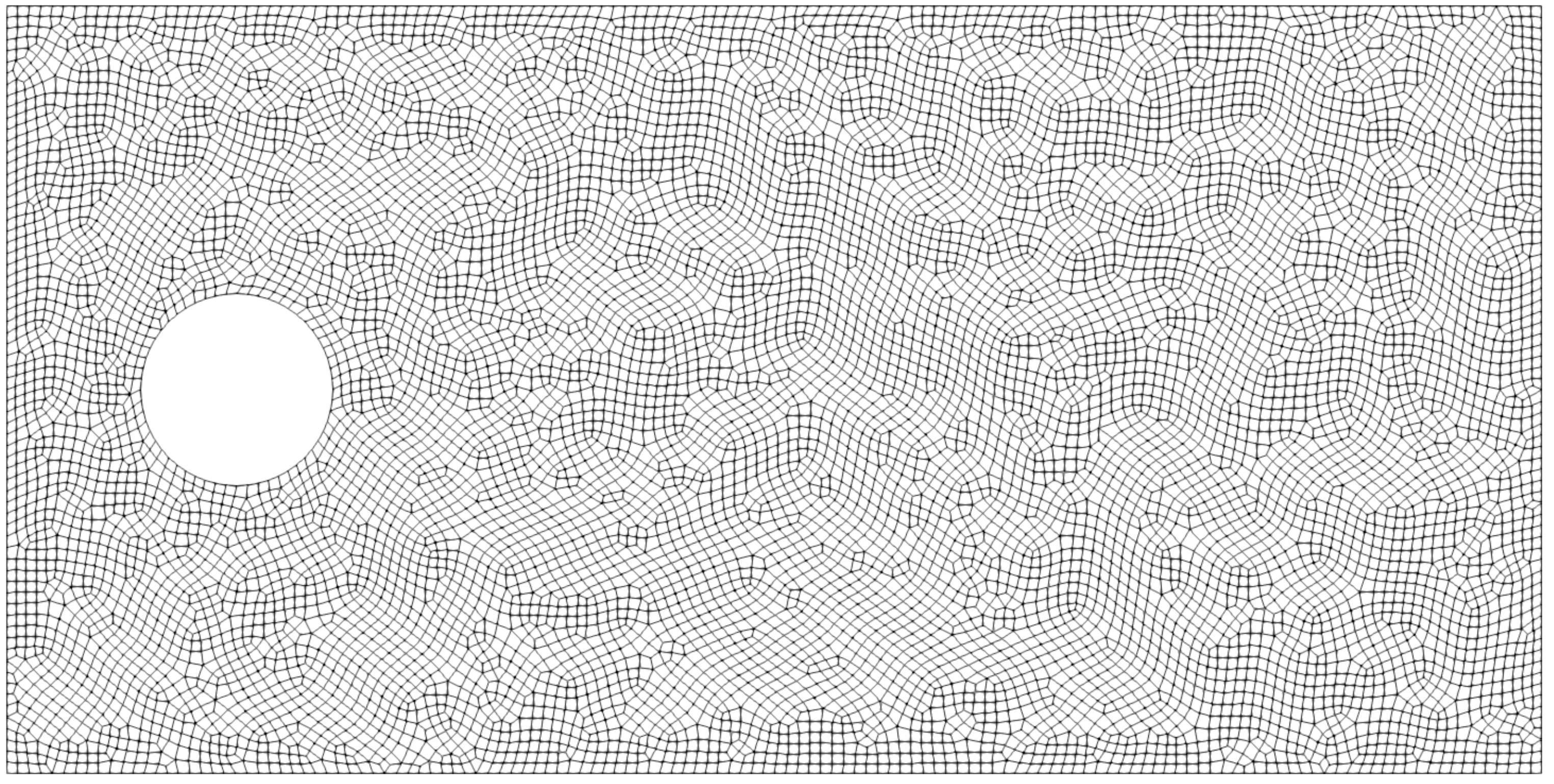}
  \caption{\textit{Supersonic flow past a circular cylinder}: initial
  third-order curvilinear quadrilateral mesh generated by
  \texttt{Gmsh}. The target mesh size in \texttt{Gmsh} is $0.03$, and
  the initial mesh contains $10{,}397$ elements.}
  \label{fig:cylinder_mesh}
\end{figure}

{We perform two $\mathbb Q_3$ calculations using the
mortar-based and entropy-stable nonconforming fluxes, respectively. Both
runs use the combined wave-speed parameter \eqref{alpha},
$\alpha=\max\{\alpha_{\rm char},\alpha_{\rm GP}\}$, a CFL number of
$0.8$, and identical OE and AMR parameters. The OE parameters are
$C_{\mathrm{oe}}=0.2$ and $s=0.2$. The AMR parameters are
$C_{\mathrm{ref}}=0.03$, $C_{\mathrm{crs}}=0.015$, and a maximum of two
refinement levels. Figure~\ref{fig:cylinder_mesh} shows the common
initial curvilinear mesh. Figures~\ref{fig:cylinder_mortar} and
\ref{fig:cylinder_es} show the Schlieren-like density fields and the
corresponding adaptive meshes for the mortar-based and entropy-stable
fluxes at $t=0.3$, $1.5$, $2.4$, and $4.2$. The Schlieren-like
visualization is defined by $\log(1+\|\nabla\rho\|)$ and is plotted over
the common range $[0,10]$.}

{Both interface treatments remain robust through $T=4.5$ and
resolve the evolving bow shock, downstream shock interactions, wake, and
the Kelvin--Helmholtz roll-up structures near the upper and lower walls
\cite{YangFu26,chen2025high}. Their large-scale wave patterns are
comparable, although the small-scale structures and the associated
adaptive meshes are not pointwise identical. At the final time, the
mortar-based and entropy-stable calculations contain $49{,}646$ and
$83{,}201$ elements, respectively, while their peak element counts are
$51{,}878$ and $84{,}056$. With identical AMR controls, this difference
arises from the flux-dependent solution and its interaction with the common
indicator. Thus, the interface flux can substantially alter the adaptive-mesh
trajectory in this fine-structure-sensitive problem, but this result is not a
universal efficiency ranking. Despite their different smooth-test convergence
rates, both fluxes remain stable and well resolved.}

\begin{figure}[!htbp]
  \centering
  \begin{subfigure}{0.516\textwidth}
    \centering
    \includegraphics[width=\linewidth]{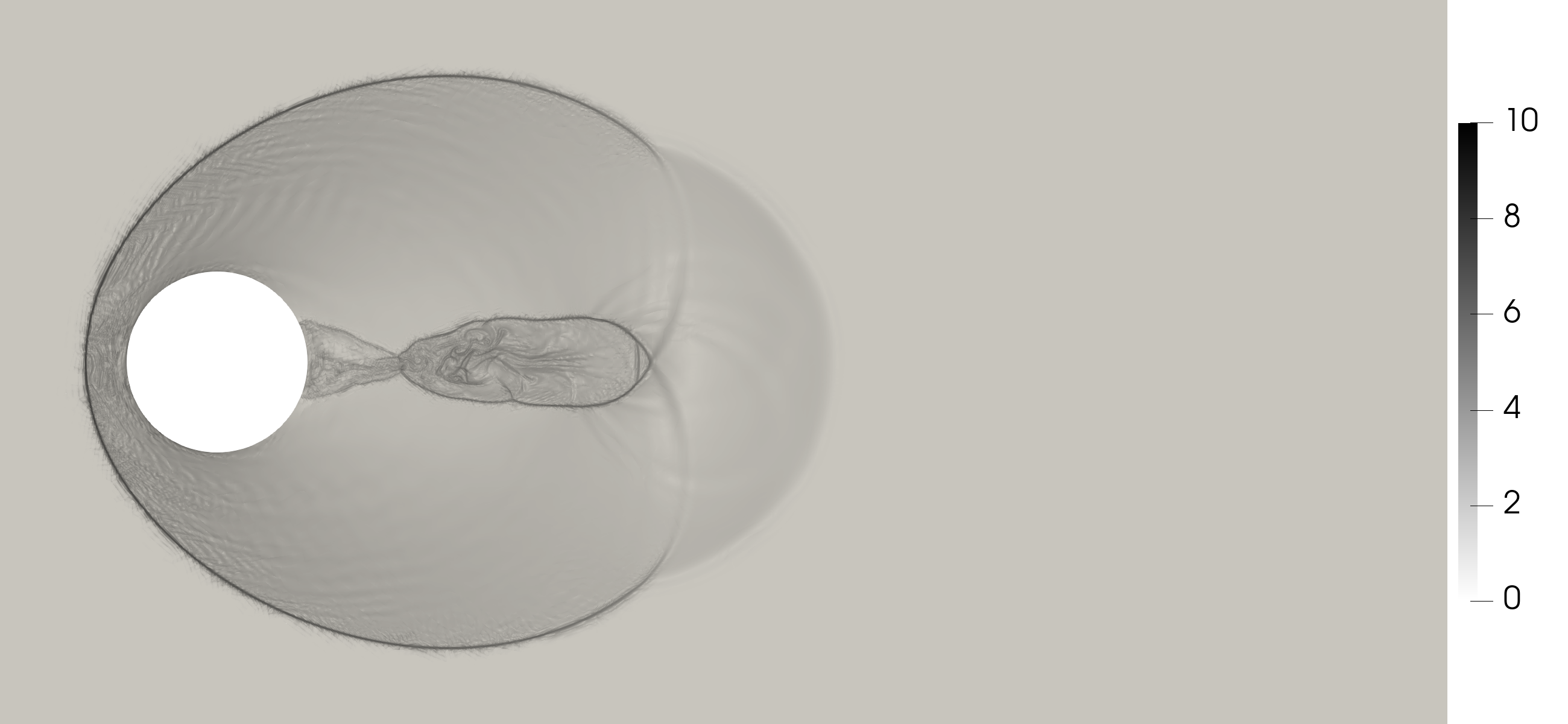}
    \caption{{Schlieren field at $t=0.3$}}
  \end{subfigure}
  \hfill
  \begin{subfigure}{0.476\textwidth}
    \centering
    \includegraphics[width=\linewidth]{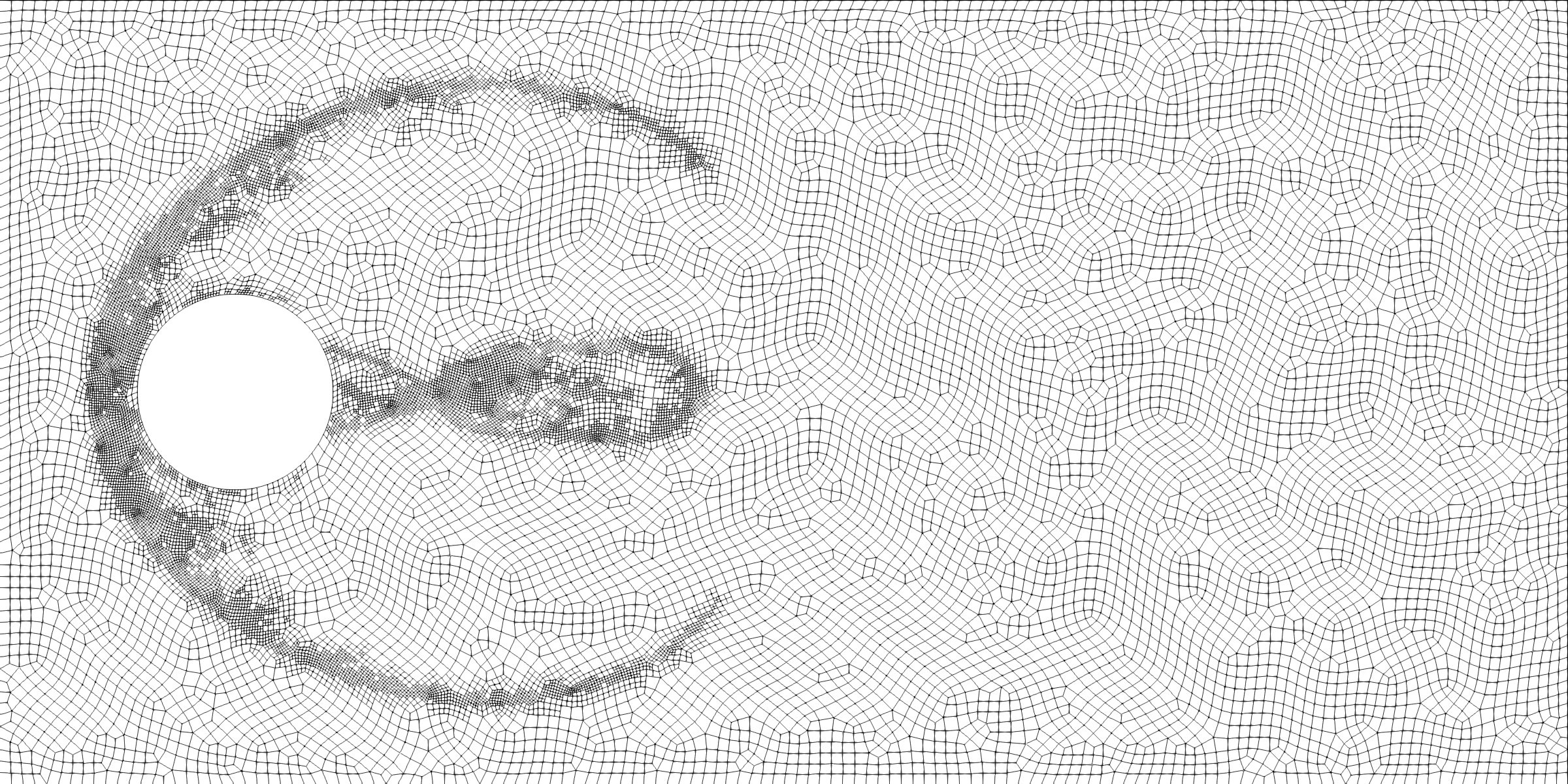}
    \caption{{adaptive mesh at $t=0.3$}}
  \end{subfigure}

  \begin{subfigure}{0.516\textwidth}
    \centering
    \includegraphics[width=\linewidth]{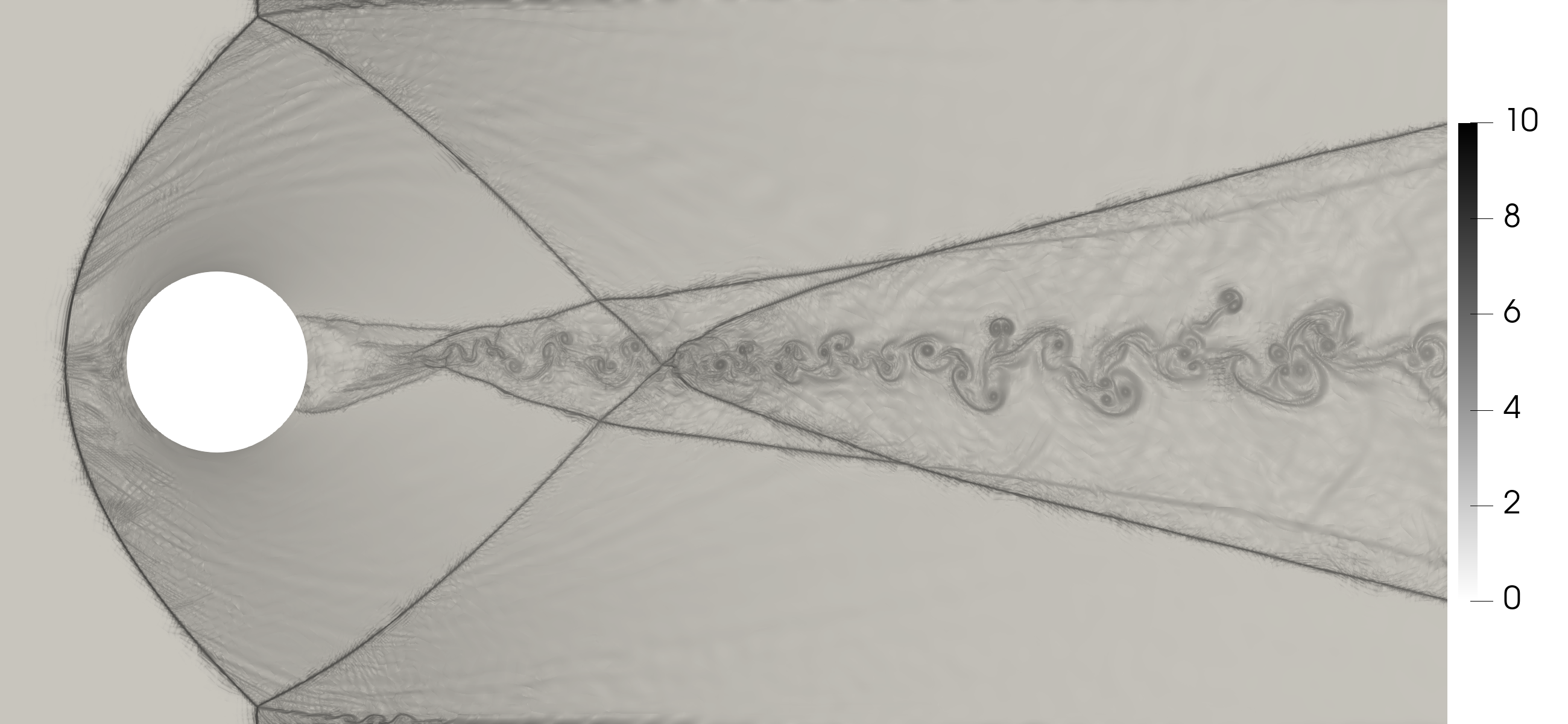}
    \caption{{Schlieren field at $t=1.5$}}
  \end{subfigure}
  \hfill
  \begin{subfigure}{0.476\textwidth}
    \centering
    \includegraphics[width=\linewidth]{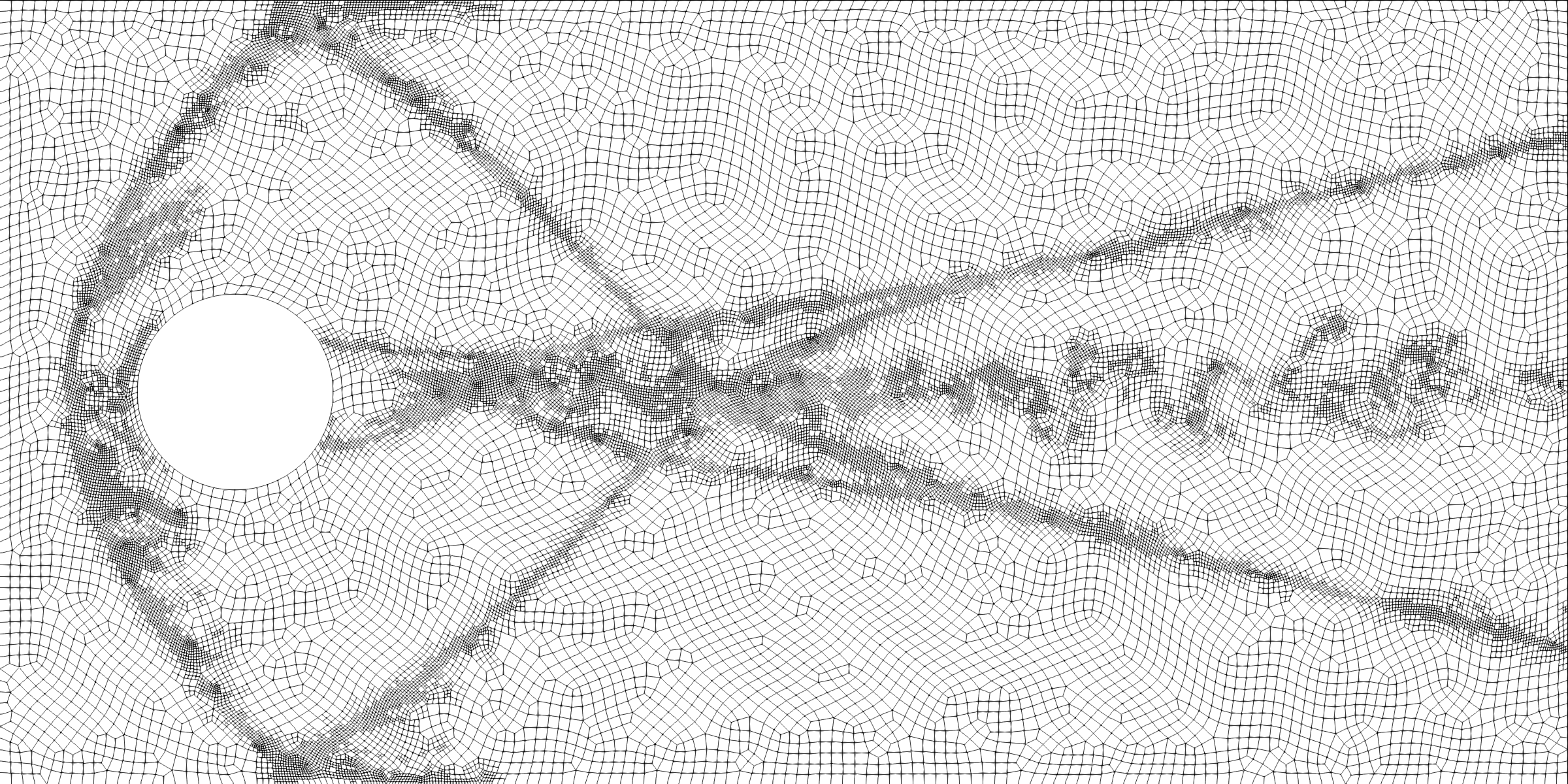}
    \caption{{adaptive mesh at $t=1.5$}}
  \end{subfigure}

  \begin{subfigure}{0.516\textwidth}
    \centering
    \includegraphics[width=\linewidth]{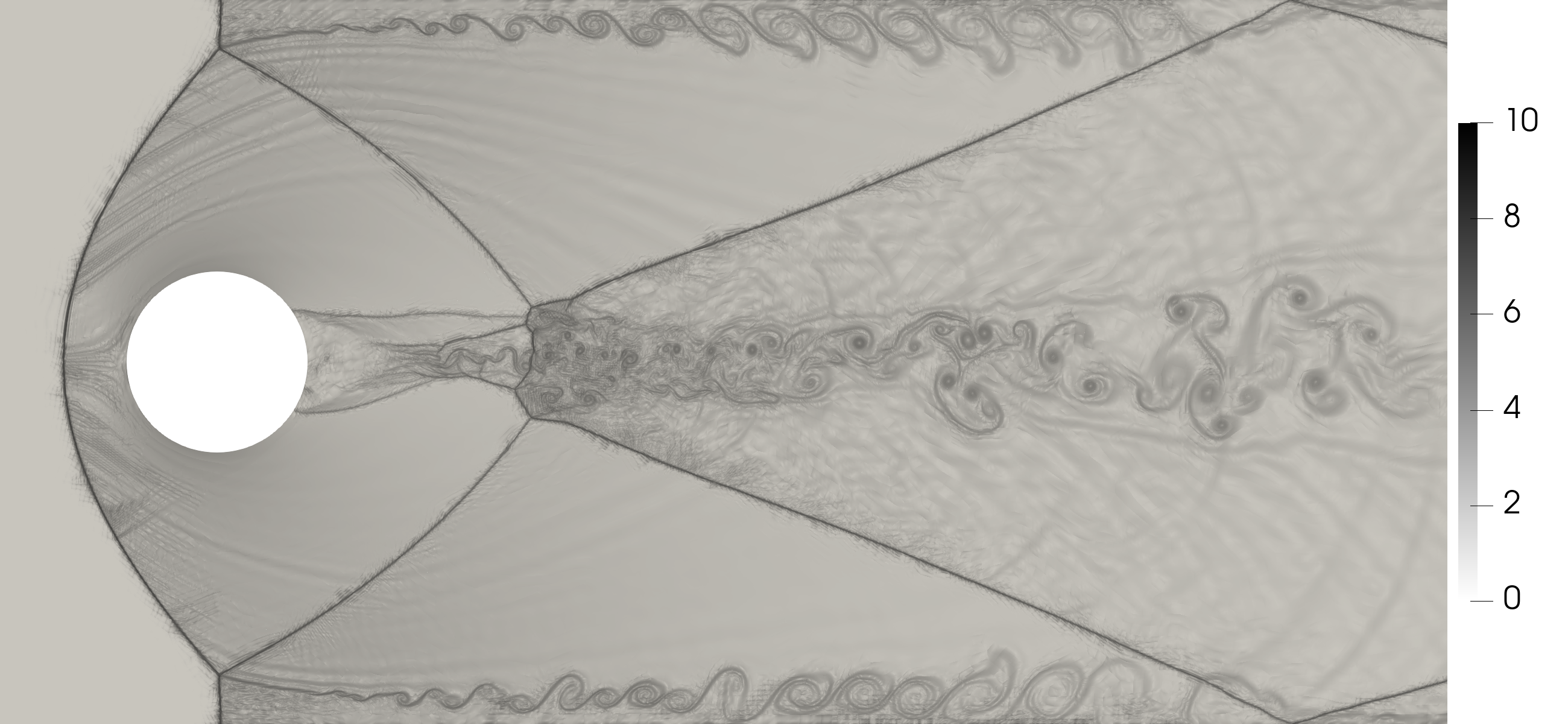}
    \caption{{Schlieren field at $t=2.4$}}
  \end{subfigure}
  \hfill
  \begin{subfigure}{0.476\textwidth}
    \centering
    \includegraphics[width=\linewidth]{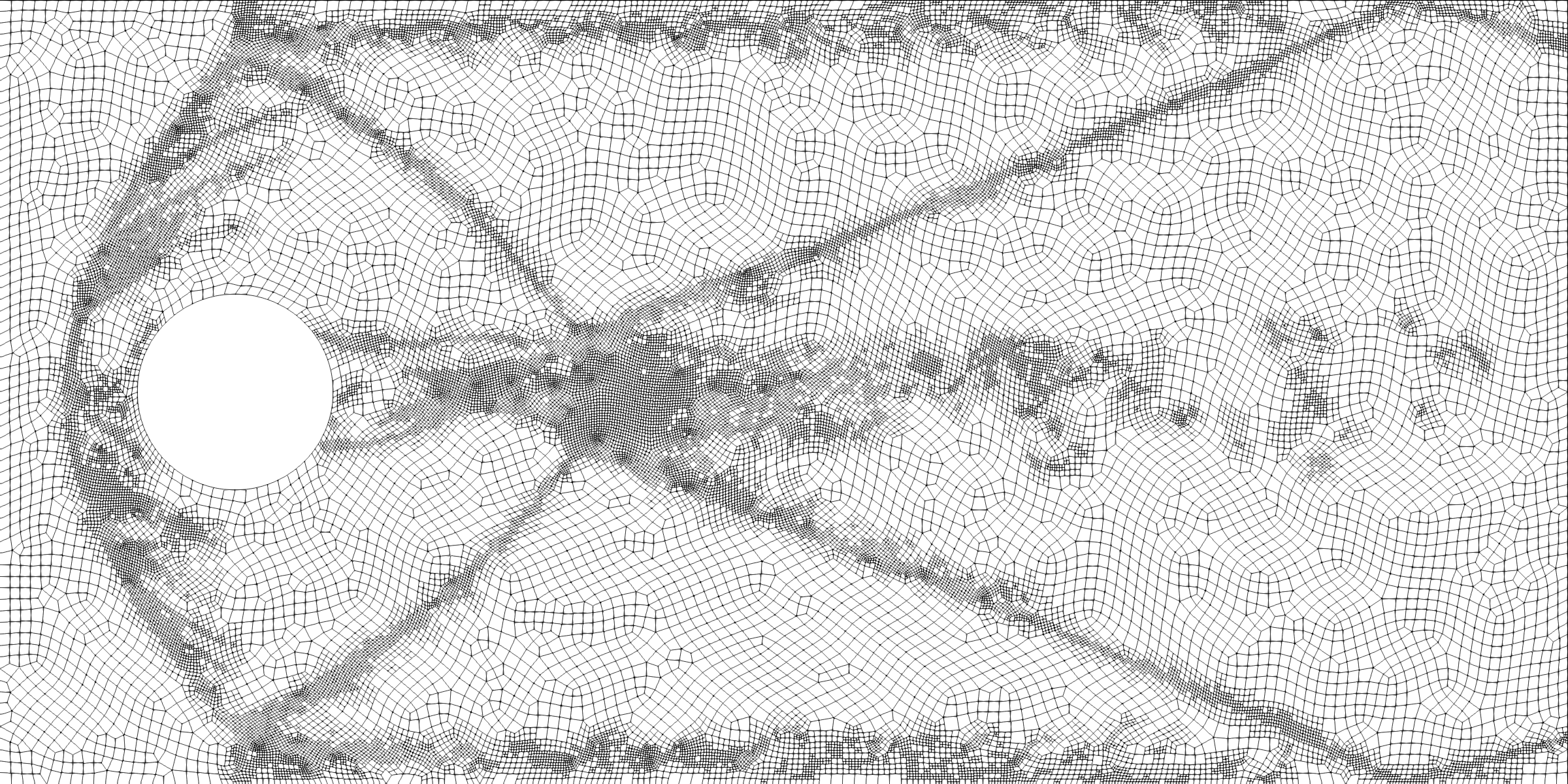}
    \caption{{adaptive mesh at $t=2.4$}}
  \end{subfigure}

  \begin{subfigure}{0.516\textwidth}
    \centering
    \includegraphics[width=\linewidth]{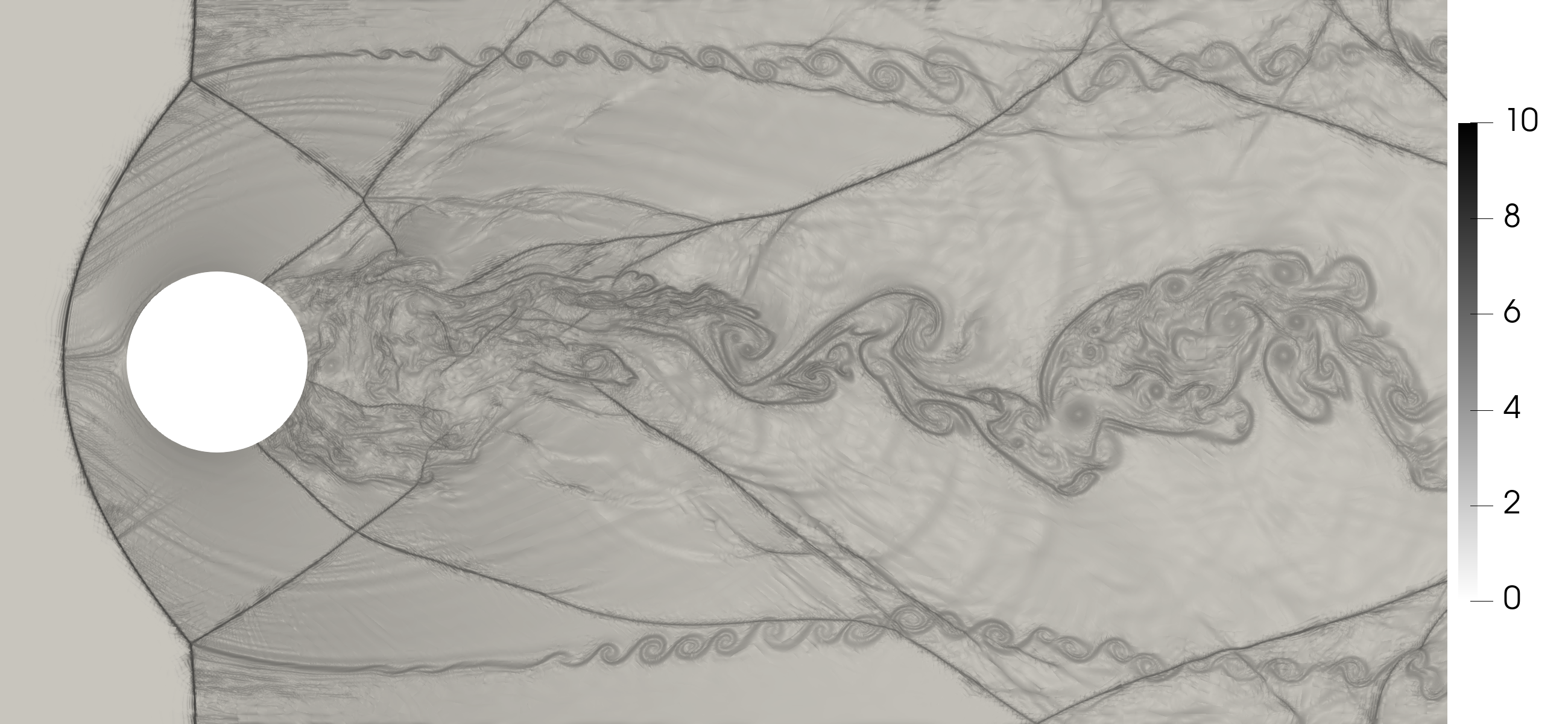}
    \caption{{Schlieren field at $t=4.2$}}
  \end{subfigure}
  \hfill
  \begin{subfigure}{0.476\textwidth}
    \centering
    \includegraphics[width=\linewidth]{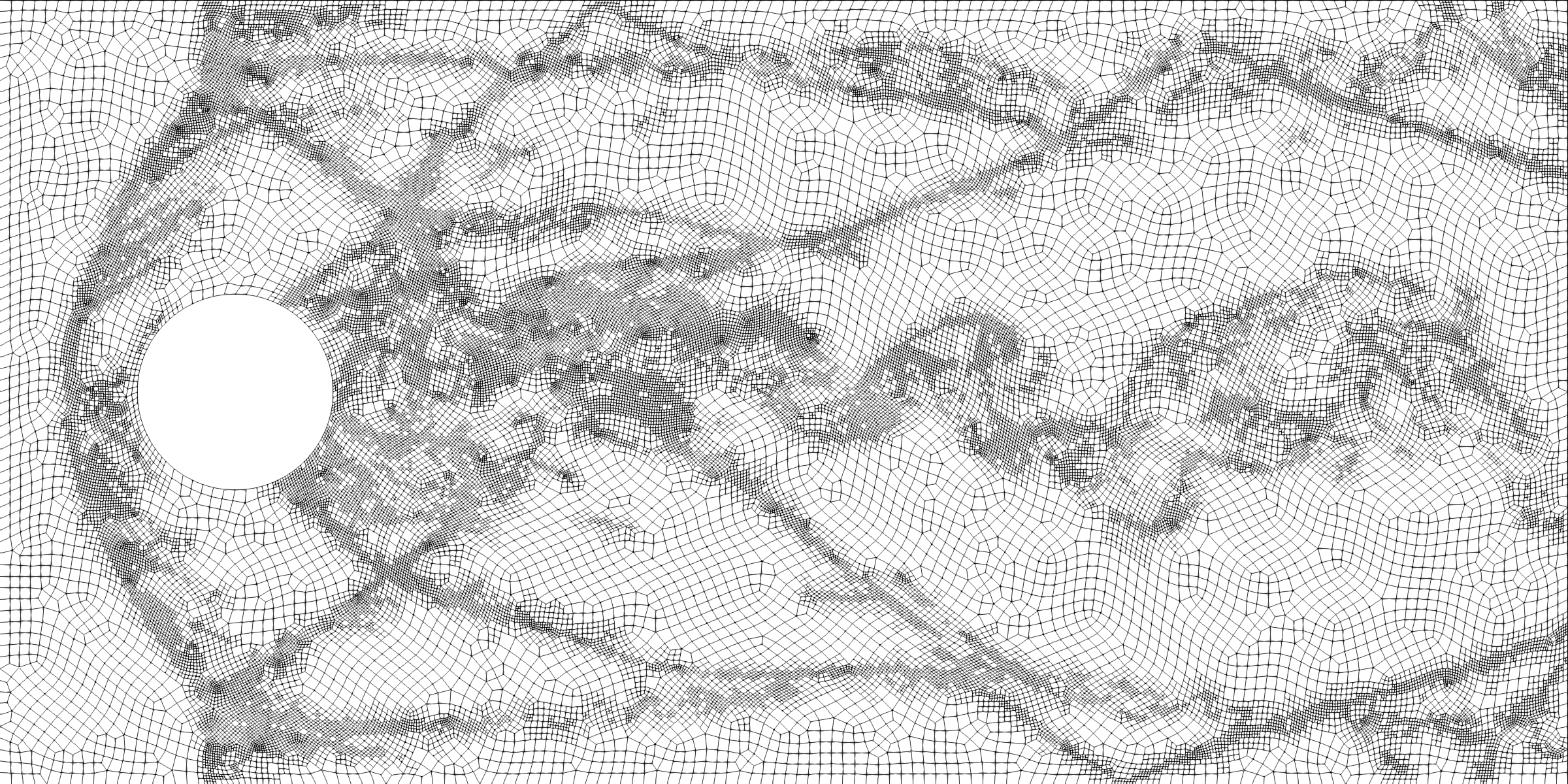}
    \caption{{adaptive mesh at $t=4.2$}}
  \end{subfigure}

  \caption{{\textit{Supersonic flow past a circular cylinder}:
  mortar-based nonconforming flux with $N=3$. The left column shows
  Schlieren-like density fields, and the right column shows the
  corresponding adaptive meshes.}}
  \label{fig:cylinder_mortar}
\end{figure}

\begin{figure}[!htbp]
  \centering
  \begin{subfigure}{0.516\textwidth}
    \centering
    \includegraphics[width=\linewidth]{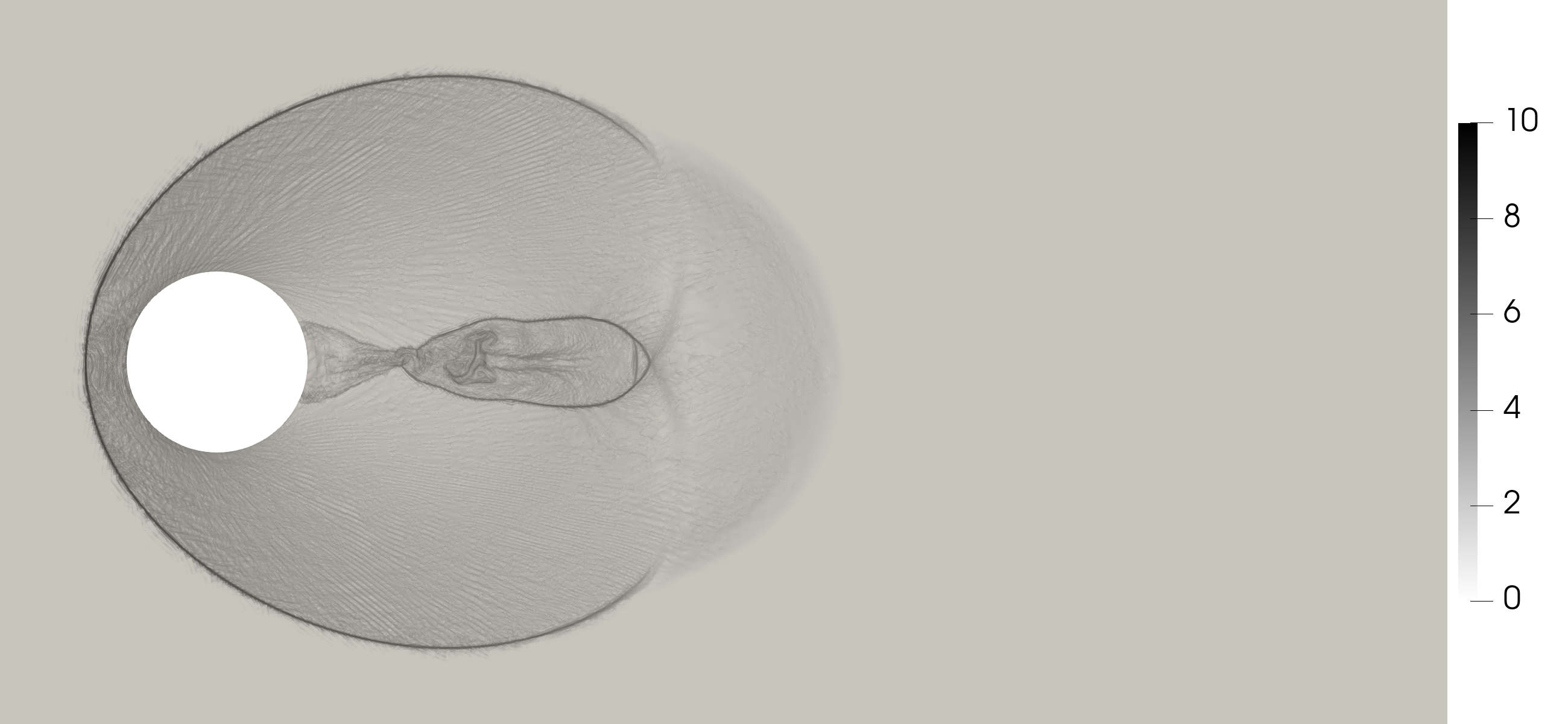}
    \caption{{Schlieren field at $t=0.3$}}
  \end{subfigure}
  \hfill
  \begin{subfigure}{0.476\textwidth}
    \centering
    \includegraphics[width=\linewidth]{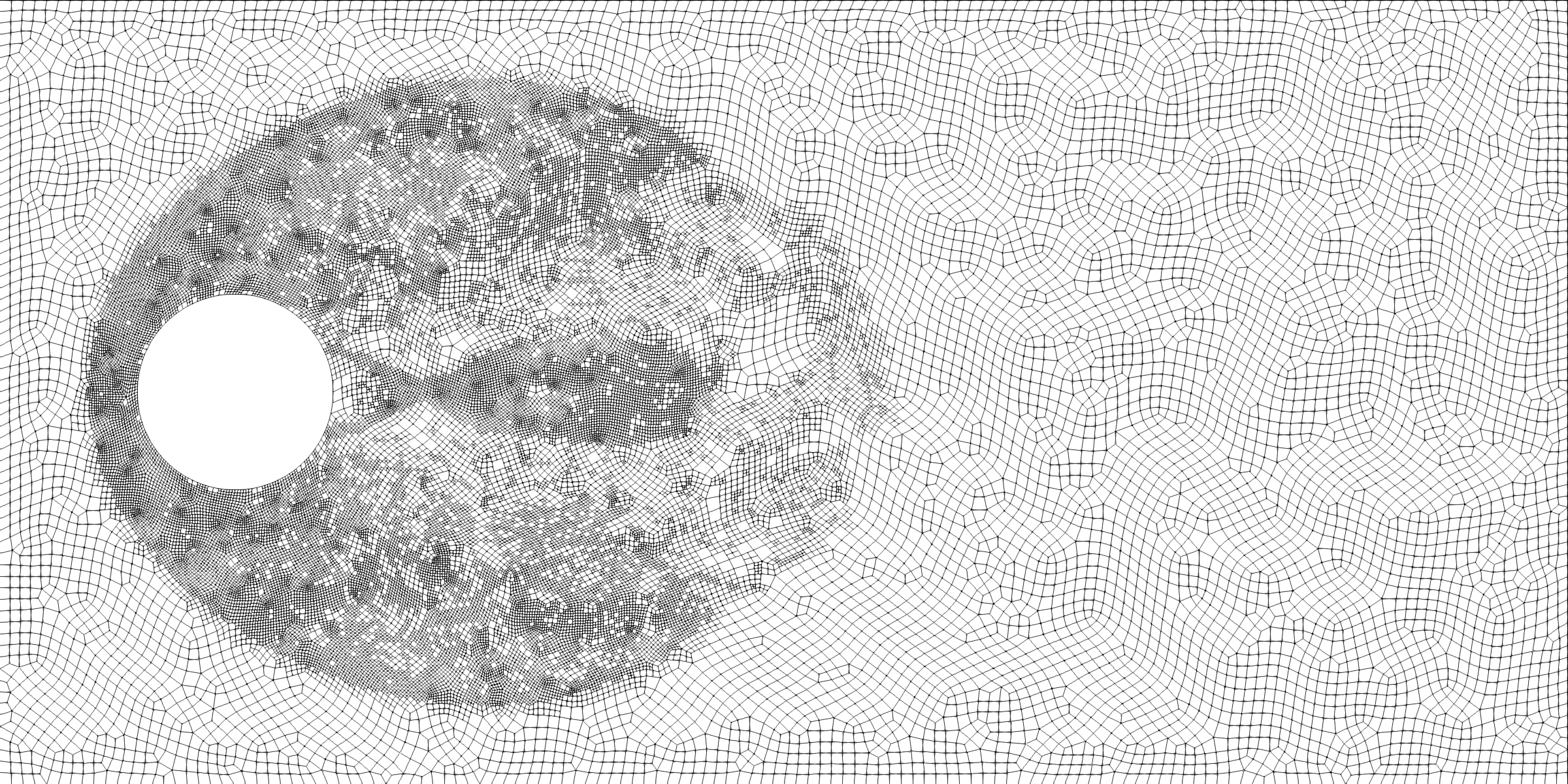}
    \caption{{adaptive mesh at $t=0.3$}}
  \end{subfigure}

  \begin{subfigure}{0.516\textwidth}
    \centering
    \includegraphics[width=\linewidth]{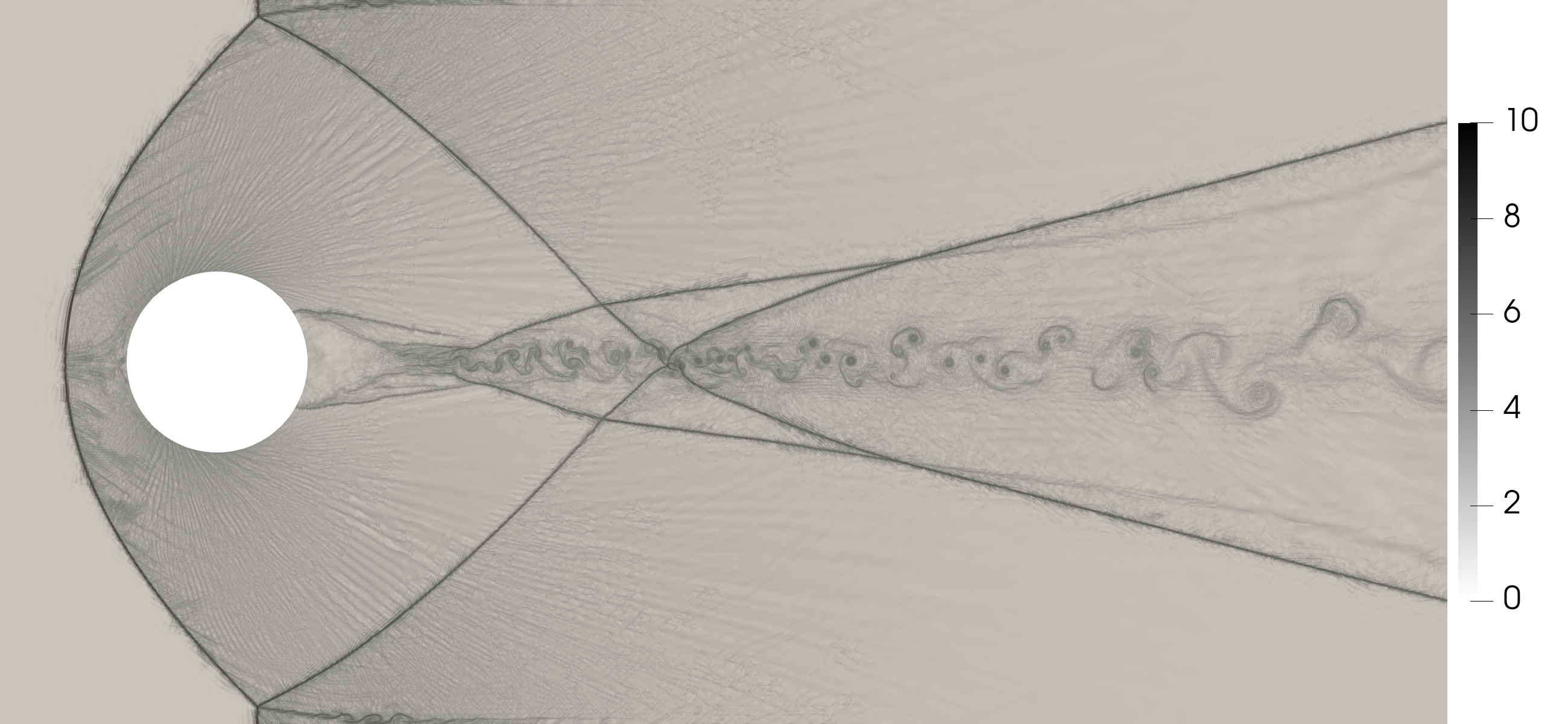}
    \caption{{Schlieren field at $t=1.5$}}
  \end{subfigure}
  \hfill
  \begin{subfigure}{0.476\textwidth}
    \centering
    \includegraphics[width=\linewidth]{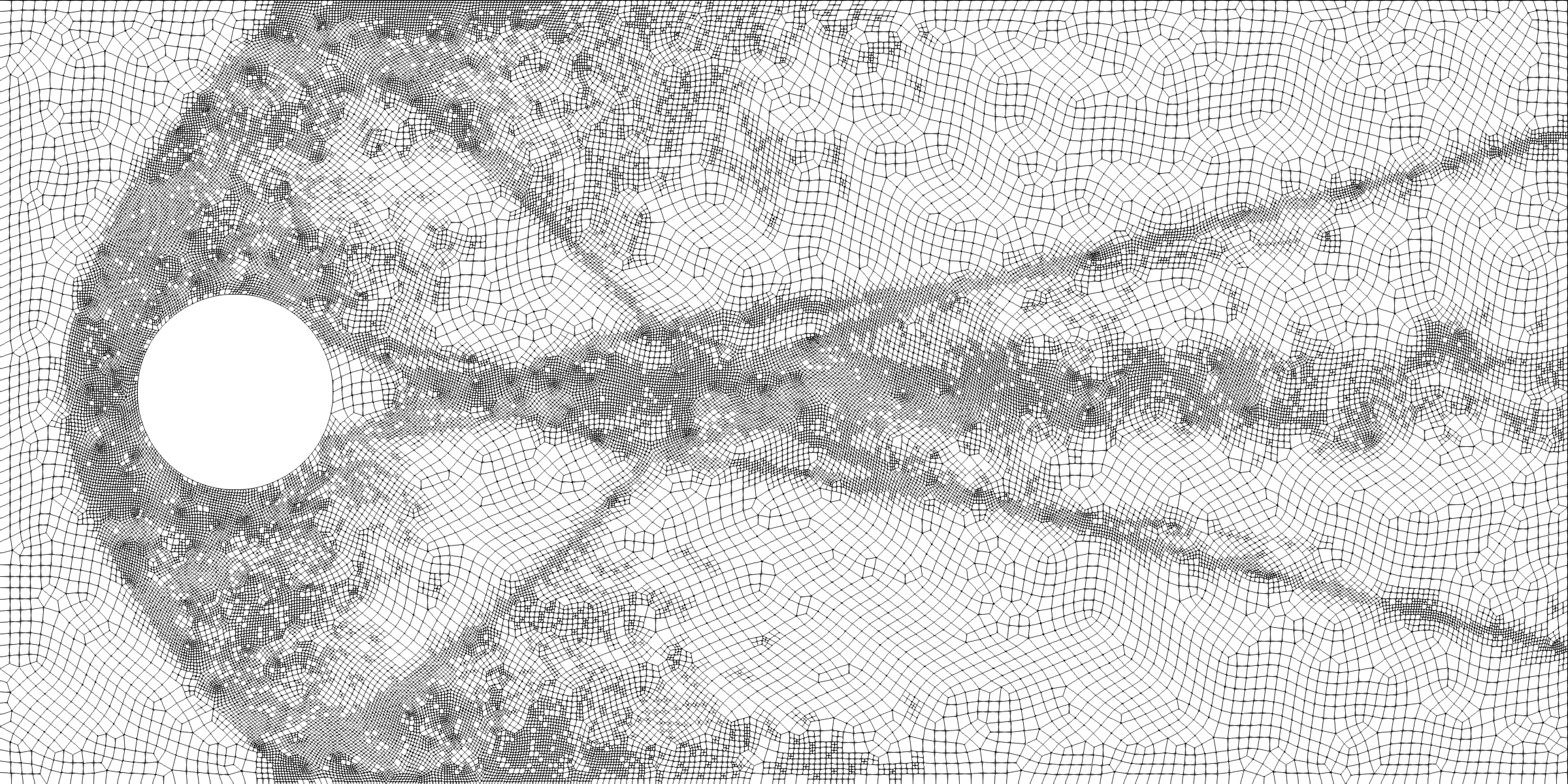}
    \caption{{adaptive mesh at $t=1.5$}}
  \end{subfigure}

  \begin{subfigure}{0.516\textwidth}
    \centering
    \includegraphics[width=\linewidth]{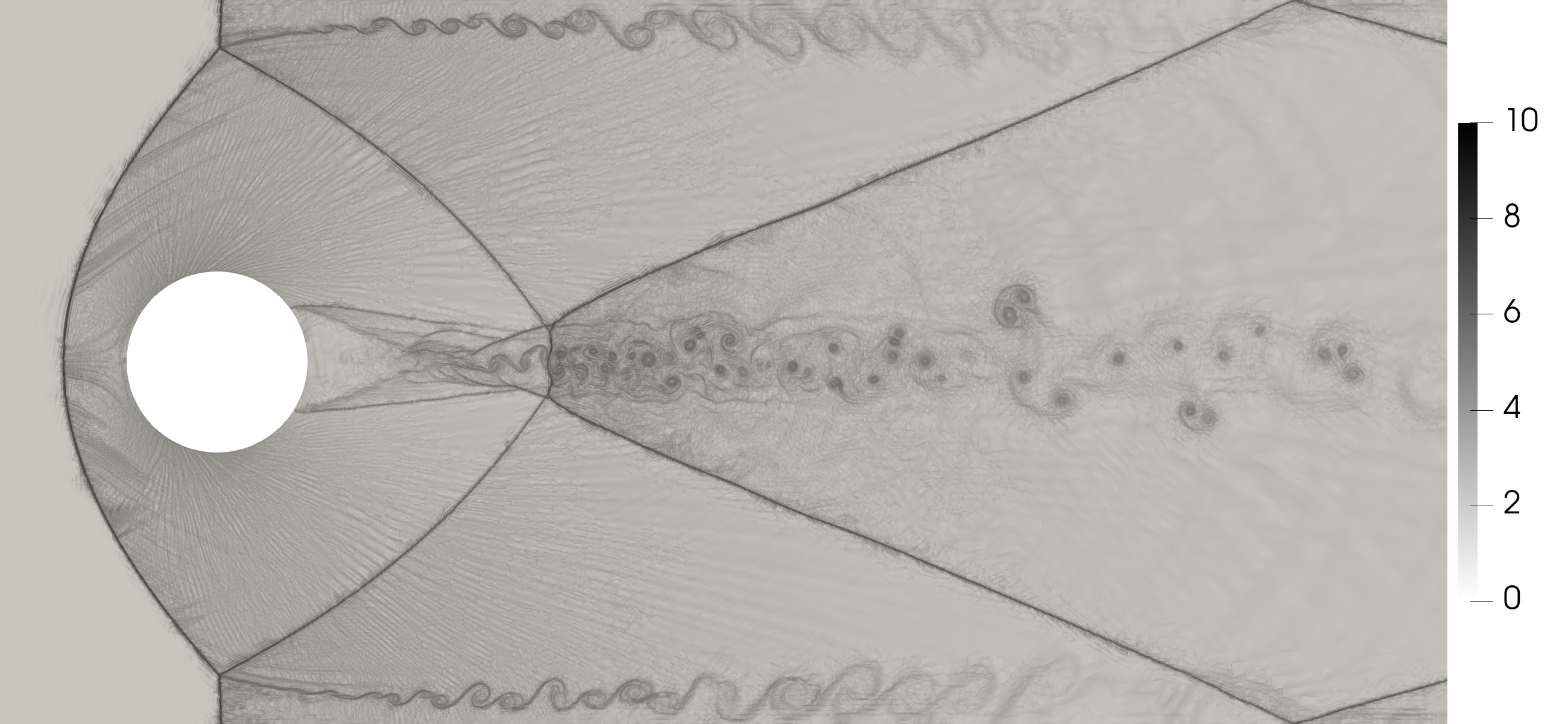}
    \caption{{Schlieren field at $t=2.4$}}
  \end{subfigure}
  \hfill
  \begin{subfigure}{0.476\textwidth}
    \centering
    \includegraphics[width=\linewidth]{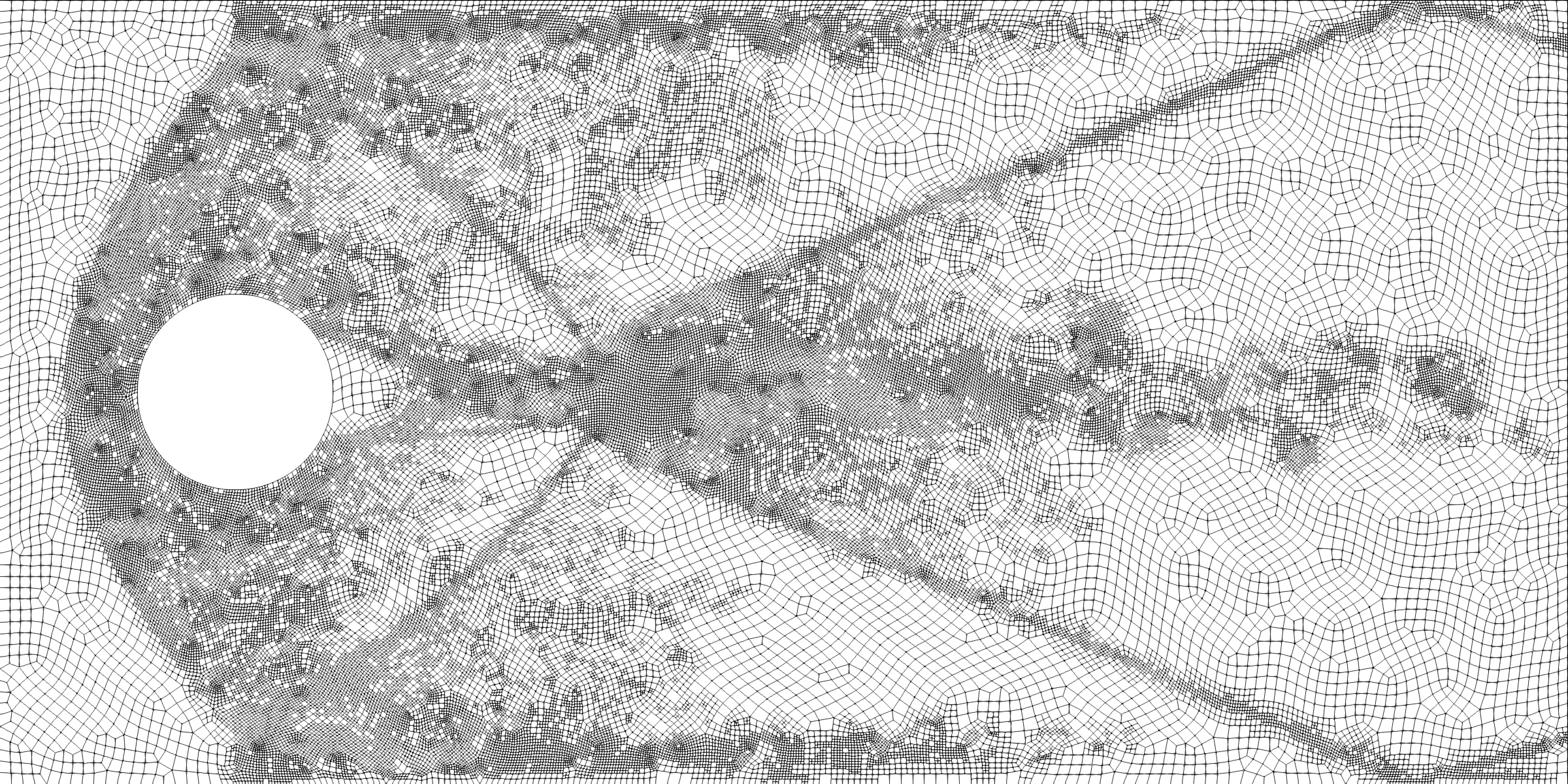}
    \caption{{adaptive mesh at $t=2.4$}}
  \end{subfigure}

  \begin{subfigure}{0.516\textwidth}
    \centering
    \includegraphics[width=\linewidth]{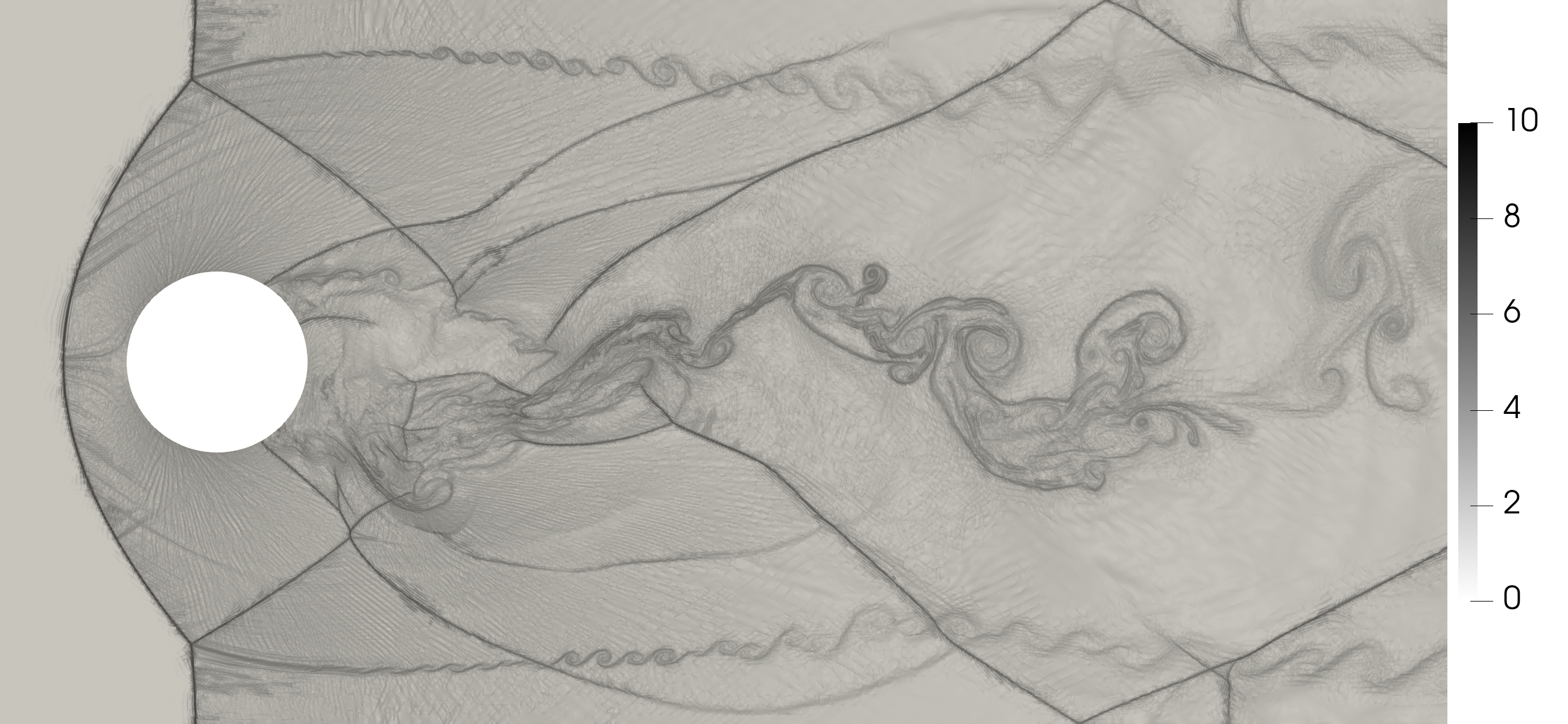}
    \caption{{Schlieren field at $t=4.2$}}
  \end{subfigure}
  \hfill
  \begin{subfigure}{0.476\textwidth}
    \centering
    \includegraphics[width=\linewidth]{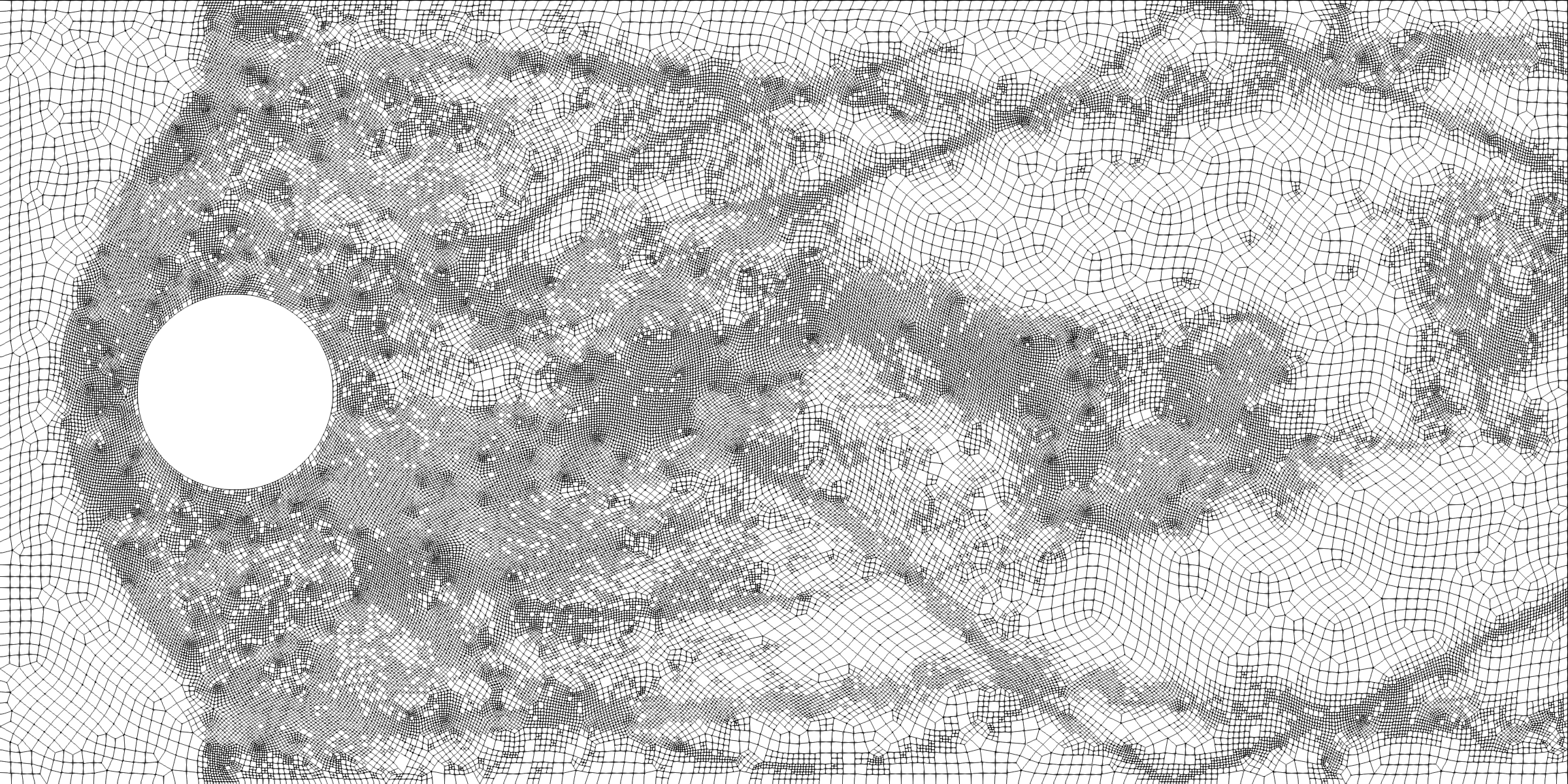}
    \caption{{adaptive mesh at $t=4.2$}}
  \end{subfigure}

  \caption{{\textit{Supersonic flow past a circular cylinder}:
  entropy-stable nonconforming flux with $N=3$. The left column shows
  Schlieren-like density fields, and the right column shows the
  corresponding adaptive meshes.}}
  \label{fig:cylinder_es}
\end{figure}

\FloatBarrier
\section{Conclusion}
\label{sec:con}
This work is a follow-up to our previous study on entropy-stable,
positivity-preserving, oscillation-eliminating DGSEM on conforming
curvilinear meshes \cite{YangFu26}. The main goal here is to extend
that framework to adaptive mesh refinement with nonconforming
interfaces, while retaining the core ingredients of curvilinear
high-order discretization, entropy stability, positivity preservation,
and oscillation elimination.

To this end, we first developed an entropy-stable nonconforming
interface treatment, referred to in the numerical section as the
entropy-stable flux. This flux preserves primary conservation, satisfies a
semi-discrete entropy inequality, and, together with the Zhang--Shu
limiter, fits into a provably positivity-preserving framework on
curvilinear AMR meshes. The analysis also reveals an important
limitation: {the nonlocal coarse-side projection, together with the
sign correction when negative interpolation coefficients occur, prevents this
entropy-stable construction from retaining design accuracy. In the smooth
nonconforming curvilinear test, its observed convergence remains between first
and second order and is close to first order for $N\ge2$.}

{For comparison, we also considered the classical mortar-based
flux. For this flux, we proved a positivity-preserving result under suitable
assumptions and showed numerically that it retains high-order accuracy on
smooth problems. Deterministic local interface tests nevertheless furnish
explicit positive entropy productions for the mortar coupling with the
guaranteed Guermond--Popov bound, while no
violation was found for the proposed entropy-stable coupling.}

Beyond the interface flux design, we developed a complete AMR
framework for curvilinear DGSEM. This includes shock-indicator-based
refinement and derefinement, positivity-preserving treatment of
nonconforming interfaces, selective oscillation elimination on AMR
meshes, and a conservative and positivity-preserving data-transfer
procedure between successive AMR meshes. These components together
yield a practical and robust adaptive algorithm for compressible Euler
flows on curvilinear meshes.
 
At present, provable entropy stability and positivity
preservation can be achieved together through the entropy-stable flux,
while high-order accuracy is retained by the mortar-based flux with a
provable positivity-preserving property. {The two interface
treatments studied here therefore do not simultaneously provide provable
entropy stability, provable positivity preservation, and full high-order
accuracy for polynomial degree $N\ge2$. Developing a single nonconforming
coupling with all three properties remains an important topic for future
work.}



\appendix

\section{Reference interpolation and projection matrices}
\label{app:interp_mats}
In this appendix, we provide explicit forms of the interpolation matrices \eqref{eq:P_C_to_Fk} and projection matrices \eqref{eq:P_Fk_to_C} used in the nonconforming interface construction.

\subsection*{Case $N=1$}

The LGL nodes on the coarse edge are $\{-1,1\}$. The fine-edge nodes are
\[
E_{F_1}: \{-1,0\}, \qquad
E_{F_2}: \{0,1\}.
\]
The interpolation matrices are
\[
P_{C\to F_1} =
\begin{pmatrix}
1 & 0 \\
\frac12 & \frac12
\end{pmatrix},
\qquad
P_{C\to F_2} =
\begin{pmatrix}
\frac12 & \frac12 \\
0 & 1
\end{pmatrix}.
\]
The corresponding projection matrices are
\[
P_{F_1\to C} =
\begin{pmatrix}
\frac12 & \frac14 \\
0 & \frac14
\end{pmatrix},
\qquad
P_{F_2\to C} =
\begin{pmatrix}
\frac14 & 0 \\
\frac14 & \frac12
\end{pmatrix}.
\]

\subsection*{Case $N=2$}

The LGL nodes on the coarse edge are $\{-1,0,1\}$. The fine-edge nodes are
\[
E_{F_1}: \left\{-1,-\tfrac12,0\right\}, \qquad
E_{F_2}: \left\{0,\tfrac12,1\right\}.
\]
The interpolation matrices are
\[
P_{C\to F_1} =
\begin{pmatrix}
1 & 0 & 0 \\
\frac{3}{8} & \frac{3}{4} & -\frac{1}{8} \\
0 & 1 & 0
\end{pmatrix},
\qquad
P_{C\to F_2} =
\begin{pmatrix}
0 & 1 & 0 \\
-\frac{1}{8} & \frac{3}{4} & \frac{3}{8} \\
0 & 0 & 1
\end{pmatrix}.
\]
The corresponding projection matrices are
\[
P_{F_1\to C} =
\begin{pmatrix}
\frac12 & \frac34 & 0 \\
0 & \frac38 & \frac18 \\
0 & -\frac14 & 0
\end{pmatrix},
\qquad
P_{F_2\to C} =
\begin{pmatrix}
0 & -\frac14 & 0 \\
\frac18 & \frac38 & 0 \\
0 & \frac34 & \frac12
\end{pmatrix}.
\]

\subsection*{Case $N=3$}

The LGL nodes on the coarse edge are
\[
\left\{-1,-\frac{1}{\sqrt{5}},\frac{1}{\sqrt{5}},1\right\}.
\]
The fine-edge nodes are obtained by affine restriction:
\[
E_{F_1}:\left\{-1,-\frac12-\frac{\sqrt5}{10},-\frac12+\frac{\sqrt5}{10},0\right\},
\qquad
E_{F_2}:\left\{0,\frac12-\frac{\sqrt5}{10},\frac12+\frac{\sqrt5}{10},1\right\}.
\]

The interpolation matrices are
\[
P_{C\to F_1}=
\begin{pmatrix}
1 & 0 & 0 & 0 \\[0.3em]
\frac18+\frac{\sqrt5}{10} &
\frac12+\frac{\sqrt5}{8} &
\frac38-\frac{\sqrt5}{4} &
\frac{\sqrt5}{40} \\[0.3em]
\frac18-\frac{\sqrt5}{10} &
\frac38+\frac{\sqrt5}{4} &
\frac12-\frac{\sqrt5}{8} &
-\frac{\sqrt5}{40} \\[0.3em]
-\frac18 & \frac58 & \frac58 & -\frac18
\end{pmatrix},
\]
\[
P_{C\to F_2}=
\begin{pmatrix}
-\frac18 & \frac58 & \frac58 & -\frac18 \\[0.3em]
-\frac{\sqrt5}{40} &
\frac12-\frac{\sqrt5}{8} &
\frac38+\frac{\sqrt5}{4} &
\frac18-\frac{\sqrt5}{10} \\[0.3em]
\frac{\sqrt5}{40} &
\frac38-\frac{\sqrt5}{4} &
\frac12+\frac{\sqrt5}{8} &
\frac18+\frac{\sqrt5}{10} \\[0.3em]
0 & 0 & 0 & 1
\end{pmatrix}.
\]

The corresponding projection matrices are
\[
P_{F_1\to C}=
\begin{pmatrix}
\frac12 &
\frac{5}{16}+\frac{\sqrt5}{4} &
\frac{5}{16}-\frac{\sqrt5}{4} &
-\frac{1}{16} \\[0.3em]
0 &
\frac14+\frac{\sqrt5}{16} &
\frac{3}{16}+\frac{\sqrt5}{8} &
\frac{1}{16} \\[0.3em]
0 &
\frac{3}{16}-\frac{\sqrt5}{8} &
\frac14-\frac{\sqrt5}{16} &
\frac{1}{16} \\[0.3em]
0 &
\frac{\sqrt5}{16} &
-\frac{\sqrt5}{16} &
-\frac{1}{16}
\end{pmatrix},
\]
\[
P_{F_2\to C}=
\begin{pmatrix}
-\frac{1}{16} &
-\frac{\sqrt5}{16} &
\frac{\sqrt5}{16} &
0 \\[0.3em]
\frac{1}{16} &
\frac14-\frac{\sqrt5}{16} &
\frac{3}{16}-\frac{\sqrt5}{8} &
0 \\[0.3em]
\frac{1}{16} &
\frac{3}{16}+\frac{\sqrt5}{8} &
\frac14+\frac{\sqrt5}{16} &
0 \\[0.3em]
-\frac{1}{16} &
\frac{5}{16}-\frac{\sqrt5}{4} &
\frac{5}{16}+\frac{\sqrt5}{4} &
\frac12
\end{pmatrix}.
\]

\section{Proof of Theorem \ref{thm:global_es_nc}}
\label{app:proof_global_es}

{
\begin{lemma}[Entropy stability of the modified LF flux]
\label{lem:modified_lf_entropy}
{Assume that all paired states belong to the admissible set
$\mathcal G$.} Assume that the fine-side and pairwise metric-scaled normals are
nondegenerate, i.e.,
$\|\bm n_i^{F_k}\|>0$ and
$\|\bm n_{ij}^{F_k,C}\|>0$ for all pairs under consideration.
Let
\[
p_{ij}^{F_k}:=(P_{C\to F_k})_{ij},
\qquad
s_{ij}^{F_k}:=\operatorname{sign}(p_{ij}^{F_k}),
\qquad
\widehat{\bm n}_{ij}^{F_k,C}
:=\frac{\bm n_{ij}^{F_k,C}}{\|\bm n_{ij}^{F_k,C}\|}.
\]
For every pair with $p_{ij}^{F_k}\ne0$, define the ordered Riemann data
\begin{equation}
\label{eq:ordered_states_mlf}
(\bm U_{ij}^{-,F_k,C},\bm U_{ij}^{+,F_k,C})
:=
\begin{cases}
(\bm U_i^{F_k},\bm U_j^C), & p_{ij}^{F_k}>0,\\
(\bm U_j^C,\bm U_i^{F_k}), & p_{ij}^{F_k}<0.
\end{cases}
\end{equation}
Let $\Lambda_{ij}^{\mathrm{GP},F_k,C}$ be the guaranteed upper bound of
Guermond and Popov~\cite{GuermondPopov16} for the maximum absolute wave
speed of the one-dimensional Euler Riemann problem associated with the
ordered states \eqref{eq:ordered_states_mlf} in the direction
$\widehat{\bm n}_{ij}^{F_k,C}$.  Choose
\begin{equation}
\label{eq:alpha_entropy_nc}
\alpha_i^{F_k}
\ge
\max_{\substack{0\le j\le N\\p_{ij}^{F_k}\ne0}}
\left\{
\frac{\|\bm n_{ij}^{F_k,C}\|}{\|\bm n_i^{F_k}\|}
\Lambda_{ij}^{\mathrm{GP},F_k,C}
\right\}.
\end{equation}
Then the modified LF flux \eqref{lxf} satisfies the weighted pairwise
Tadmor inequality \eqref{eq:esf_nc}.  Consequently, it is an
entropy-stable nonconforming interface flux.
\end{lemma}

\begin{proof}
{
Fix two admissible ordered states $\bm U^-$ and $\bm U^+$ and a unit
normal $\widehat{\bm n}$, and write
\[
\bm f_n(\bm U):=\bm F(\bm U)\cdot\widehat{\bm n},
\qquad
q_n(\bm U):=\bm q(\bm U)\cdot\widehat{\bm n},
\qquad
\psi_n(\bm U):={\boldsymbol\psi}(\bm U)\cdot\widehat{\bm n}.
\]
If $\lambda$ bounds the complete Riemann fan, Lemmas~2.1--2.2 of
Guermond and Popov~\cite{GuermondPopovInvariant16}, applied with
$\tau=(2\lambda)^{-1}$, give
\begin{align*}
\overline{\bm U}
&=\frac{\bm U^-+\bm U^+}{2}
-\frac{\bm f_n(\bm U^+)-\bm f_n(\bm U^-)}{2\lambda},\\
\eta(\overline{\bm U})
&\le\frac{\eta(\bm U^-)+\eta(\bm U^+)}{2}
-\frac{q_n(\bm U^+)-q_n(\bm U^-)}{2\lambda}.
\end{align*}
Combining these cited relations with the two supporting-plane inequalities
\[
\eta(\overline{\bm U})
\ge \eta(\bm U^\pm)
+\bm V^\pm\cdot(\overline{\bm U}-\bm U^\pm).
\]
{gives
\[
\bm V^-\!\cdot(\overline{\bm U}-\bm U^-)
+\bm V^+\!\cdot(\overline{\bm U}-\bm U^+)
\le -\frac{q_n(\bm U^+)-q_n(\bm U^-)}{\lambda}.
\]
Substituting the expression for $\overline{\bm U}$ and multiplying by
$2\lambda$ yields
\[
-(\bm V^++\bm V^-)\cdot
 \bigl(\bm f_n(\bm U^+)-\bm f_n(\bm U^-)\bigr)
-\lambda(\bm V^+-\bm V^-)\cdot(\bm U^+-\bm U^-)
+2\bigl(q_n(\bm U^+)-q_n(\bm U^-)\bigr)\le0.
\]
Using $\psi_n=\bm V\cdot\bm f_n-q_n$, the left-hand side is exactly twice
the following pairwise Tadmor residual:}
\[
(\bm V^+-\bm V^-)\cdot
\bm f_n^{\rm LLF}(\bm U^-,\bm U^+)
-\bigl(\psi_n(\bm U^+)-\psi_n(\bm U^-)\bigr)
\le0,
\]
where
\[
\bm f_n^{\rm LLF}(\bm U^-,\bm U^+)
=\frac{\bm f_n(\bm U^-)+\bm f_n(\bm U^+)}{2}
-\frac{\lambda}{2}(\bm U^+-\bm U^-).
\]
We now apply this consequence to \eqref{lxf}.  Set
\[
\lambda_{ij}^{F_k,C}
=\frac{\alpha_i^{F_k}\|\bm n_i^{F_k}\|}
       {\|\bm n_{ij}^{F_k,C}\|}.
\]
Condition \eqref{eq:alpha_entropy_nc} and the guaranteed estimate of
Guermond and Popov~\cite{GuermondPopov16} imply that
$\lambda_{ij}^{F_k,C}$ bounds the exact Riemann fan for every ordered pair
in \eqref{eq:ordered_states_mlf}.  If $p_{ij}^{F_k}>0$, flux
\eqref{lxf}, divided by $\|\bm n_{ij}^{F_k,C}\|$, is precisely the LLF
flux for the fine-to-coarse ordering; hence the bracket in
\eqref{eq:esf_nc} is nonpositive.  If $p_{ij}^{F_k}<0$, the sign
correction makes \eqref{lxf} the LLF flux for the reversed ordering.
The bracket in \eqref{eq:esf_nc} is then the negative of that ordered
Tadmor residual and is nonnegative; multiplication by
$p_{ij}^{F_k}<0$ again gives \eqref{eq:esf_nc}.  The case
$p_{ij}^{F_k}=0$ is immediate.  The guaranteed Euler wave-speed bound
used here is proved in \cite{GuermondPopov16} for
$1<\gamma\le5/3$, including $\gamma=1.4$ used in the numerical
experiments.
}
\end{proof}
}
\begin{proof}
{The combined conforming wave-speed choice \eqref{alpha}
contains the exact Riemann fan, so the standard LLF entropy argument gives
\eqref{eq:esf_conf}.  On a nonconforming interface,
\eqref{eq:alpha_pp_nc} implies the entropy bound
\eqref{eq:alpha_entropy_nc}; hence
Lemma~\ref{lem:modified_lf_entropy} gives \eqref{eq:esf_nc}.  Thus both
interface inequalities used below follow from the stated wave-speed
conditions rather than being additional assumptions.}

Summing the single-element conservation identity \eqref{eq_cc} and the single-element entropy balance \eqref{eq_ss} over all elements, and using periodic boundary conditions, the global evolution of the conserved variables and of the entropy reduces to the sum of net contributions across interior interfaces. On conforming interfaces, the standard pairwise cancellation argument applies, and conservation together with the entropy inequality follow directly from \eqref{eq:esf_conf}. It therefore suffices to analyze one representative nonconforming interface
\[
E_C = E_{F_1}\cup E_{F_2},
\]
with notation as in Section~\ref{subsec:nc_flux}.

We first prove global conservation. The net contribution of this nonconforming interface to the conservation balance is
\begin{equation}
\label{eq:net_cons_pf}
\mathrm{NET}_{\mathrm{cons}}
=
\sum_{k=1}^2 \sum_{i=0}^N w_i
\bigl(\bm F_i^{\,F_k,*}\cdot \bm n_i^{F_k}\bigr)
+
\sum_{j=0}^N w_j
\bigl(\bm F_j^{\,C,*}\cdot \bm n_j^C\bigr),
\end{equation}
where {$w_i$ and $w_j$} are the quadrature weights on the reference edges $E_{F_k}^{\mathrm{ref}}$ and $E_C^{\mathrm{ref}}$, respectively. Substituting \eqref{eq:fine_flux} and \eqref{eq:coarse_flux} into \eqref{eq:net_cons_pf}, we obtain
\begin{align*}
\mathrm{NET}_{\mathrm{cons}}
&=
\sum_{k=1}^2 \sum_{i=0}^N w_i
\sum_{j=0}^N (P_{C\to F_k})_{ij}
\,
\bm F^\star(\bm U_i^{F_k},\bm U_j^C)\cdot \bm n_{ij}^{F_k,C}
\\
&\quad
-
\sum_{k=1}^2 \sum_{j=0}^N 2w_j
\sum_{i=0}^N (P_{F_k\to C})_{ji}
\,
\bm F^\star(\bm U_i^{F_k},\bm U_j^C)\cdot \bm n_{ij}^{F_k,C}
\\
&=
\sum_{k=1}^2 \sum_{i=0}^N \sum_{j=0}^N
\Bigl(
w_i(P_{C\to F_k})_{ij}
-
2w_j(P_{F_k\to C})_{ji}
\Bigr)
\bm F^\star(\bm U_i^{F_k},\bm U_j^C)\cdot \bm n_{ij}^{F_k,C}.
\end{align*}
By the compatibility relation \eqref{eq:proj_compat},
\[
P_{C\to F_k}^T M_{F_k}=M_C P_{F_k\to C},
\]
we have, componentwise,
\begin{align}
\label{comp}    
w_i(P_{C\to F_k})_{ij}
=
2w_j(P_{F_k\to C})_{ji}.
\end{align}
Hence every term cancels and
\[
\mathrm{NET}_{\mathrm{cons}}=0.
\]
Since the same cancellation holds on every nonconforming interface, while conforming interfaces are conservative by the standard DGSEM argument, summing over all interfaces yields \eqref{eq:global_cons}.

{We next turn to entropy stability.  In accordance with the
single-element balance \eqref{eq_ss}, the contribution of an interface to
the entropy time derivative is the negative of the sum of its outward
boundary entropy terms.  Therefore, the entropy-rate contribution of the
same nonconforming interface is}
{
\begin{align}
\label{eq:net_entropy_pf}
\mathcal I_{\mathrm{ent}}
&=
-\sum_{k=1}^2 \sum_{i=0}^N w_i
\Bigl[
\bigl(\bm F_i^{\,F_k,*}\cdot \bm n_i^{F_k}\bigr)\cdot \bm V_i^{F_k}
-
{\boldsymbol\psi_i^{F_k}}\cdot \bm n_i^{F_k}
\Bigr]
\nonumber\\
&\quad
-
\sum_{j=0}^N w_j
\Bigl[
\bigl(\bm F_j^{\,C,*}\cdot \bm n_j^C\bigr)\cdot \bm V_j^C
-
{\boldsymbol\psi_j^C}\cdot \bm n_j^C
\Bigr].
\end{align}
}

We first rewrite the potential terms on the fine edges. Using \eqref{nrm1}--\eqref{nrm2},
\[
\sum_{j=0}^N (P_{C\to F_k})_{ij}\,\bm n_j^C = -2\bm n_i^{F_k},
\qquad
\sum_{j=0}^N (P_{C\to F_k})_{ij}=1,
\]
we obtain
\begin{align}
\label{eq:navg_fine}
\sum_{j=0}^N (P_{C\to F_k})_{ij}\,\bm n_{ij}^{F_k,C}
&=
\frac12
\sum_{j=0}^N (P_{C\to F_k})_{ij}
\bigl(\bm n_i^{F_k}-\frac12\bm n_j^C\bigr)
\nonumber\\
&=
\frac12 \bm n_i^{F_k}
-
\frac14 \sum_{j=0}^N (P_{C\to F_k})_{ij}\,\bm n_j^C
=
\bm n_i^{F_k}.
\end{align}
Therefore,
\begin{equation}
\label{eq:psi_fine_rewrite}
\sum_{i=0}^N w_i\,{\boldsymbol\psi_i^{F_k}}\cdot \bm n_i^{F_k}
=
\sum_{i=0}^N \sum_{j=0}^N
w_i(P_{C\to F_k})_{ij}\,
{\boldsymbol\psi_i^{F_k}}\cdot \bm n_{ij}^{F_k,C}.
\end{equation}

We next rewrite the coarse-edge potential term. {In addition to
\eqref{eq:proj_compat} and \eqref{nrm1}--\eqref{nrm2}, this step uses the
exactness of the $(N+1)$-point LGL rule through degree $2N-1$. Indeed,
$\phi_j\in\mathbb P_N$ and the compatible scaled normal satisfies
$\bm n^C\in\mathbb P_{N-1}$, so
$\phi_j\bm n^C\in\mathbb P_{2N-1}$ and its childwise LGL quadratures equal
its exact coarse-edge integral. The corresponding argument for the constant
uses only $\phi_j\in\mathbb P_N$.} Consequently,
\begin{equation}
\label{eq:proj_sum_one}
\sum_{k=1}^2 \sum_{i=0}^N (P_{F_k\to C})_{ji}=1,
\end{equation}
and
\begin{equation}
\label{eq:proj_sum_normal}
\sum_{k=1}^2 \sum_{i=0}^N (P_{F_k\to C})_{ji}\,\bm n_i^{F_k}
=
-\frac12\bm n_j^C.
\end{equation}
Consequently,
\begin{align}
\label{eq:navg_coarse_correct}
\sum_{k=1}^2 \sum_{i=0}^N (P_{F_k\to C})_{ji}\,\bm n_{ij}^{F_k,C}
&=
\frac12
\sum_{k=1}^2 \sum_{i=0}^N (P_{F_k\to C})_{ji}
\bigl(\bm n_i^{F_k}-\frac12\bm n_j^C\bigr)
\nonumber\\
&=
\frac12\bigl(-\frac12\bm n_j^C\bigr)
-
\frac14\bm n_j^C
=
-\frac12\bm n_j^C.
\end{align}
It follows that
\begin{align}
\label{eq:psi_coarse_rewrite}
\sum_{j=0}^N w_j\,{\boldsymbol\psi_j^C}\cdot \bm n_j^C
&=
-\sum_{k=1}^2 \sum_{j=0}^N \sum_{i=0}^N
2w_j (P_{F_k\to C})_{ji}\,
{\boldsymbol\psi_j^C}\cdot \bm n_{ij}^{F_k,C}
\nonumber\\
&=
-\sum_{k=1}^2 \sum_{i=0}^N \sum_{j=0}^N
w_i (P_{C\to F_k})_{ij}\,
{\boldsymbol\psi_j^C}\cdot \bm n_{ij}^{F_k,C},
\end{align}
where the second equality again uses \eqref{comp}.

{Substituting \eqref{eq:fine_flux}, \eqref{eq:coarse_flux},
\eqref{eq:psi_fine_rewrite}, and \eqref{eq:psi_coarse_rewrite} into
\eqref{eq:net_entropy_pf}, and using the compatibility relation
\eqref{comp}, we arrive at the standard left-to-right Tadmor residual}
{
\begin{align*}
\mathcal I_{\mathrm{ent}}
&=
\sum_{k=1}^2 \sum_{i=0}^N \sum_{j=0}^N
w_i(P_{C\to F_k})_{ij}
\Bigl[
\bigl(\bm F^\star(\bm U_i^{F_k},\bm U_j^C)\cdot \bm n_{ij}^{F_k,C}\bigr)
\cdot (\bm V_j^C-\bm V_i^{F_k})
\\
&\qquad\qquad\qquad\qquad
-
({\boldsymbol\psi_j^C-\boldsymbol\psi_i^{F_k}})\cdot \bm n_{ij}^{F_k,C}
\Bigr].
\end{align*}
}
{By the established inequality \eqref{eq:esf_nc}, each
\emph{weighted pairwise residual} in this sum is non-positive. Hence}
\[
{\mathcal I_{\mathrm{ent}}\le 0.}
\]
{Since the same argument applies to every nonconforming
interface, while conforming interfaces satisfy the same left-to-right
Tadmor inequality \eqref{eq:esf_conf}, summing over all interfaces yields
\eqref{eq:global_entropy}. This completes the proof.}
\end{proof}


\section{Proof of Theorem \ref{thm:cell_average_pp}}
\label{app:proof_pp}

\begin{proof}
The proof proceeds by rewriting the forward Euler update \eqref{eq:compact_FE} of the cell average as a nonnegative linear combination of admissible states. Since the set $\mathcal G$ is convex and positively homogeneous, such a representation immediately implies that the updated cell average is admissible.

To this end, we consider the update \eqref{eq:compact_FE} in the form
\begin{equation}
\label{eq:wt}
|\Omega_e|\,\overline{\bm U}^{e,\,n+1}
=
\sum_{i,j=0}^N w_i w_j \mathcal J_{ij}^e \bm U_{ij}^{e,n}
-
\Delta t
\sum_{E\subset\partial\Omega_e}\sum_{r=0}^N w_r
\bigl(\bm F_r^{E,*,n}\cdot \bm n_r^E\bigr),
\end{equation}
and show that the right-hand side can be written as a nonnegative linear combination of admissible states.

We therefore examine each edge contribution in \eqref{eq:wt} and isolate the part that subtracts the local nodal state. For notational simplicity, we suppress the time index $n$ in the following.

\medskip
\noindent
\textit{Conforming edges.}
Let $E\subset\partial\Omega_e$ be a conforming edge, and define the unit normal
\[
\bm n_{u,r}^E := \frac{\bm n_r^E}{\|\bm n_r^E\|}.
\]
Let $\bm U_r^E$ and $\bm U_r^{E,\mathrm{nbr}}$ denote the interior and exterior nodal states at node $r$ on $E$. Using the local Lax--Friedrichs flux \eqref{lax}, we obtain the pointwise decomposition
\begin{align}
\label{eq:conf_pp_decomp}
-\,\bm F_r^{E,*}\cdot \bm n_r^E
&=
\frac{\alpha_r^E\|\bm n_r^E\|}{2}
\left(
\bm U_r^{E,\mathrm{nbr}}
-
\frac{\bm F(\bm U_r^{E,\mathrm{nbr}})\cdot \bm n_{u,r}^E}{\alpha_r^E}
\right)
\nonumber\\
&\quad+
\frac{\alpha_r^E\|\bm n_r^E\|}{2}
\left(
\bm U_r^{E}
-
\frac{\bm F(\bm U_r^{E})\cdot \bm n_{u,r}^E}{\alpha_r^E}
\right)
-
\alpha_r^E\|\bm n_r^E\|\,\bm U_r^E.
\end{align}
By Lemma~\ref{lem:LF_admissible} and the choice of $\alpha_r^E$ in \eqref{alpha}, each state in parentheses belongs to $\mathcal G$. Hence, the first two terms form a positive combination of admissible states, while the only negative contribution is the final term.

Multiplying \eqref{eq:conf_pp_decomp} by $w_r$ and summing over all nodes $r$ on $E$ gives
\begin{align}
-\,\sum_{r=0}^N w_r\bigl(\bm F_r^{E,*}\cdot \bm n_r^E\bigr)
&=
\bm \Pi_E
-
\sum_{r=0}^N \beta_r^E \bm U_r^E,
\end{align}
where $\bm \Pi_E\in \mathcal{G}$ is a positive combination of admissible states and
\begin{align}
\label{betaE}
\beta_r^E := w_r \alpha_r^E \|\bm n_r^E\| \ge 0.
\end{align}
Thus, the contribution of a conforming edge is a positive combination of admissible states minus a nonnegative multiple of the interior nodal states.

\medskip
\noindent
\textit{Fine nonconforming edges.}
Let $E=E_{F_k}\subset\partial\Omega_e$ be a fine nonconforming edge. Let $\{\bm U_i^{F_k}\}_{i=0}^N$ denote the interior nodal states on $E$, and let $\{\bm U_j^C\}_{j=0}^N$ denote the nodal states on the neighboring coarse side. Define
\[
\bm n_{u,ij}^{F_k,C} := \frac{\bm n_{ij}^{F_k,C}}{\|\bm n_{ij}^{F_k,C}\|}.
\]
{For each nonzero interpolation coefficient, define
\[
s_{ij}^{F_k}:=\operatorname{sign}\!\left((P_{C\to F_k})_{ij}\right),
\qquad
\bm n_{s,ij}^{F_k,C}:=s_{ij}^{F_k}\bm n_{u,ij}^{F_k,C}.
\]
The signed unit normal \(\bm n_{s,ij}^{F_k,C}\) is used only to expose the
Lax--Friedrichs splitting associated with each signed interpolation
coefficient; it does not change the physical orientation used in the
definition of the numerical flux. Terms for which
\((P_{C\to F_k})_{ij}=0\) vanish identically and are omitted from the
decompositions below.}
Using \eqref{eq:fine_flux} together with the modified Lax--Friedrichs flux \eqref{lxf}, we obtain the pointwise decomposition
\begin{align}
\label{eq:fine_pp_decomp}
-\,\bm F_i^{\,F_k,*}\cdot \bm n_i^{F_k}
&=
\sum_{j=0}^N
\frac{\bigl|(P_{C\to F_k})_{ij}\bigr|\alpha_{i}^{F_k}\|\bm n_{i}^{F_k}\|}{2}
\left(
\bm U_j^C
-
\frac{\bm F(\bm U_j^C)\cdot {\bm n_{s,ij}^{F_k,C}}}{\alpha_{i}^{F_k}\|\bm n_{i}^{F_k}\|/\|{\bm n_{ij}^{F_k,C}}\|}
\right)
\nonumber\\
&\quad+
\sum_{j=0}^N
\frac{\bigl|(P_{C\to F_k})_{ij}\bigr|\alpha_{i}^{F_k}\|\bm n_{i}^{F_k}\|}{2}
\left(
\bm U_i^{F_k}
-
\frac{\bm F(\bm U_i^{F_k})\cdot {\bm n_{s,ij}^{F_k,C}}}{\alpha_{i}^{F_k}\|\bm n_{i}^{F_k}\|/\|{\bm n_{ij}^{F_k,C}}\|}
\right)
\nonumber\\
&\quad-
\sum_{j=0}^N
\bigl|(P_{C\to F_k})_{ij}\bigr|\alpha_{i}^{F_k}\|\bm n_{i}^{F_k}\|\,\bm U_i^{F_k}.
\end{align}

{Indeed,
\[
\bigl|(P_{C\to F_k})_{ij}\bigr|\bm n_{s,ij}^{F_k,C}
=
(P_{C\to F_k})_{ij}\bm n_{u,ij}^{F_k,C},
\]
so the central physical-flux terms in \eqref{eq:fine_pp_decomp} retain the
signed coefficient in \eqref{eq:fine_flux}, whereas the dissipative
coefficients are the nonnegative quantities
\(\bigl|(P_{C\to F_k})_{ij}\bigr|\alpha_i^{F_k}\). Moreover,
\[
\left|\bm u\cdot\bm n_{s,ij}^{F_k,C}\right|+c
=
\left|\bm u\cdot\bm n_{u,ij}^{F_k,C}\right|+c,
\]
so reversing the splitting direction does not alter the endpoint
wave-speed condition required by Lemma~\ref{lem:LF_admissible}.}

{Condition~\eqref{eq:alpha_pp_nc} and
Lemma~\ref{lem:LF_admissible} imply that every state in parentheses in
\eqref{eq:fine_pp_decomp} belongs to \(\mathcal G\).}

Multiplying \eqref{eq:fine_pp_decomp} by $w_i$ and summing over all nodes $i$ on $E$ gives
\begin{align}
-\,\sum_{i=0}^N w_i\bigl(\bm F_i^{\,F_k,*}\cdot \bm n_i^{F_k}\bigr)
&=
\bm \Pi_{E_{F_k}}
-
\sum_{i=0}^N \beta_i^{F_k}\bm U_i^{F_k},
\end{align}
where $\bm \Pi_{E_{F_k}}\in \mathcal{G}$ is a positive combination of admissible states and
\begin{align}
\label{betaF}
\beta_i^{F_k}
:=
w_i \sum_{j=0}^N
\bigl|(P_{C\to F_k})_{ij}\bigr|\alpha_{i}^{F_k}\|\bm n_{i}^{F_k}\|
\ge 0.    
\end{align}
{Here the absolute value
$\bigl|(P_{C\to F_k})_{ij}\bigr|$, together with the signed direction
$\bm n_{s,ij}^{F_k,C}$, is essential: the former makes every coefficient in
the convex decomposition nonnegative, while the latter preserves the signed
central physical-flux contribution of the original numerical flux.}

\medskip
\noindent
\textit{Coarse nonconforming edges.}
Let $E=E_C\subset\partial\Omega_e$ be a coarse nonconforming edge. Let $\{\bm U_j^C\}_{j=0}^N$ denote the interior nodal states on the coarse side of $E$, and let $\{\bm U_i^{F_k}\}_{i=0}^N$, $k=1,2$, denote the nodal states on the two fine segments of the same physical edge. Using \eqref{eq:coarse_flux} together with the modified Lax--Friedrichs flux \eqref{lxf}, we obtain the pointwise decomposition
\begin{align}
\label{eq:coarse_pp_decomp}
-\,\bm F_j^{\,C,*}\cdot \bm n_j^C
&=
\sum_{k=1}^2\sum_{i=0}^N
\bigl| (P_{F_k\to C})_{ji} \bigr|
\alpha_{i}^{F_k}\|\bm n_{i}^{F_k}\|
\left(
\bm U_i^{F_k}
+
\frac{\bm F(\bm U_i^{F_k})\cdot {\bm n_{s,ij}^{F_k,C}}}{\alpha_{i}^{F_k}\|\bm n_{i}^{F_k}\|/\|\bm n_{ij}^{F_k,C}\|
}
\right)
\nonumber\\
&\quad+
\sum_{k=1}^2\sum_{i=0}^N
\bigl| (P_{F_k\to C})_{ji} \bigr|
\alpha_{i}^{F_k}\|\bm n_{i}^{F_k}\|\left(
\bm U_j^{C}
+
\frac{\bm F(\bm U_j^{C})\cdot {\bm n_{s,ij}^{F_k,C}}}{\alpha_{i}^{F_k}\|\bm n_{i}^{F_k}\|/\|\bm n_{ij}^{F_k,C}\|}
\right)
\nonumber\\
&\quad-
2\sum_{k=1}^2\sum_{i=0}^N
\bigl| (P_{F_k\to C})_{ji} \bigr|
\alpha_{i}^{F_k}\|\bm n_{i}^{F_k}\|\,\bm U_j^C.
\end{align}
{The same signed direction is required in
\eqref{eq:coarse_pp_decomp}. By \eqref{comp},
\((P_{F_k\to C})_{ji}\) and \((P_{C\to F_k})_{ij}\) have the same sign;
hence
\[
\bigl|(P_{F_k\to C})_{ji}\bigr|\bm n_{s,ij}^{F_k,C}
=
(P_{F_k\to C})_{ji}\bm n_{u,ij}^{F_k,C}.
\]
Consequently, \eqref{eq:coarse_pp_decomp} is algebraically identical to the
coarse numerical flux \eqref{eq:coarse_flux}, including when an interpolation
coefficient is negative.}
{The same pairwise directional condition
\eqref{eq:alpha_pp_nc}, now used with the plus sign in
Lemma~\ref{lem:LF_admissible}, implies that every state in parentheses in
\eqref{eq:coarse_pp_decomp} belongs to \(\mathcal G\).}
Here we have used the compatibility relation \eqref{comp}
\[
 (P_{F_k\to C})_{ji}\,\mathrm{sign}\bigl((P_{C\to F_k})_{ij}\bigr)
=
\bigl|(P_{F_k\to C})_{ji}\bigr|,
\]
which ensures nonnegativity of all coefficients.

Multiplying \eqref{eq:coarse_pp_decomp} by $w_j$ and summing over all nodes $j$ on $E$ gives
\begin{align}
-\,\sum_{j=0}^N w_j\bigl(\bm F_j^{\,C,*}\cdot \bm n_j^C\bigr)
&=
\bm \Pi_{E_C}
-
\sum_{j=0}^N \beta_j^C\,\bm U_j^C,
\end{align}
where $\bm \Pi_{E_C}\in \mathcal{G}$ is a positive combination of admissible states and
\begin{align}
\label{betaC}
\beta_j^C
:=
2w_j \sum_{k=1}^2\sum_{i=0}^N
\bigl| (P_{F_k\to C})_{ji} \bigr|
\alpha_{i}^{F_k}\|\bm n_{i}^{F_k}\|
\ge 0.    
\end{align}

\medskip
Summing the edge contributions on the right-hand side of \eqref{eq:wt}, we obtain
\begin{equation}
\label{eq:FE_convex_combination}
\sum_{i,j=0}^N
w_i w_j \mathcal J_{ij}^e \bm U_{ij}^{e}
-
\Delta t\,\bm D_e
+
\Delta t\,\bm \Pi_e,
\end{equation}
where $\bm \Pi_e$ is a positive combination of admissible states, and $\bm D_e$ collects the negative contributions from the edge decompositions. 
More precisely, $\bm D_e$ can be written in nodal form as
\[
\bm D_e
=
\sum_{(m,n)\in\mathcal B_e}
\left(
\sum_{E\ni(m,n)} \beta_r^E
+\sum_{k:\,E_{F_k}\ni(m,n)} \beta_i^{F_k}
+\sum_{E_C\ni(m,n)} \beta_j^C
\right)\bm U_{mn}^{e},
\]
where each sum collects the contributions from all edges incident to the boundary node $(m,n)$.
Each of these terms is associated with a boundary nodal state of $\Omega_e$ and therefore already appears in the cell-average sum
\[
\sum_{i,j=0}^N w_i w_j \mathcal J_{ij}^e \bm U_{ij}^{e}.
\]
By the CFL condition \eqref{eq:CFL_pp}, the coefficient of each boundary nodal state remains nonnegative after subtracting the corresponding edge contributions. For interior nodes, no subtraction occurs, so the coefficient remains $w_i w_j \mathcal J_{ij}^e>0$.

It follows that \eqref{eq:FE_convex_combination} expresses $|\Omega_e|\,\overline{\bm U}^{e,\,n+1}$ as a nonnegative linear combination of admissible states. Since all nodal states $\bm U_{ij}^{e}$ belong to $\mathcal G$, and $\bm \Pi_e\in\mathcal G$, convexity and positive homogeneity of $\mathcal G$ imply that
\[
\overline{\bm U}^{e,\,n+1}\in\mathcal G.
\]
This completes the proof.
\end{proof}

\section{\texorpdfstring{{Proof of Theorem~\ref{thm:mortar_pp}}}{Proof of the mortar positivity theorem}}
\label{app:proof_mortar_pp}

{
\begin{proof}
Only the nonconforming-edge algebra differs from the proof of
Theorem~\ref{thm:cell_average_pp}. The treatment of conforming edges is
identical to \eqref{eq:conf_pp_decomp}--\eqref{betaE}, and the final assembly
into the forward Euler cell-average update follows the same argument as in
\ref{app:proof_pp}.

For brevity, define the mortar flux at a fine node by
\[
\bm f_i^{F_k,M}
:=
\bm F^*\!\left(\bm U_i^{F_k},\widetilde{\bm U}_i^{C,k}\right)
\cdot\bm n_i^{F_k}.
\]
On a fine element, the standard LF splitting gives
\begin{align}
-\bm f_i^{F_k,M}
={}&
\frac{a_i^{F_k,M}}{2}
\left(
\bm U_i^{F_k}
-
\frac{\bm F(\bm U_i^{F_k})\cdot\widehat{\bm n}_i^{F_k}}
     {\alpha_i^{F_k,M}}
\right)
\nonumber\\
&+
\frac{a_i^{F_k,M}}{2}
\left(
\widetilde{\bm U}_i^{C,k}
-
\frac{\bm F(\widetilde{\bm U}_i^{C,k})
      \cdot\widehat{\bm n}_i^{F_k}}
     {\alpha_i^{F_k,M}}
\right)
-a_i^{F_k,M}\bm U_i^{F_k}.
\label{eq:mortar_fine_pp_decomp}
\end{align}
By \eqref{eq:interp_admissible}, \eqref{eq:alpha_pp_mortar}, and
Lemma~\ref{lem:LF_admissible}, both states in parentheses belong to
\(\mathcal G\). After multiplication by \(w_i\), the only term subtracted
from the fine-element nodal contribution is therefore
\(\beta_i^{F_k,M}\bm U_i^{F_k}\), with
\(\beta_i^{F_k,M}\) defined in \eqref{eq:beta_mortar}.

The coarse side requires a different algebraic step. The compatibility
condition \eqref{eq:proj_compat}, the relation
\(M_{F_k}=\tfrac12M_C\), and the partition-of-unity identity
\(P_{C\to F_k}\bm 1=\bm 1\) imply
\begin{equation}
\sum_{j=0}^N w_j(P_{F_k\to C})_{ji}
=
\frac12w_i.
\label{eq:mortar_weight_identity}
\end{equation}
Consequently, summing \eqref{eq:coarse_flux_mortar} with the coarse-edge
quadrature weights yields the exact identity
\begin{equation}
-\sum_{j=0}^Nw_j\bm F_j^{C,*}\cdot\bm n_j^C
=
\sum_{k=1}^2\sum_{i=0}^Nw_i\bm f_i^{F_k,M}.
\label{eq:mortar_coarse_integrated_flux}
\end{equation}
For the right-hand side of \eqref{eq:mortar_coarse_integrated_flux}, the LF
splitting with the opposite sign is
\begin{align}
\bm f_i^{F_k,M}
={}&
\frac{a_i^{F_k,M}}{2}
\left(
\bm U_i^{F_k}
+
\frac{\bm F(\bm U_i^{F_k})\cdot\widehat{\bm n}_i^{F_k}}
     {\alpha_i^{F_k,M}}
\right)
\nonumber\\
&+
\frac{a_i^{F_k,M}}{2}
\left(
\widetilde{\bm U}_i^{C,k}
+
\frac{\bm F(\widetilde{\bm U}_i^{C,k})
      \cdot\widehat{\bm n}_i^{F_k}}
     {\alpha_i^{F_k,M}}
\right)
-a_i^{F_k,M}\widetilde{\bm U}_i^{C,k}.
\label{eq:mortar_coarse_pp_decomp}
\end{align}
Again, the first two terms are positive multiples of admissible states. For
the last term, using \eqref{eq:coarse_interp} and exchanging the finite sums
gives
\begin{equation}
\sum_{k=1}^2\sum_{i=0}^N
w_i a_i^{F_k,M}\widetilde{\bm U}_i^{C,k}
=
\sum_{j=0}^N\beta_j^{C,M}\bm U_j^C,
\label{eq:mortar_coarse_beta_identity}
\end{equation}
where \(\beta_j^{C,M}\) is defined in \eqref{eq:beta_mortar}. Unlike the
coefficients in the entropy-stable-flux proof, \(\beta_j^{C,M}\) need not be
nonnegative. This causes no difficulty: if \(\beta_j^{C,M}<0\), its
contribution increases rather than decreases the coefficient of
\(\bm U_j^C\) in the forward Euler update. In particular,
\[
w_mw_n\mathcal J_{mn}^e
-\Delta t\sum_{E_C\ni(m,n)}\beta_j^{C,M}
\geq
w_mw_n\mathcal J_{mn}^e
-\Delta t\sum_{E_C\ni(m,n)}(\beta_j^{C,M})_+.
\]
Combining this estimate with the conforming-edge decomposition and the
fine-side decomposition \eqref{eq:mortar_fine_pp_decomp}, condition
\eqref{eq:CFL_pp_mortar} makes every coefficient of an original nodal state
nonnegative. All remaining terms are positive multiples of the admissible LF
splitting states. The rest of the convexity argument is identical to the last
part of the proof of Theorem~\ref{thm:cell_average_pp}; hence
\(\overline{\bm U}^{e,n+1}\in\mathcal G\).
\end{proof}
}

\section{\texorpdfstring{{Reproducibility data for the local entropy counterexample}}{Reproducibility data for the local entropy counterexample}}
\label{app:local_entropy_repro}

{
For the deterministic lattice \eqref{eq:entropy_lattice}--\eqref{eq:entropy_lattice_state}
with $N=1$, the strongest mortar counterexample obtained with the
Guermond--Popov wave speed occurs at $(r,c)=(247,100)$. The six conservative
states, ordered as $(\rho,\rho u,\rho v,E)$, are
\begin{table}[H]
\centering
\resizebox{\textwidth}{!}{%
\begin{tabular}{c|rrrr}
\hline
state & $\rho$ & $\rho u$ & $\rho v$ & $E$\\
\hline
$C_0$ & $3.2706868214659521$ & $-8.7654406815287516$ & $1.1447403875130833\times10^{1}$ & $3.2810432727925104\times10^{1}$\\
$C_1$ & $4.1578578427560062$ & $-8.6483443129324939$ & $-1.6465117057313783\times10^{1}$ & $4.2849150031596224\times10^{1}$\\
$F_{1,0}$ & $2.4598124150548634\times10^{-1}$ & $-4.0340923606899765\times10^{-1}$ & $-8.7569321975953140\times10^{-1}$ & $3.3247101577396343$\\
$F_{1,1}$ & $2.7734307534502189\times10^{-1}$ & $-3.7718658246922976\times10^{-1}$ & $-9.2077901014547270\times10^{-1}$ & $3.3319385124359133$\\
$F_{2,0}$ & $2.9449301333262756\times10^{-1}$ & $-3.6517133653245826\times10^{-1}$ & $-9.5415736319771338\times10^{-1}$ & $3.3424789200045222$\\
$F_{2,1}$ & $2.9449301333262756\times10^{-1}$ & $-3.7695105706576321\times10^{-1}$ & $-9.7771680426432350\times10^{-1}$ & $3.3654975186315301$\\
\hline
\end{tabular}}
\caption{Conservative states for the reproducible $N=1$ local entropy
counterexample. All nodal states and interpolated mortar states are
admissible.}
\label{tab:local_entropy_counterexample_states}
\end{table}
Direct substitution into \eqref{eq:local_mortar_entropy_production} and
\eqref{eq:local_es_entropy_production} gives
\[
\mathcal D_{\rm M}=5.7967955763667237\times10^2,
\qquad
\mathcal D_{\rm ES}=-1.7734974470034146\times10^2.
\]
The childwise contributions are
\[
(\mathcal D_{\rm M}^{F_1},\mathcal D_{\rm M}^{F_2})
=(2.7162120068478981\times10^2,\ 3.0805835695188256\times10^2),
\]
and
\[
(\mathcal D_{\rm ES}^{F_1},\mathcal D_{\rm ES}^{F_2})
=(-1.0987458968013931\times10^2,\ -6.7475155020202152\times10^1).
\]
Thus the mortar entropy inequality is
violated by a large positive margin, whereas the proposed ES coupling remains
strictly dissipative for the same states and wave-speed estimate.
}

\section*{Acknowledgements}
The authors would like to thank Jesse Chan for helpful discussion on {entropy-stable numerical-flux construction} on curvilinear nonconforming interfaces, and Chen Liu for the setup on the {high Mach} astrophysical jet problem.

\section*{CRediT authorship contribution statement}
\textbf{Jielin Yang}: Writing – original draft, Writing – review \& editing, Software, Methodology, Conceptualization.
\textbf{Guosheng Fu}: Writing – original draft, Writing – review \& editing, Software, Methodology, Conceptualization.


\end{document}